\documentclass[leqno,onefignum,onetabnum,eqref]{siamltex}

\usepackage[margin=0.5in]{geometry}
\usepackage{epsfig,amsmath,amssymb,subfigure,version,graphicx,fancybox,mathrsfs,amscd}
\usepackage{cases,bm,multirow}
\usepackage{algorithm,algorithmic}
\usepackage{epstopdf}
\usepackage{url,xcolor}
\usepackage[colorlinks,urlcolor=blue]{hyperref}
\newtheorem{thm}{\bf Theorem}
\newtheorem{lem}{\bf Lemma}[section]
\newtheorem{rem}{ Remark}[section]
\graphicspath{{images/}}
\newtheorem{prop}{\bf Proposition}

\newtheorem{assump}{\bf Assumption}
\newtheorem{definit}{\bf Definition}

\title{
Blind Ptychographic  Phase Retrieval via Convergent Alternating Direction Method of Multipliers
}

\author{Huibin Chang\thanks{Corresponding author. School of Mathematical Sciences, Tianjin Normal University, Tianjin, 300387,  China, {\tt E-mail:changhuibin@gmail.com}. The author is currently a visiting researcher of the Computational Research Division at Lawrence Berkeley National Laboratory.}
\and{Pablo Enfedaque}\thanks{Computational Research Division, Lawrence Berkeley National Laboratory, Berkeley, CA 94720, USA,
    {\tt Email:penfedaque@lbl.gov}},
    \and {Stefano Marchesini}\thanks{Corresponding author. Computational Research Division, Lawrence Berkeley National Laboratory, Berkeley, CA 94720, USA,
    {\tt Email:smarchesini@lbl.gov}}    
    }

\begin{document}
\maketitle
\slugger{siims}{xxxx}{xx}{x}{x--x}

\begin{abstract}
Ptychography  has  risen  as  a  reference X-ray imaging  technique: it achieves resolutions of one billionth of a meter, macroscopic field of view, or the capability to retrieve chemical or magnetic contrast, among other features. A ptychographyic reconstruction is normally formulated as a blind phase retrieval problem, where  both the image (sample) and the probe (illumination) have to be recovered from phaseless measured data.
In this article we address a nonlinear least squares model for the blind ptychography problem with constraints on the image and the probe by maximum likelihood estimation of the Poisson noise model.
We formulate a variant model that incorporates the information of phaseless measurements of the probe to eliminate possible artifacts.
Next, we propose a generalized alternating direction method of multipliers designed for the proposed nonconvex models with convergence guarantee under mild conditions, where  their subproblems can be solved by fast element-wise operations. Numerically, the proposed algorithm outperforms state-of-the-art algorithms in both speed and image quality.
\end{abstract}

\begin{keywords}
 Blind Ptychography; Phase Retrieval;  Alternating Direction Method of Multipliers
\end{keywords}

\begin{AMS}
46N10,~49N30,~49N45,~65F22,~65N21
\end{AMS}

\pagestyle{myheadings}
\thispagestyle{plain}
\markboth{Blind ptychographic phase retrieval via ADMM }{H. Chang, P. Enfedaque and S. Marchesini}

\section{Introduction}\label{intro}
 Ptychographic phase retrieval (Ptycho-PR) \cite{nellist1995resolution,rodenburg2004phase,maiden2009improved}
  is an increasingly popular imaging technique used in scientific fields as diverse as condensed matter physics, cell biology or materials science, among others.
\begin{figure}[hb!]
\begin{center}
\includegraphics[width=.5\textwidth]{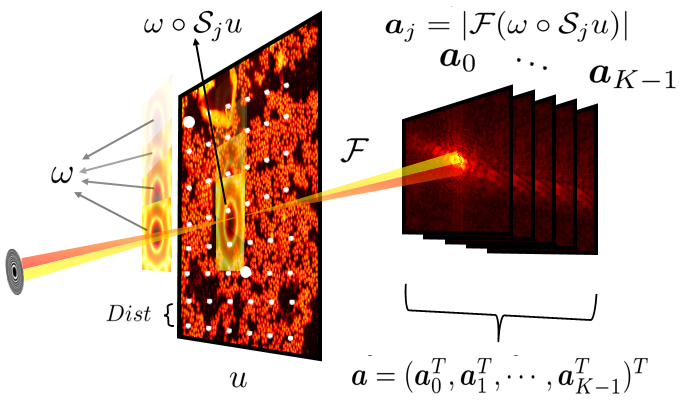}
\end{center}
\caption{Ptychographic phase retrieval (Far-field): A stack of  amplitudes $\bm a_j=|\mathcal F(\omega\circ\mathcal S_j u)|$ is collected, with
 $\omega$ being the localized coherent probe, and $u$ being the image of interest (sample).
 The white dots represent the scanning lattice points, with $D$ denoting the sliding distance between centers of two adjacent frames.}
\label{fig1}
\end{figure}
In a ptychography experiment (Figure \ref{fig1}),  a localized coherent X-ray probe (illumination) $\omega$  scans through an image (sample) $u$, while a 2D detector collects a sequence of phaseless intensities  in the far field.
Throughout the paper, we consider Ptycho-PR in a discrete setting as follows:
A 2D image is expressed as $u\in\mathbb C^n$ with $\sqrt{n}\times\sqrt{n}$ pixels, $\omega\in\mathbb C^{\bar m}$ is a localized 2D  probe   with  $\sqrt{\bar m}\times \sqrt{\bar m}$ pixels ($u$ and $\omega$ are both written as  a vector   by a lexicographical order), and ~$\bm a_j\in \mathbb R_+^{\bar m}~\forall 0\leq j\leq J-1$ corresponds to a stack of phaseless measurements, collected with $
\bm a_j=|\mathcal F(\omega\circ \mathcal S_j u)|.
$
Operation $|\cdot|$ represents the element-wise absolute value of a vector,  
 $\circ$ denotes the element-wise multiplication,  $\mathcal S_j\in \mathbb R^{\bar m\times n}$ is a binary  matrix that defines a small window  with the index $j$ and  size  $\bar m$ over the entire image $u$, and $\mathcal F$ denotes the normalized discrete Fourier transformation.
In practice, as the probe is almost never completely known, one has to solve a blind ptychographic phase retrieval (BP-PR) problem \cite{hesse2015proximal}:  
\begin{equation}\label{PtychoPR}
{\text{BP-PR:\qquad To find ~}\omega \in \mathbb C^{\bar m}\text{~and~} u\in \mathbb C^n,~~} s.t.~~|\mathcal A(\omega,u)|= \bm a,
\end{equation}
where  bilinear operators  $\mathcal A:\mathbb C^{\bar m}\times \mathbb C^{n}\rightarrow \mathbb C^{m}$ and
 $\mathcal A_j:\mathbb C^{\bar m}\times \mathbb C^{n}\rightarrow \mathbb C^{\bar m}~\forall 0\leq j\leq J-1$,  are denoted as follows:
  $
 \mathcal A(\omega,u):=(\mathcal A_0^T (\omega,u), \mathcal A_1^T(\omega,u),\cdots, \mathcal A_{J-1}^T(\omega,u))^T,$
 $\mathcal A_j(\omega,u):=\mathcal F(\omega\circ \mathcal S_j u),$
 and $\bm a:=(\bm a^T_0, \bm a^T_1, \cdots, \bm a^T_{J-1})^T\in \mathbb R^m_+.$

\subsection{State-of-the-art algorithms}\label{sec1-1}
When the probe  is exactly known, Ptycho-PR corresponds to a general phase retrieval (PR) problem.
Multiple efficient algorithms have been designed during the last years to solve both general phase retrieval~\cite{marchesini2007invited,Shechtman2014}, and non-blind (no probe retrieval) Ptycho-PR problems~\cite{rodenburg2004phase,guizar2008phase,wen2012,marchesini2013augmented,qian2014efficient,tian2014multiplexed,marchesini2015alternating,chen2018coded,xu2018accelerated}. { An early work by Chapman \cite{chapman1996phase} to solve the blind problem used the  Wigner-Distribution deconvolution method to retrieve the probe (zone-plate pupil).}  The more recent algorithms  that focus on the blind Ptycho-PR problem~\eqref{PtychoPR} are described below.

\subsubsection{Projection Algorithms}
In the optics community,   projection  algorithms such as Alternating Projections (AP) are very popular for non-blind PR problems~\cite{marchesini2007invited}.
Some projection algorithms have also been applied to  BP-PR problems, \emph{e.g.} Douglas-Rachford (DR) algorithm  by Thibault \emph{et al.} \cite{thibault2009probe}, extended ptychographic engine (ePIE) by Maiden and Rodenburg  \cite{maiden2009improved}, and relaxed averaged alternating reflections \cite{Luke2005} based projection algorithm by Marchesini \emph{et al.}  \cite{marchesini2016sharp}.
Projection algorithms for PR involve a projection onto a nonconvex modulus constraint set, and therefore, their convergence studies are very challenging.
Although some progress has been made for the general PR problem using projection algorithms~ \cite{hesse2013nonconvex,marchesini2015alternating,chen2016fourier},  the corresponding convergence for the blind problem is still unclear.
Arguably, the most popular projection algorithms for BP-PR are ePIE and DR, described in detail below.

\textbf{i) DR.} Given an exit wave $\Psi:=(\Psi^T_0,\Psi^T_1,\cdots,\Psi_{J-1}^T)^T\in \mathbb C^{m},$ with $\Psi_j:=\omega\circ\mathcal S_j u~\forall 0\leq j\leq J-1,$ one needs to find the exit wave $\Psi^\star$ that belongs to the intersection of two sets, \emph{i.e.}
$\Psi^\star\in\widehat{\mathscr X}_1\bigcap \widehat{\mathscr X}_2,$
with
\begin{equation}\label{eqX}
\widehat{\mathscr X}_1:=\{\Psi:=(\Psi^T_0,\Psi^T_1,\cdots,\Psi_{J-1}^T)^T\in \mathbb C^m:~|\mathcal F\Psi_j|=\bm a_j~\forall 0\leq j\leq J-1\},
\end{equation}
\hspace{12mm}
$
\widehat{\mathscr X}_2:=\{\Psi\in \mathbb C^m:~\exists \omega\in \mathbb C^{\bar m}, u\in \mathbb C^n, s.t. ~~\omega\circ \mathcal S_j u=\Psi_j~\forall 0\leq  j\leq J-1 \}.$
\vspace{2mm}

The AP algorithm that determines this intersection needs to calculate the projections onto $\widehat{\mathscr X}_1$ and $\widehat{\mathscr X}_2$.
Regarding the projection onto $\widehat{\mathscr X}_1$ as
$
\widehat {\mathcal P}_1(\Psi):=\arg\min_{\widehat\Psi\in \widehat {\mathscr X}_1}\|\widehat \Psi-\Psi\|^2,
$
with $\|\cdot\|$ denoting the $L^2$ norm in complex Euclidean space, one readily gets  a closed form solution
$\widehat {\mathcal P}_1(\Psi):=\big((\widehat {\mathcal P}_1^0(\Psi))^T,\cdots, (\widehat {\mathcal P}_1^{J-1}(\Psi))^T\big)^T,
$
with
$
\widehat {\mathcal P}_1^j(\Psi)=\mathcal F^{-1}(\bm a_j\circ \mathrm{sign}(\mathcal F\Psi_j))~0\leq j\leq J-1,
$
where $\forall~z\in \mathbb C^m,$  $(\mathrm{sign}(z))(t):=\mathrm{sign}(z(t))~\forall 0\leq t\leq m-1$.\footnote {$\mathrm{sign}(x)$ for a scalar $x\in \mathbb C$ is denoted as  $\mathrm{sign}(x)=\tfrac{x}{|x|}$ if $x\not=0;$ otherwise $\mathrm{sign}(0):=c$ with an arbitrary constant $c\in \mathbb C$ with unity length.}
For the projection onto $\widehat{\mathscr X}_2$, given $\Psi^k$ as the solution
in the $k^{\text{th}}$ iteration, one gets
 $\widehat {\mathcal P}_2(\Psi^{k}):=((\omega^{k+1}\circ \mathcal S_0 u^{k+1})^T, (\omega^{k+1}\circ \mathcal S_1 u^{k+1})^T,\cdots, (\omega^{k+1}\circ \mathcal S_{J-1} u^{k+1})^T)^T,$
where
\begin{equation}\label{modelLS}
(\omega^{k+1}, u^{k+1})=\arg\min\limits_{\omega, u} F(\omega,u,\Psi^k):=\tfrac12  {\textstyle \sum\nolimits_j}\|\Psi^k_j-\omega\circ \mathcal S_j u\|^2,
\end{equation}
with $\sum\nolimits_j$ being a simplified form of $\sum\nolimits_{j=0}^{J-1}.$
When solving \eqref{modelLS} inexactly by alternating minimization (after $T$ steps) as
\begin{equation}
\omega_{l+1}=\arg\min_{\omega} F(\omega, u_l,\Psi^k),  u_{l+1}=\arg\min_u F(\omega_{l+1},u,\Psi^k) ~\forall l=0,1,\cdots,T-1,
\label{eqAPmu}
\end{equation}
the DR algorithm for BP-PR can be formulated in two steps  \cite{thibault2009probe}, as follows:
\vspace{2mm}

\hspace{1mm}(1) Compute $\widehat\Psi^k$ by $\widehat\Psi^k_j=\omega^{k+1}\circ\mathcal S_j u^{k+1}$, where $(\omega^{k+1},u^{k+1})$ is solved  by \eqref{eqAPmu}.

\hspace{1mm}(2) Compute $\Psi^{k+1}$ by
\vspace{3mm}
$
\Psi^{k+1}=\Psi^k+\widehat {\mathcal P}_1(2\widehat\Psi^k-\Psi^k)-\widehat\Psi^k
$.

\textbf{ii) ePIE.}
This iterative algorithm can be expressed as an AP method for BP-PR as follows: To find $\Psi^\star_{n_k}$ belonging to the intersection as
\[
\Psi^\star_{n_k}\in\{|\mathcal F\Psi_{n_k}|=\bm a_{n_k}\}\cap\{\Psi_{n_k}:~\exists \omega\in\mathbb C^{\bar m}, u\in \mathbb C^n, ~~s.t. ~~\omega\circ\mathcal S_{n_k} u=\Psi_{n_k}\},
\]
with  a random frame index $n_k$.
 By first computing the projection $\psi^k_{n_k}$  by $\widehat {\mathcal P}_1^{n_k}(\psi^k_{n_k})$, and then updating $\omega^{k+1}$ and $u^{k+1}$ by the gradient descent algorithm (inexact projection),
the ePIE algorithm \cite{maiden2009improved} can be expressed by
updating $\omega^{k+1}$ and $u^{k+1}$  in parallel as
\[
\left\{
\begin{split}
&\omega^{k+1}=\omega^k-\tfrac{d_2}{\|\mathcal  S_{n_k} (u^{k})^*\|^2_{\infty}}\mathcal S_{n_k} (u^{k})^*\circ(\Psi^{k}_{n_k}-\widehat {\mathcal P}_1^{n_k} (\Psi^{k}_{n_k}))\\
&u^{k+1}=u^k-\tfrac{d_1}{{\|\omega^k\|^2_{\infty}}} \mathcal S_{n_k}^T\left({(\omega^k)^*}\circ(\Psi^{k}_{n_k}-\widehat {\mathcal P}_1^{n_k} (\Psi^{k}_{n_k}))\right),
\end{split}
\right.
\]
with frame index $n_k\in \{0,1,\cdots,J-1\}$   generated randomly, and positive parameters $d_1$ and $d_2$,
 where $\omega^k, u^k$ are the solutions in the $k^{\text{th}}$ iteration, $(\cdot)^*$  denotes the complex conjugate of a vector pointwise, and $\|\omega\|_\infty:=\max_t |\omega(t)|$.

\subsubsection{Proximal Alternating Linearized Minimization}\label{sec1-1-2}
The proximal alternating linearized minimization (PALM) method \cite{bolte2014proximal} was applied to BP-PR in \cite{hesse2015proximal}, which solved  a least squares problem with a nonconvex constraint set  as follows:
 \begin{equation}\label{eqLS}
 \min\limits_{\omega,u,\Psi} F(\omega,u,\Psi)+\mathbb I_{\widehat{\mathscr X}_1}(\Psi),
 \end{equation}
 with $F(\omega, u, \Psi)$ and $\widehat{\mathscr X}_1$  denoted in \eqref{modelLS} and \eqref{eqX}, respectively,
 and the indicator function $\mathbb I_{\widehat{\mathscr X}_1}$  denoted as:
 $ \mathbb I_{\widehat{\mathscr X}_1} (\Psi)\!=\!0,~\text{if~}\Psi\in \widehat{\mathscr X}_1, \text{otherwise~} \mathbb I_{\widehat{\mathscr X}_1} (\Psi)\!=\!+\infty.
 $
PALM [See \cite{hesse2015proximal}, Algorithm 2.1] can be expressed in two steps:
\vspace{1.5mm}

\hspace{1mm}(1) Compute $\omega^{k+1}, u^{k+1}$ sequentially by alternating  minimizing $F(\omega,u,\Psi^k)$ by the gradient descent scheme as ePIE.

\vspace{1.5mm}

\hspace{1mm}(2) Compute $\Psi^{k+1}$ by
$
\Psi^{k+1}=\widehat {\mathcal P}_1\big(\tfrac{1}{{1+\gamma}}({\widehat\Psi^{k+1}+\gamma \Psi^k})\big),
$
with $
\widehat\Psi^{k+1}_j:=\omega^{k+1}\circ S_j u^{k+1},
$
and the positive parameter $\gamma$.

\vspace{1.5mm}

%
%

Under  the assumption of  boundedness of iterative sequences, the convergence of PALM to  stationary points of \eqref{eqLS} was proved~\cite{hesse2015proximal}.  The parallel version of PALM  was also provided, presenting multiple similarities with ePIE. The main difference between them  is that for ePIE, only the gradient of $F$ \emph{w.r.t.} a randomly selected single frame is adopted to update $\omega$ and $u$ per outer loop, while for PALM, each block of $\omega$ and $u$ can be updated in parallel by employing the gradient \emph{w.r.t.} all adjacent frames. Therefore, the parallel version of PALM is numerically more robust than ePIE.

\subsection{Motivations and Contributions}

Existing algorithms for BP-PR still have several limitations.
The convergence speed of PALM \cite{hesse2015proximal} is not satisfactory (See the convergence  curves in cyan in the first row of Figure \ref{fig4}).  AP algorithms for BP-PR as DR \cite{thibault2009probe} and ePIE \cite{maiden2009improved} tend to be  unstable: visible drifts of the probe  happen during iterations of DR (See Figure \ref{fig5} (a)),  and the iterative sequence of ePIE  could diverge if the measurements are few or noisy. In addition, the existing algorithms reviewed in section \ref{sec1-1} for  BP-PR were derived based on a least squares model that ignored  the impact of experimental noise, \emph{e.g.} Poisson noise.

The alternating direction method of multipliers (ADMM) \cite{Esser2010,Wu&Tai2010,boyd2011distributed}
 was adopted to solve non-blind PR problems in \cite{wen2012,chang2016Total}, where it demonstrated fast convergence and robustness. However, it has  never been applied to BP-PR problems \eqref{PtychoPR}. On the theoretical side, only the subsequence convergence  of ADMM  was proved for non-blind PR problems in \cite{wen2012}, which relied on a relatively strong assumption of the convergence for the successive error of the multiplier.


In this paper we propose a more efficient method that employs {the generalized form of ADMM \cite{eckstein1994some,zhang2011unified,fazel2013hankel,deng2016global}, a typical extension of standard ADMM\footnote{Setting the preconditioning matrices to zeros in the generalized ADMM yields  standard ADMM.}}.
The main contributions of this paper are summarized below:
\begin{itemize}
\item
We propose a nonlinear least squares model  for  BP-PR  with
constraints on the image and probe, driven by the maximum likelihood estimation  of the Poisson noise.
We also formulate a  variant optimization model  by incorporating the Fourier magnitudes  of the probe to  attenuate artifacts caused by  the periodical lattices based scanning. 
\item
We design fast generalized ADMMs to solve the proposed models, whose subproblems can be efficiently solved with comparatively few elementwise operations.  With  suitable stepsizes, we prove  their convergence  to stationary points for the proposed models under  mild conditions.
\item
We conduct numerous experiments  to show the  convergence and efficiency of the proposed algorithms. Compared with the state of the art, the proposed algorithms present a significantly better convergence speed: their R-factor decreases faster, while producing higher quality images.  When using additional information from the probe, the proposed algorithms  also successfully   attenuate artifacts from the periodical lattice based scanning. We report more than $10\times$ acceleration by using a non-optimized GPU implementation in \textsc{Matlab}, which demonstrates how the method can be trivially parallelized, showing great potential for large-scale problems.

\end{itemize}
The reminder of the paper is organized as follows: Section \ref{sec2} gives the  nonlinear least squares models for BP-PR, as well as  the Lipschitz property. 
Fast generalized ADMMs are given in section \ref{sec3} with convergence guarantees  shown in section \ref{sec4}. Numerous experiments are conducted to verify the effectiveness of proposed algorithms in section \ref{sec5}. Section~\ref{sec6} summarizes this work.\!\!
\section{Proposed  Models}\label{sec2}

Instead of directly solving the quadratic multidimensional systems in \eqref{PtychoPR}, following \cite{qian2014efficient} for non-blind Ptycho-PR,
a nonlinear least squares model for BP-PR  can be given as
 \[
 \min\nolimits_{\omega\in \mathbb C^{\bar m},u\in\mathbb C^n}\tfrac12\||\mathcal A(\omega,u)|-\bm a\|^2.
 \]
In order to deal with data contaminated by different types of noise,  based on the maximum likelihood estimation (MLE),   more general mappings \cite{chang2018variational}  to measure the distance between the recovered intensity $g\in\mathbb R_+^m$  and the collected noisy intensity $f\in \mathbb R_+^m$, have been extensively studied for the PR problem:
\begin{equation}\label{DF}
\left\{
\begin{split}
& \tfrac{1}{2}\| \sqrt{g}\!-\!\sqrt{f} \|^2, \quad\qquad\quad\mbox{~Amplitude  Gaussian Metric (AGM)~ \cite{qian2014efficient,wen2012}}\\
&\tfrac12\langle g\!-\!f\circ \log(g), \bm 1_{m}\rangle, \phantom{..}\quad\mbox{Intensity  Poisson Metric (IPM) \cite{thibault2012maximum,chang2016Total}}\\
& \tfrac{1}{2}\big\|g\!-\!f\big\|^2, \phantom{.}\qquad\quad\qquad\enspace\mbox{~Intensity based Gaussian Metric (IGM)~ \cite{qian2014efficient}}\\
&\!\!\tfrac{1}{2}\|\tfrac{g}{\sqrt{f+\varepsilon \mathbf 1_{m}}}\!-\!\tfrac{f}{\sqrt{f+\varepsilon \mathbf 1_{m}}}\|^2, \phantom{-}\mbox{Weighted Intensity  Gaussian Metric (wIGM)~  \cite{qian2014efficient},}
\end{split}
\right.
\end{equation}
where
$\mathbf 1_{m}\in\mathbb R^{m}$ represents a vector  whose elements are all  ones,
 $\sqrt{\cdot}$ and $\tfrac{~\cdot~}{~\cdot~}$ denote  element-wise square root and division of a vector, respectively, $\varepsilon>0$ is the penalization  factor,  and $\langle\cdot, \cdot\rangle$ denotes the $L^2$ inner product in Euclidean space.
IPM is derived from the MLE for Poisson noise, and AGM can be interpreted as a metric based on the variance-stabilizing transform 
 for Poisson noise \cite{marchesini2016sharp}.

 {In this paper, we focus on  Poisson noisy measurements, which are typical  during photon-counting procedures in real experiments.} However, the resulted objective functionals for BP-PR directly based on AGM and IPM in \eqref{DF} are not  Lipschitz differentiable, such that it is difficult to design  first order splitting algorithms with  convergence guarantee. Therefore,  we consider the following two modified metrics with the penalization parameter $\varepsilon>0$ 
 as follows:
\begin{equation}\label{DF-P}
\!\mathcal B(g,f)\!=\!
\left\{
\begin{split}
& \tfrac{1}{2}\| \sqrt{g+\varepsilon \mathbf 1_{m}}\!-\!\sqrt{f+\varepsilon \mathbf 1_{m}} \|^2, \qquad\qquad\qquad\phantom{,.}\mbox{~Penalized-AGM (pAGM)}\\
&\tfrac12\langle g+\varepsilon \mathbf 1_{m}\!-\!(f+\varepsilon \mathbf 1_{m})\circ \log(g+\varepsilon \mathbf 1_{m}), \bm 1_{m}\rangle, \phantom{..}\mbox{Penalized-IPM (pIPM) }\\
\end{split}
\right.
\end{equation}
with  $\mathcal B(\cdot, \cdot): \mathbb R^{m}_+\times \mathbb R^{m}_+\rightarrow \mathbb R_+$. 
Consequently, a nonlinear  optimization model  for BP-PR  is given as:
\begin{equation}\label{eqModel}
\text{Model I:\qquad} \min\nolimits_{\omega\in \mathbb C^{\bar m},u\in\mathbb C^n}\mathcal G(\mathcal A(\omega,u)){+\mathbb I_{\mathscr X_1}(\omega)+\mathbb I_{\mathscr X_2}(u),}
\end{equation}
with $\mathcal G(z):=\mathcal B(|z|^2,f)$, where
 the amplitude constraints of the probe and  image as  \cite{hesse2015proximal,marchesini2016sharp} are incorporated, {where $\mathscr X_1:=\{\omega\in\mathbb C^{\bar m}:~\|\omega\|_\infty\leq C_{\omega}\}$ and $\mathscr X_2:=\{u\in \mathbb C^n:~\|u\|_\infty\leq C_{u}\}$ with two positive constants $C_\omega, C_u$.}

If the collected data is noise free, the minimizer to \eqref{eqModel} exists if the solution set $\{(\omega, u)\in \mathscr X_1\times \mathscr X_2: |\mathcal A(\omega, u)|=\bm a\}\not=\emptyset$. For the data possibly contaminated by noise, the following proposition shows the existence of the minimizer to Model I.
\begin{prop}\label{thm0}
Model I  admits at least one minimizer. 
\end{prop}

It is standard to show the existence of a minimizer to Model I, and we omit the details.

The generalized derivative for complex-valued variables are adopted as in \cite{hesse2015proximal,chang2016Total} by separating the real and imaginary parts of the variables and operators defined in the complex space.
Below we denote the  Lipschitz property of the function $\mathcal G$ defined on $\mathbb C^m$.
\begin{definit}[Lipschitz Property]
\label{def1}
The function $\Phi:\mathbb C^m\rightarrow\mathbb R$  has the Lipschitz property in $\mathbb C^m$ with a    Lipschitz constant $L$ if
\begin{itemize}
\item[1)] It is Lipschitz differentiable (its gradient is Lipschitz continuous), \emph{i.e.}
\begin{equation}\label{eqLC}
\|\nabla\Phi(v_1)-\nabla\Phi(v_2)\|\leq L \|v_1-v_2\|~\forall v_1, v_2\in \mathbb C^m;
\end{equation}
\item[2)] It has the descent property as
\begin{equation}\label{eqLipDescent}
\Phi(v_2)-\Phi(v_1)-\Re(\langle\nabla \Phi( v_1),  v_2-v_1\rangle)\leq \tfrac{L}{2}\|v_2-v_1\|^2~\forall v_1, v_2\in  \mathbb C^m.
\end{equation}
\end{itemize}
\end{definit}
Following \cite{attouch2010proximal}, letting  $\Phi$ be smooth in $C^1$, the second relation \eqref{eqLipDescent}  holds if  it  is Lipschitz differentiable as \eqref{eqLC}.
{It is well known that  the Lipschitz property of the objective functional is of high importance  to prove the convergence of first-order splitting algorithms for a nonconvex minimization problem \cite{xu2013block,bolte2014proximal,wang2015global,hong2016convergence}.} We  show the Lipschitz property of $\mathcal G$ in the following lemma.
\begin{lem}
\label{assump1}
 The function $\mathcal G(\cdot)$ has the Lipschitz property.
\end{lem}
\begin{proof}
The function $\mathcal G(z)$ with each metric in \eqref{DF-P} is $C^\infty$ smooth, and we just need to show the Lipschitz continuity of its gradient.
  The corresponding first order derivatives are given below:
\begin{equation}
\label{eqGrad-P}
\nabla \mathcal G(z)=
\left\{
\begin{split}
&z-\tfrac{\sqrt{f+\varepsilon\bm 1_m}}{\sqrt{|z|^2+\varepsilon \bm 1_m}}\circ z~\text{for pAGM};\\
&z-\tfrac{f+\varepsilon\bm 1_m}{|z|^2+\varepsilon \bm 1_m} \circ z~\text{for pIPM}.
\end{split}
\right.
\end{equation}
A basic estimate is given below:
\begin{equation}\label{eq3}
\|v_1\circ v_2\|\leq \|v_1\|_\infty \|v_2\|,~\forall~v_1, v_2\in \mathbb C^{m}.
\end{equation}
In order to prove the first property in \eqref{eqLC}, $\forall~v_1, v_2\in \mathbb C^m$, by \eqref{eqGrad-P}, we have
\[
\begin{split}
&\|\nabla\mathcal G(v_2)-\nabla\mathcal G(v_1)\|\leq \|v_2-v_1\|+\big\|\tfrac{\sqrt{f+\varepsilon\bm 1_m}}{\sqrt{|v_2|^2+\varepsilon \bm 1_m}}\circ v_2-\tfrac{\sqrt{f+\varepsilon\bm 1_m}}{\sqrt{|v_2|^2+\varepsilon \bm 1_m}}\circ v_1\big\|\\
&\hskip 5cm+\big\|\tfrac{\sqrt{f+\varepsilon\bm 1_m}}{\sqrt{|v_2|^2+\varepsilon \bm 1_m}}\circ v_1-\tfrac{\sqrt{f+\varepsilon\bm 1_m}}{\sqrt{|v_1|^2+\varepsilon \bm 1_m}}\circ v_1\big\|\\
&\stackrel{\eqref{eq3}}{\leq} (1+\big\|\tfrac{\sqrt{f+\varepsilon\bm 1_m}}{\sqrt{|v_2|^2+\varepsilon \bm 1_m}}\big\|_\infty)\|v_2-v_1\|
\\
&\hskip 1cm+
\big\|\sqrt{f+\varepsilon \bm 1_m}\circ \tfrac{v_1\circ(|v_1|+|v_2|)}{\sqrt{|v_1|^2+\varepsilon}\sqrt{|v_2|^2+\varepsilon}(\sqrt{|v_1|^2+\varepsilon}+\sqrt{|v_2|^2+\varepsilon})}\big\|_\infty\big\||v_1|-|v_2|\big\|\\
&\leq (1+\tfrac{2}{\sqrt{\varepsilon}}\|{\sqrt{f+\varepsilon\bm 1_m}}\big\|_\infty)\|v_2-v_1\|.
\end{split}
\]
Similarly, in the case of pIPM, we have
\[
\|\nabla\mathcal G(v_2)-\nabla\mathcal G(v_1)\|\leq (1+\tfrac{2}{{\varepsilon}}\|{{f+\varepsilon\bm 1_m}}\big\|_\infty)\|v_2-v_1\|,
\]
that completes the Lemma.
%
%
%
\end{proof}

\vskip .1in
For BP-PR  \eqref{PtychoPR}, or Model I in the noise free case (without constraints), trivial solutions always exist, \emph{i.e.} scaling, reflection or translation.  Moreover, there exists  nontrivial solutions, which possibly depends on the scanning geometry.
Especially, with periodical-lattice based scanning, there exists evident  structural artifacts in the recovered images \cite{thibault2009probe}. Specifically, assuming that  $(\omega,u)$ is a solution to \eqref{PtychoPR} for the noiseless case,
 the nontrivial solution
 $( \hat\omega, \hat u)$
 can be generated as 
 $( \hat\omega, \hat u):=( p_1\circ \omega, p_2\circ u)$ with two periodical   functions $p_1,p_2$ (not constants) satisfying
 $\mathcal S_{j_1} p_2=\mathcal S_{j_2} p_2\equiv \tfrac{\bm 1_{\bar m}}{p_1}~\forall 0\leq j_1, j_2\leq J-1,$
such that
$p_1\circ\mathcal S_j p_2\equiv \bm 1_{\bar m}.$
Readily, it can be seen how $( \hat \omega, \hat u)$ is a solution to \eqref{PtychoPR}, since
\[
\hat \omega\circ \mathcal S_j \hat u= (p_1 \circ\omega)\circ \mathcal S_j  (p_2\circ u)=(p_1\circ \mathcal S_j p_2)\circ (\omega\circ\mathcal S_j u)=\omega\circ\mathcal S_j u.
\]

Due to the existence of nontrivial solutions to \eqref{PtychoPR}, the existing algorithms may get trapped into finding a solution with periodical structures (Figure \ref{fig3} (b, k) and Figure \ref{fig66} (b, k)), which  completely breaks features of the image of interest.
In order to deal with  nontrivial ambiguities,   non-periodical lattices based scanning can be considered  experimentally to remove the periodicity of the scanning geometry, \emph{e.g.} adding a small amount of
random offsets to a set of square-lattice (or raster grid) \cite{maiden2009improved}, using a circular scan geometry \cite{dierolf2010ptychographic} or Fermat spiral scan geometry \cite{huang2014optimization}.
However, in practice, square lattice  or other periodical lattices based  scanning are very popular for being straightforward and easy to implement, and we have to explore  more a-priori information of the unknown probe to  attenuate artifacts with the periodical scanning geometry.

In real experiments, the specimen can be simply removed to collect the diffraction pattern of the unknown probe in far-field. Hence, in the following optimization problem, we
leverage the additional data $\bm c\in\mathbb R^{\bar m}_+$ to eliminate structural artifacts and  further increase the reconstruction accuracy:
\begin{equation}\label{eqModelpriorII}
\text{Model II:\qquad} \min\limits_{\omega\in \mathbb C^{\bar m},u\in\mathbb C^n}
 \mathcal G(\mathcal A(\omega,u)){+\mathbb I_{\mathscr X_1}(\omega)+\mathbb I_{\mathscr X_2}(u)}+\tau \widehat{\mathcal G}(\mathcal F\omega),
\end{equation}
where $\tau$ is a positive parameter,
{the additional measurement $\bm c$ is the diffraction pattern (absolute value of Fourier transform of the probe) as $\bm c:=|\mathcal F u|$, and
$
\widehat{\mathcal G}(z):=\mathcal B(|z|^2,\bm c^2),
$
with $\bm c^2$ denoting the pointwise square of the vector $c$.}
For simplicity, assume that $\mathcal G$ and $\hat{\mathcal G}$ adopt the same metric. We remark that such information can be used as a constraint condition as in \cite{marchesini2016sharp}.
Similarly to Proposition \ref{thm0}, the existence of a minimizer to Model II can be readily derived.
%
\vskip .1in
\section{Numerical Algorithms}\label{sec3}
The proposed models are non-convex and non-smooth, and they can be solved by a
{  block coordinate descent method with proximal linear version \cite{xu2013block}, or  PALM \cite{bolte2014proximal}, equivalently}.
We conduct an experiment using PALM \cite{bolte2014proximal} for Model I. Equivalent results can be obtained using the block coordinate descent method, BCD  \cite{xu2013block}, and the detailed iterative scheme can be found in Appendix \ref{apdx-0}.
The recovery results using PALM (BCD) are presented in Figure \ref{fig1-R}.
By observing the recovery images in Figure \ref{fig1-R} (b, c), obvious artifacts including the periodical artifacts and ringing artifacts of the edges can be observed.
Moreover, one can readily see that PALM (BCD) for Model I  converges very slow inferred from Figure \ref{fig1-R} (d),  
{since in order to make these iterative algorithms be stable,  very small stepsizes are selected manually, which results in  slow convergence speed.}
\begin{figure}[h!]
\begin{center}
\subfigure[Truth]{\includegraphics[width=.3\textwidth]{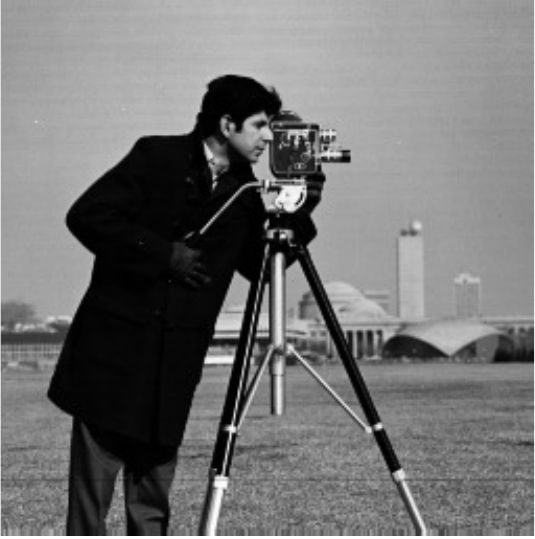}}
\subfigure[pAGM]{\includegraphics[width=.3\textwidth]{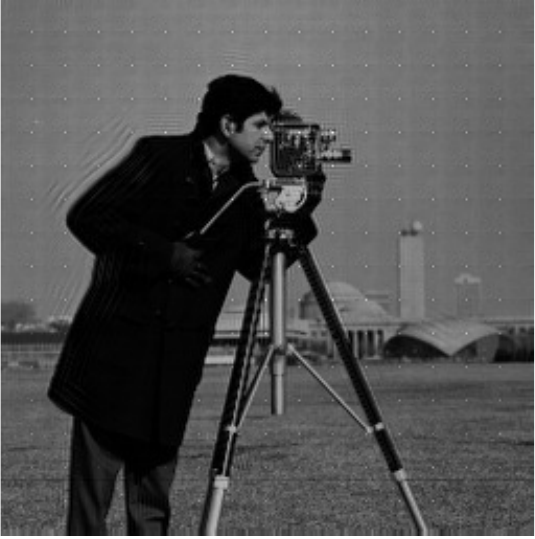}}
\subfigure[pIPM]{\includegraphics[width=.3\textwidth]{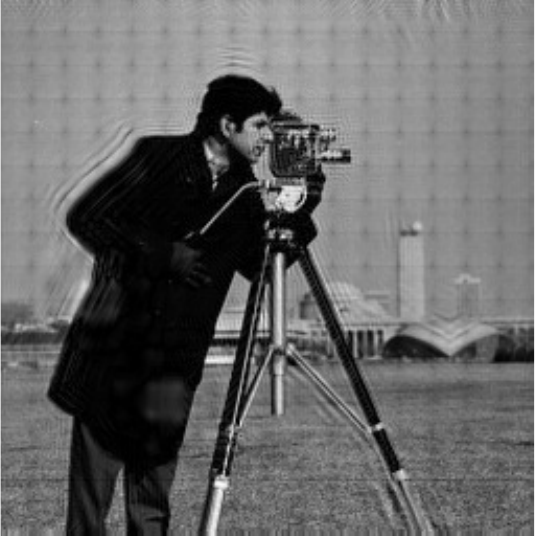}}\\
\subfigure[]{\includegraphics[width=.4\textwidth]{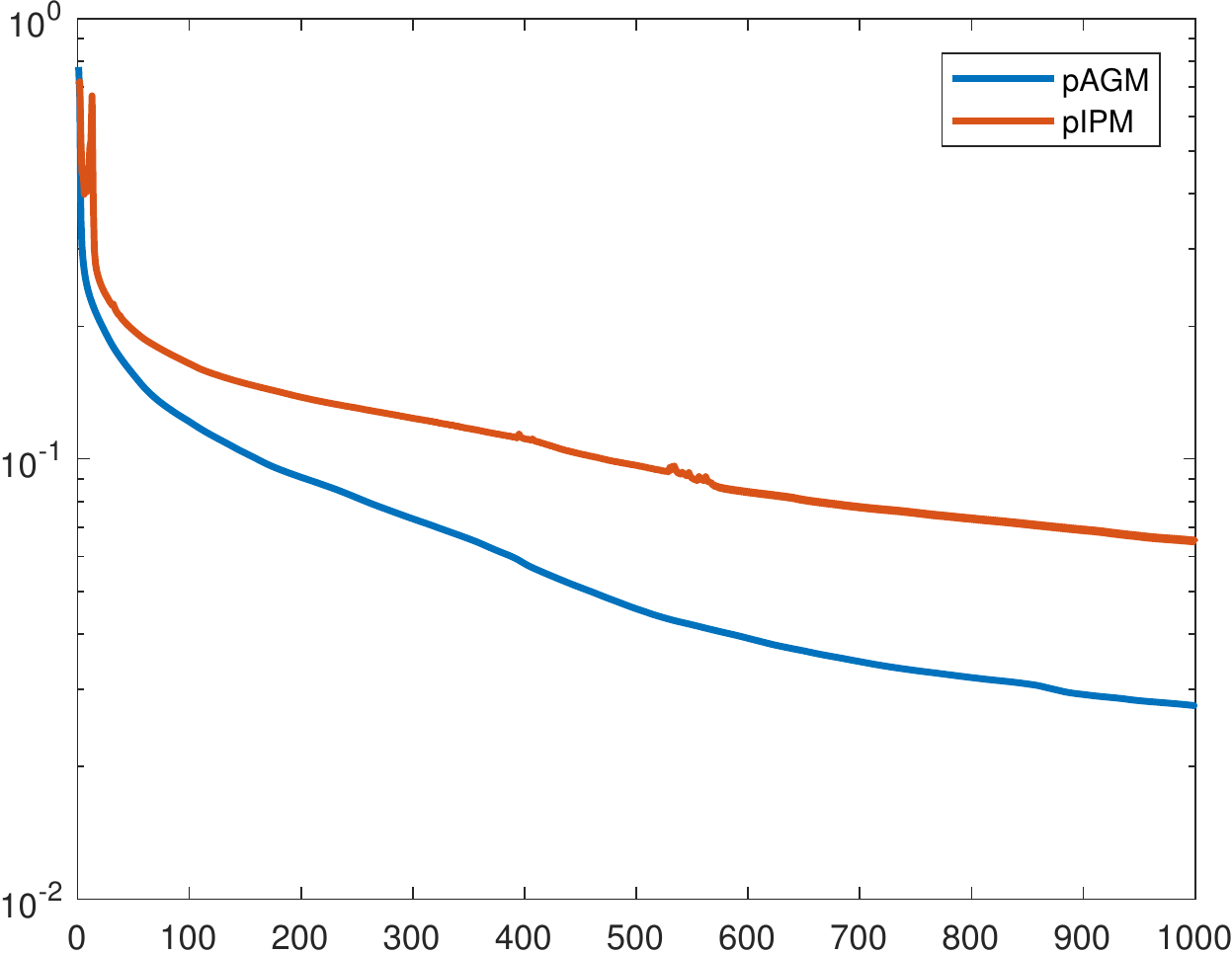}}
\end{center}
\caption
{ Truth in (a), recovery images in (b, c) for Model I with pAGM and pIPM, and convergence histories  (R-factor denotes the system residual) in (d) . All algorithms stop if  $\mathrm{R-factor}^k\leq 1.0\times 10^{-6}$ or iteration number reaches $1000.$ $\mathrm{Dist}=16.$ Stepsizes $(\tau^k_1,\tau^k_2)=$
$(1.0\times 10^{-6},1.5\times 10^{-11})$, and $(5.0\times 10^{-6}, 1.0\times 10^{-6})$ for (b, c) respectively.  Set $C_u=C_\omega=1.0\times 10^8.$ $\varepsilon=1.0\times 10^{-8}\times \|f\|_\infty.$
}
\label{fig1-R}
\end{figure}

Alternatively,   the generalized ADMM   will be adopted to solve  the proposed models, which permits bigger stepsizes by avoiding  directly calculating the gradient of the objective functional such that  fast convergence speed is gained. By introducing  new variables,  the proposed optimization problems  are decoupled by operator-splitting of  ADMM and the  resulted subproblems  can be easily solved.  {Additionally,  proximal terms in the subproblems  are introduced in order to guarantee the sufficient decrease of the augmented Lagrangian, such that the convergence of the proposed algorithm can be derived under mild conditions.}

\subsection{Generalized ADMM  for Model I}
An auxiliary variable
$z=\mathcal A(\omega, u)\in\mathbb C^{m}$ is introduced such that  an  equivalent form of \eqref{eqModel} is formulated as follows:
\begin{equation}\label{eqI-1}
\min\limits_{\omega,u,z} \mathcal G(z){+\mathbb I_{\mathscr X_1}(\omega)+\mathbb I_{\mathscr X_2}(u),}~~s.t.~~~z-\mathcal A(\omega, u)=0.
\end{equation}
The corresponding augmented Lagrangian reads
\begin{equation}\label{eqAL}
\begin{split}
\!\!\!\!\Upsilon_\beta&(\omega,u,z,\Lambda):=\mathcal G(z){+\mathbb I_{\mathscr X_1}(\omega)+\mathbb I_{\mathscr X_2}(u)}
+\Re(\langle z-\mathcal A(\omega, u), \Lambda \rangle)
+\tfrac{\beta}{2}\|z-\mathcal A(\omega, u)\|^2,\!\!\!\!
\end{split}
\end{equation}
with the multiplier $\Lambda\in \mathbb C^m$ and  a positive parameter $\beta,$
where $\langle\cdot,\cdot\rangle$ corresponds to the $L^2$ inner product in  complex Euclidean space and $\Re(\cdot)$ denotes the real part of a complex number.
Consequently, instead of  minimizing \eqref{eqModel} directly, one seeks a saddle point of the following problem:
\begin{equation}\label{eqSaddle}
\max_{\Lambda}\min\limits_{\omega,u,z} \Upsilon_\beta(\omega,u,z,\Lambda).
\end{equation}
A natural scheme to solve the above saddle point problem is to split them, which consists of four-step iterations for { the generalized ADMM (only the subproblems \emph{w.r.t.} $\omega$ or $u$ have proximal terms), as follows:
\begin{equation}\label{eqADMM}
\left\{
\begin{aligned}
&\mbox{Step 1:\quad} \omega^{k+1}=\arg\min\limits_\omega {\Upsilon_\beta(\omega,u^k,z^k,\Lambda^k)+\tfrac{\alpha_1}{2}\|\omega-\omega^k\|^2_{M_1^k}}\\
&\mbox{Step 2:}\quad u^{k+1}=\arg\min\limits_u  {\Upsilon_\beta(\omega^{k+1},u,z^k,\Lambda^k)+\tfrac{\alpha_2}{2}\|u-u^k\|^2_{M_2^k}}\\
&\mbox{Step 3:\quad} z^{k+1}=\arg\min\limits_z \Upsilon_\beta(\omega^{k+1},u^{k+1},z,\Lambda^k)\\
&\mbox{Step 4:\quad} \Lambda^{k+1}=\Lambda^{k}+\beta(z^{k+1}-\mathcal A(\omega^{k+1},u^{k+1})),\\
\end{aligned}
\right.
\end{equation}
{with  diagonal positive semidefinite  matrices  $M^k_1\in\mathbb R_+^{\bar m\times\bar m}$ and $M^k_2\in \mathbb R_+^{n\times n}$ (we call them preconditioning matrices as in \cite{zhang2011unified}) and two penalization parameters $\alpha_1, \alpha_2>0,$}
where $\|\omega\|^2_{M_1^k}:=\langle M_1^k\omega,\omega\rangle$ and
$\|u\|^2_{M_2^k}:=\langle M_2^ku,u\rangle$, given the approximated solution $(\omega^k,u^k,z^k,\Lambda^k)$ in the $k^{\text{th}}$ iteration.
 Here we remark that these two matrices $M_1^k, M_2^k$ are assumed to be diagonal so that subproblems in Step 1 and Step 2 have closed form solutions.}
{In practice, these two matrices are chosen heuristically, and a special strategy  to guarantee the convergence will be specified in section \ref{sec4} (See \eqref{eqM1}).}

\subsubsection{Subproblems w.r.t. $\omega$ and $u$}
First, we consider the subproblem in Step 1 of \eqref{eqADMM} \emph{w.r.t.} the probe $\omega$:
\[
\begin{split}
&\omega^{k+1}=\arg\min\limits_{\omega\in \mathscr X_1} \tfrac{1}{2}\| \hat z^k-\mathcal A(\omega,u^k)\|^2{+\tfrac{\alpha_1}{2\beta}\|\omega-\omega^k\|^2_{M_1^k}}\\
&=\arg\min\limits_{\omega\in \mathscr X_1} \tfrac{1}{2} {\sum\nolimits_j}\|\mathcal F^{-1} \hat z_j^k-\omega\circ \mathcal S_j u^k\|^2{+\tfrac{\alpha_1}{2\beta}\|\omega-\omega^k\|^2_{M_1^k}}\\
&=\arg\min\limits_{\omega\in \mathscr X_1} \sum\limits_{t=0}^{\bar m-1} \rho^k_t(\omega(t)), \text{~with~}
\hat z^k:=z^k+\tfrac{\Lambda^k}{\beta},
\end{split}
\]
 and
$\rho^k_t(x):=\tfrac{1}{2} \sum_j|(\mathcal F^{-1}  \hat z_j^k)(t)-x \times(\mathcal S_j u^k)(t)|^2{+\tfrac{\alpha_1}{2\beta} \mathrm{diag}(M_1^k)(t)\times|x-\omega^k(t)|^2}~\forall x\in\mathbb C,
$
{where $\omega(t)$ denotes the $t^{\text{th}}$ element of $\omega\forall 0\leq t \leq \bar m-1.$}
{ Essentially this subproblem \emph{w.r.t.} each element $\omega(t)$ is independent, such that one just needs to solve the following one-dimension constraint quadratic  problem:
\[
\omega^{k+1}(t)=\arg\min_{|x|\leq C_\omega} \rho_t^k(x)~0\leq t\leq \bar m-1.
\]
}
The derivative of $\rho^k_t(x)$  is calculated  as
\[
\begin{split}
&\quad\nabla \rho^k_t(x)\\
                &=\sum\nolimits_j\big(x\times \left|\left(\mathcal S_j u^k\right)(t)\right|^2-(\mathcal S_j u^k)^*(t)(\mathcal F^{-1}  \hat z_j^k)(t)\big){+\tfrac{\alpha_1}{\beta} \mathrm{diag}(M_1^k)(t)\times(x-\omega^k(t))}\\
                 &=x\times\big( \sum\nolimits_j\left|\left(\mathcal S_j u^k\right)(t)\right|^2+{\tfrac{\alpha_1}{\beta}\mathrm{diag}(M_1^k)(t)}\big)-\sum\nolimits_j\left((\mathcal S_j u^k)^*(t)(\mathcal F^{-1} \hat z_j^k)(t)\right)\\
                 &\hskip 7cm {-\tfrac{\alpha_1}{\beta} \mathrm{diag}(M_1^k)(t)\omega^k(t)},
\end{split}
\]
{where $\mathrm{diag}(M_1^k)(t)$ denotes the $t^{\text{th}}$ element of $\mathrm{diag}(M_1^k)$.}
{Letting $M_1^k$  satisfy
\begin{equation}\label{eqM1-Cond}
\min\nolimits_t  \sum\nolimits_j\left|\left(\mathcal S_j u^k\right)(t)\right|^2+\tfrac{\alpha_1}{\beta}\mathrm{diag}(M_1^k)(t)>0,
\end{equation}
the close form solution
 of Step 1 is given  as
\begin{equation}\label{eqOmega}
\omega^{k+1}=\mathrm{Proj}\Big(
\tfrac{\beta\sum\nolimits_j(\mathcal S_j u^k)^*\circ(\mathcal F^{-1} \hat z_j^k){+{\alpha_1} \mathrm{diag}(M_1^k)\circ\omega^k}}
{\beta\sum\nolimits_j\left|\mathcal S_j u^k\right|^2+{\alpha_1}\mathrm{diag}(M_1^k)}
; C_\omega\Big),
\end{equation}
with the projection operator\footnote{
It is the close form expression for the minimizer to the problem
$\min_{\|\tilde\omega\|_\infty\leq C_\omega} \tfrac12\|\tilde\omega-\omega\|^2.$
}
onto  $\mathscr X_1$ defined as
$\mathrm{Proj}(\omega;C_\omega):=\min\{C_\omega, |\omega|\}\circ \mathrm{sign}(\omega).$
}

 %
Second, we consider the subproblem in Step 2 of \eqref{eqADMM} \emph{w.r.t.} the variable $u$ as
\[
\begin{aligned}
u^{k+1}=&\arg\min_{u\in \mathscr X_2} \tfrac{\beta}{2}\sum\nolimits_j\|\hat z_j^k\!-\! \mathcal F(\omega^{k+1} \circ \mathcal S_j u)\|^2{+{\alpha_2}\|u-u^k\|^2_{M_2^k}}\\
=&\arg\min_u \tfrac{\beta}{2}\sum\nolimits_j\|\mathcal S_j^T\mathcal F^{-1}\hat z_j^k- (\mathcal S_j^T\omega^{k+1}) \circ  u\|^2{+{\alpha_2}\|u-u^k\|^2_{M_2^k}}.
\end{aligned}
\]
Similarly to \eqref{eqOmega},
{letting $M_2^k$  satisfy
\begin{equation}\label{eqM2-Cond}
\min\nolimits_t  \sum\nolimits_j\left|\left(\mathcal S_j^T \omega^{k+1}\right)(t)\right|^2+\tfrac{\alpha_2}{\beta}\mathrm{diag}(M_2^k)(t)>0,
\end{equation}
}
we have
{
\begin{equation}\label{eqU}
u^{k+1}=\mathrm{Proj}\Big(\tfrac{\beta\sum_j \mathcal S_j^T ((\omega^{k+1})^*\circ\mathcal F^{-1} \hat z_j^k)+\alpha_2\mathrm{diag}(M_2^k)\circ u^k}{\beta \sum_j (\mathcal S_j^T |\omega^{k+1}|^2) +\alpha_2\mathrm{diag}(M_2^k)}; C_u \Big).
\end{equation}
}

%
%

\subsubsection{Subproblem w.r.t. $z$}

The subproblem  \emph{w.r.t.} the variable $z$ reads
\begin{equation}\label{eqZ}
\begin{split}
\!\!\!z^{k+1}\!=\arg\min\limits_z \mathcal G(z)+\tfrac{\beta}{2}\big\|z- \mathcal A(\omega^{k+1},u^{k+1})+\tfrac{\Lambda^k}{\beta}\big\|^2=\mathrm{Prox}^\beta_{\mathcal G}\big(\mathcal A(\omega^{k+1},u^{k+1})-\tfrac{\Lambda^k}{\beta}\big),\!\!\!\\
\end{split}
\end{equation}
where the proximal operator \cite{moreau1962fonctions} $\mathrm{Prox}_{\mathcal G}^\beta: \mathbb C^m\rightarrow \mathbb C^m$ is defined as:
\[
\mathrm{Prox}_{\mathcal G}^\beta(z):=\arg\min_{\hat z} \mathcal G(\hat z)+\tfrac{\beta}{2}\|\hat z-z\|^2,
\]
and it is simplified by  $\mathrm{Prox}_{\mathcal G}(z)$ in the case of $\beta=1.$


For the case of pIPM, we have
\[
z^{k+1}=\arg\min_z \tfrac12\langle |z|^2+\varepsilon\mathbf 1_m-(f+\varepsilon\mathbf 1_m)\circ\log(|z|^2+\varepsilon\mathbf 1_m), \mathbf 1_m \rangle+\tfrac{\beta}{2}\|z-z^+\|^2,
\]
with $z^+=\mathcal A(\omega^{k+1},u^{k+1})-\tfrac{\Lambda^k}{\beta}.$
The solution can be expressed as
\begin{equation}\label{eqZ-1}
z^{k+1}=\rho^\star \circ\mathrm{sign}(z^+),
\end{equation}
 where
\begin{equation}\label{eqZ-2}
\rho^\star(t)=\arg\min_{0\leq x\in\mathbb R}\tfrac12 x^2-(f(t)+\varepsilon)\log(x^2+\varepsilon)+\tfrac\beta2(x-|z^+(t)|)^2.
\end{equation}
The first optimality condition of the above optimization problem without constraints yields a quartic equation, such that it is possible to compare the function values at all non-negative roots. The root with the minimal value is the exact solution to this optimization problem. However, the computation cost to calculate all  roots of a quartic equation is still relatively high. Alternatively, the gradient projection scheme can be expressed as:
\begin{equation}\label{eqZ-4}
x_{l+1}=\max\left\{0,x_l-\delta \big((1+\beta-\tfrac{f(t)+\varepsilon}{|x_l|^2+\varepsilon})x_l-\beta z^+(t) \big)\right\}, \forall~l=0,1,\ldots
\end{equation}
with the stepsize $\delta>0,$ and   $x_0:=|z^k(t)|.$

For the case of pAGM, similarly to the case of pIPM, the solution $z^{k+1}$ can be given by \eqref{eqZ-1}, and the iterative scheme for $\rho^\star(t)$ is given as:
\begin{equation}\label{eqZ-3}
x_{l+1}=\max\left\{0, x_l-\delta \big((1+\beta-\tfrac{\sqrt{f(t)+\varepsilon}}{\sqrt{|x_l|^2+\varepsilon}})x_l-\beta z^+(t) \big)\right\}, \forall~l=0,1,\ldots
\end{equation}
Based on the above calculations, the generalized ADMM for Model I is summarized in Algorithm \ref{algADMM}.

\vskip .1in
\begin{algorithm}[tb!]
\caption{\qquad\qquad  Generalized ADMM for Model I in \eqref{eqModel}}
\begin{algorithmic}[1]
\REQUIRE~~Set $\omega^0, u^0, z^0=\mathcal A(\omega^0, u^0), \Lambda^0=0,  k:=0,$ maximum iteration number $Iter_{Max}$, and parameters $\beta$, {$\alpha_1$ and $\alpha_2$.}
\ENSURE~~$u^\star:=u^{Iter_{Max}}$ and $\omega^\star:=\omega^{Iter_{Max}}.$\vskip .1in
\FOR{$k=0$ to $Iter_{Max}-1$}
\STATE Compute $\omega^{k+1}$ by \eqref{eqOmega} with $\hat z^k:=z^k+\tfrac{\Lambda^k}{\beta},$ where $M_1^k$ satisfies \eqref{eqM1-Cond}.
\STATE
Compute $u^{k+1}$ by \eqref{eqU}, where $M_2^k$ satisfies \eqref{eqM2-Cond}. 
\STATE Compute $z^{k+1}$ by \eqref{eqZ}.
\STATE Update the multiplier as Step 4 of \eqref{eqADMM}.
\ENDFOR
\end{algorithmic}
\label{algADMM}
\end{algorithm}


\subsection{Generalized ADMM for Model II}
By introducing two auxiliary variables $z_1\in \mathbb C^m$ and $z_2\in \mathbb C^{\bar m}$, an equivalent form of \eqref{eqModelpriorII} can be obtained as:
\begin{equation*}
\begin{split}
 \min\limits_{\omega,u,z_1,z_2} &\qquad \mathcal G(z_1){+\mathbb I_{\mathscr X_1}(\omega)+\mathbb I_{\mathscr X_2}(u)}+\tau\widehat{\mathcal G}(z_2),\\
         s.t.     ~~  & \qquad z_1-\mathcal A(\omega, u)=0,~z_2-\mathcal F \omega=0.
 \end{split}
\end{equation*}
The corresponding augmented Lagrangian reads
\begin{equation}\label{eqALII}
\begin{split}
&\widehat\Upsilon_{\beta_1,\beta_2}(\omega,u,z_1,z_2,\Lambda_1,\Lambda_2 ):=\mathcal G(z_1){+\mathbb I_{\mathscr X_1}(\omega)+\mathbb I_{\mathscr X_2}(u)}+\Re(\langle z_1-\mathcal A(\omega, u), \Lambda_1 \rangle)\\
&\qquad\quad+\tfrac{\beta_1}{2}\|z_1-\mathcal A(\omega, u)\|^2
+\tau\widehat{\mathcal G}(z_2)
+\Re(\langle z_2-\mathcal F \omega, \Lambda_2 \rangle)
+\tfrac{\beta_2}{2}\|z_2-\mathcal F \omega\|^2,
\end{split}
\end{equation}
with two multipliers $\Lambda_1\in \mathbb C^m, \Lambda_2\in \mathbb C^{\bar m}$ and positive parameters $\beta_1$ and $\beta_2.$
Consequently, one needs to solve a saddle point problem as follows:
\[
\max_{\Lambda_1,\Lambda_2}\min\limits_{\omega,u,z_1,z_2} \widehat\Upsilon_{\beta_1,\beta_2}(\omega,u,z_1,z_2,\Lambda_1,\Lambda_2).
\]
The generalized ADMM  is also used to solve the above problem,
and in the following discussion, we focus on  the subproblems of ADMM.
First, we consider the subproblem  \emph{w.r.t.}  $\omega$:
\[
\!\!\omega^{k+1}
=\arg\min\limits_{\omega\in \mathscr X_1} 
\tfrac{\beta_1}{2}\big\|z^k_1+\tfrac{\Lambda^k_1}{\beta_1}-\mathcal A(\omega, u^{k})\big\|^2+\tfrac{\beta_2}{2}\big\|\mathcal F^{-1}(z^k_2+\tfrac{\Lambda^k_2}{\beta_2})-\omega\big\|^2
+\tfrac{\bar\alpha_1}{2}
\|\omega-\omega^k\|^2_{\bar M_1^k},\!\!
\]
{
with the penalization parameter $\bar\alpha_1>0$ and the preconditioning matrix $\bar M_1^k.$}
Similarly to \eqref{eqOmega}, by computing the first-order derivative,  one readily obtains the solution:
\begin{equation}\label{eqOmega-2}
{{ \omega}^{k+1}=\mathrm{Proj}\Big(\tfrac {\beta_1\sum\nolimits_j(\mathcal S_j u^k)^*\circ(\mathcal F^{-1}\bar z^k_{1,j})+\beta_2 \mathcal F^{-1}\bar z^k_{2}+\bar\alpha_1 \bar M_1^k \omega^k}
{\beta_1\sum\limits\nolimits_j\left|\mathcal S_j u^k\right|^2+\beta_2\bm 1_{\bar m}+\bar \alpha_1 \mathrm{diag}(\bar M_1^k) };
C_\omega\Big),}
\end{equation}
with  $\bar z_l^k=z_l^k+\tfrac{\Lambda^k_l}{\beta_l},~l=1,2$ and
$\bar z_1^k=((\bar z_{1,0}^k)^T,(\bar z_{1,1}^k)^T,\cdots,(\bar z_{1,J-1}^k)^T)^T\in \mathbb C^m$ with $\bar z^k_{1,j}\in \mathbb C^{\bar m}~\forall 0\leq j\leq J-1.$

Next, we consider the subproblem \emph{w.r.t.} $u$. According to  \eqref{eqU},  its solution is directly given below:
{
\begin{equation}\label{eq16}
\min_t \beta_1\sum_j (\mathcal S_j^T |\omega^{k+1}|^2)(t)+ \bar\alpha_2\mathrm{diag}(\bar M_2^k)(t)>0,
\end{equation}
\begin{equation}\label{eqU-2}
u^{k+1}=\mathrm{Proj}\Big(
\tfrac{\beta_1\sum_j (\mathcal S_j^T ((\omega^{k+1})^*\circ\mathcal F^{-1} \bar z_{1,j}^k))+\bar\alpha_2\bar M_2^k u^k}{{ \beta_1\sum_j (\mathcal S_j^T |\omega^{k+1}|^2)+ \bar\alpha_2\mathrm{diag}(\bar M_2^k)}}; C_u\Big),
\end{equation}
with the penalization parameter $\bar \alpha_2>0,$ and the preconditioning matrix $\bar M_2^k.$
}

Regarding the subproblems \emph{w.r.t.}  $z_1$ and $z_2,$  their closed form solutions can be directly derived as follows:
\begin{equation}\label{eqZII}
z_1^{k+1}=\mathrm{Prox}^{\beta_1}_{\mathcal G}\big(\mathcal A(\omega^{k+1}, u^{k+1})-\tfrac{\Lambda_1^k}{\beta_1}\big),\text{~and~}
z_2^{k+1}=\mathrm{Prox}_{\widehat {\mathcal G}}^{\beta_2/\tau}\big(\mathcal F\omega^{k+1}-\tfrac{\Lambda_2^k}{\beta_2}\big),
\end{equation}
where  the proximal operators can  be solved similarly to \eqref{eqZ-1}. 

The generalized ADMM for Model II is summarized in Algorithm~\ref{algADMM-II}.
\begin{algorithm}[]
\caption{\qquad\qquad\qquad Generalized ADMM for Model II in \eqref{eqModelpriorII}}
\begin{algorithmic}[1]
\REQUIRE~~Set $\omega^0, u^0, z_1^0=\mathcal A(\omega^0, u^0), z_2^0=\mathcal F \omega^0, \Lambda_1^0=0, \Lambda_2^0=0, k:=0,$ maximum iteration number $ Iter_{Max}$, and parameters $\beta_1, \beta_2,\tau,$  {$\bar\alpha_1$ and $\bar\alpha_2$.}
\ENSURE~~$u^\star:=u^{Iter_{Max}}$ and $\omega^\star:=\omega^{Iter_{Max}}.$\vskip .1in
\FOR{$k=0$ to $Iter_{Max}-1$}
\STATE Compute $\omega^{k+1}$ by \eqref{eqOmega-2} with $\bar z_l^k:=z_l^k+\tfrac{\Lambda_l^k}{\beta}, l=1,2.$
\STATE
Compute $u^{k+1}$ by \eqref{eqU-2} where $\bar M_2^k$ satisfies \eqref{eq16}. 
\STATE Compute $z_1^{k+1}, z_2^{k+1} $ by \eqref{eqZII}.
\STATE Update multipliers as
\[
\begin{split}
&\Lambda_1^{k+1}=\Lambda_1^{k}+\beta_1(z_1^{k+1}-\mathcal A(\omega^{k+1}, u^{k+1})),\\
&\Lambda_2^{k+1}=\Lambda_2^{k}+\beta_2(z_2^{k+1}-\mathcal F \omega^{k+1}).
\end{split}
\]
\ENDFOR
\end{algorithmic}
\label{algADMM-II}
\end{algorithm}

\vskip .05in
Since $\mathcal F$ is an orthogonal operator, and the operator $\mathcal A$ is block-wise defined based on element-wise multiplications of the probe and image  for Ptycho-PR, their subproblems \emph{w.r.t.} $\omega$ and $u$   have closed form solutions.
We remark that, for the $z-$subproblem,  very few iterations are needed to guarantee the convergence of the entire ADMM in practice, due to the error-forgetting property \cite{yin2013error}.

\section{Convergence Analysis}\label{sec4}
First, we briefly review the framework for the convergence analysis of an iterative algorithm for nonconvex optimization problems. Consider an optimization problem
$\min_v \Phi(v),$
with the functional $\Phi$ being  proper, lower semi-continuous, and bounded below.
Let an iterative sequence $\{v^k\}_{k=0}^\infty$ be generated by some algorithm to solve the above minimization problem. {In light of \cite{attouch2010proximal,bolte2014proximal,hong2016convergence,wang2015global,lou2016fast,mei2016cauchy,yang2017alter},} in order to prove that the iterative sequence  converges to stationary points of $\Phi$, the following three conditions should  be satisfied:

\vspace{1mm}

\hspace{-3mm} \textbf{i)} Sufficient decrease condition as
\vspace{1mm}
\begin{equation}\label{condSD}
\Phi(v^k)-\Phi(v^{k+1})\geq c_1\|v^{k+1}-v^k\|^2, \phantom{..}\mbox{with a positive constant $c_1$.}
\end{equation}

\vspace{1mm}
\hspace{-3mm} \textbf{ii)} Relative error condition as
\vspace{1mm}
\begin{equation}\label{condRE}
\|p^{k+1}\|\leq c_2\|v^{k+1}-v^k\|, \mbox{ with some $p^{k+1}\in \partial \Phi(v^{k+1})$, and a positive constant $c_2$.}
\end{equation}

\hspace{-3mm} \textbf{iii)} Kurdyka-{\L}ojasiewicz (KL) property~\cite{attouch2010proximal}.
\vspace{2mm}

The first two conditions can guarantee that each limiting point of the iterative sequence is a stationary point of $\Phi,$ which demonstrates the subsequence convergence. Furthermore, by the KL property of $\Phi$, the iterative sequence can be proved to be a Cauchy sequence and therefore convergence can be reached.
Along this line, the convergence for  ADMM  was established  for noncovex  problems~\cite{hong2016convergence,wang2015global,lou2016fast,mei2016cauchy,yang2017alter}. 

In this section we analyze the convergence of the proposed generalized ADMM algorithm for this nonconvex optimization problem with the bilinear constraint. 
The sufficient-overlap condition of iterative sequences $\{\omega^k\}$ and $\{u^k\}$, which will be assumed to be  bounded below by a positive constant,  the difficulty brought by the bilinear constraint  can be overcome. Thanks to the Lipschitz property of pAGM/pIPM based $\mathcal G$ (in Lemma \ref{assump1}), and the KL property of the objective functional of the proposed models, we shall be able to prove the convergence of the proposed algorithms.

The following notation is used for the successive error of the iterative sequences:
\begin{equation}\label{eqST}
E_\omega^{k+1}:=\omega^{k+1}\!-\!\omega^k, E_u^{k+1}:=u^{k+1}\!-\!u^k,
E_z^{k+1}:=z^{k+1}\!-\!z^k, E_\Lambda^{k+1}:=\Lambda^{k+1}\!-\!\Lambda^k.
\end{equation}

\vskip .1in
\subsection{Convergence of Algorithm \ref{algADMM}}\label{sec4-1}

We show the first order optimality conditions for  subproblems of \eqref{eqADMM}.
Letting $X^k:=(\omega^k, u^k, z^k,\Lambda^k)$ be generated by  Algorithm \ref{algADMM}, the following relations hold:
{
\begin{align}
&0\in{\omega}^{k+1}\circ{\sum\nolimits_j\left|\mathcal S_j u^k\right|^2}-{\sum\nolimits_j(\mathcal S_j u^k)^*\circ\mathcal F^{-1}(z_j^k+\tfrac{1}{\beta}\Lambda_j^k)}\label{eqOptC-1}\\
&\hskip 3cm+\tfrac{1}{\beta}\partial \mathbb I_{\mathscr X_1}(\omega^{k+1}) +\tfrac{\alpha_1}{\beta}\mathrm{diag}(M_1^k) \circ E_\omega^{k+1};\nonumber\\
&0\in u^{k+1}\circ\!\sum\nolimits_j\mathcal S_j^T|\omega^{k+1}|^2\!-\sum\nolimits_j\mathcal S_j^T((\omega^{k+1})^*\circ\mathcal F^{-1}( z_j^k\!+\tfrac{1}{\beta}\Lambda_j^k))\label{eqOptC-2}\\
&\hskip 3cm+\tfrac{1}{\beta}\partial \mathbb I_{\mathscr X_2}(u^{k+1})+\tfrac{\alpha_2}{\beta}\mathrm{diag}(M_2^k)\circ E_u^{k+1};\nonumber\\
&0=\nabla \mathcal G(z^{k+1})+\beta(z^{k+1}-\mathcal A(\omega^{k+1},u^{k+1}))+\Lambda^{k}.
\label{eqOptC-3}
\end{align}
}
Using \eqref{eqOptC-3} and the Step 4 of \eqref{eqADMM}, one can derive
\begin{equation}\label{eqOptC-3-2}
0=\nabla \mathcal G(z^{k+1})+\Lambda^{k+1},~\forall k\geq 0.
\end{equation}

%

%

We first estimate the differences of the augmented Lagrangian  between two successive iterations to show the sufficient decrease as in \eqref{condSD}.
\begin{lem}[Sufficient decrease condition]
\label{lem1}
For all $k\geq 1,$ we have
\begin{equation}\label{eqlem1}
 \begin{split}
 \!\Upsilon_\beta(X^k)\!-\! \Upsilon_\beta(X^{k+1})
 &\geq  \tfrac{\beta}{2} I^k_u\|E_\omega^{k+1}\|^2\!+\!\tfrac{\beta}{2}I^k_\omega \|E_u^{k+1}\|^2\!+\!(\tfrac{\beta\!-\!3L}{2}\!-\!\tfrac{L^2}{\beta})\|E_z^{k+1}\|^2,
 \end{split}
\end{equation}
where $I^k_u\in \mathbb R_+$ and  $I^k_\omega\in \mathbb R_+$ are defined below:
\begin{equation}\label{condBDB}
\left\{
\begin{split}
&I^k_u:=\min_{0\leq t_1\leq \bar m-1}\sum\nolimits_j|(\mathcal S_j u^k)(t_1)|^2{+\tfrac{2\alpha_1}{\beta}\mathrm{diag}(M_1^k)(t_1),}\\
&I^k_\omega:=\min_{0\leq t_2\leq n-1}\sum\nolimits_j |(\mathcal S^T_j \omega^{k+1})(t_2)|^2{+\tfrac{2\alpha_2}{\beta}\mathrm{diag}(M_2^k)(t_2).}
\end{split}
\right.
\end{equation}
\end{lem}
The corresponding proof can be found in Appendix~\ref{apdx1-2}.
We show  the boundedness of the iterative sequence $\{X^k\}$ in the following lemma:
%
%


\begin{lem}(Boundedness)\label{lem2}
Letting $\beta>\tfrac{3+\sqrt{17}}{2}L$, the  sequence $\{X^k\}$ generated by Algorithm \ref{algADMM} is bounded. Furthermore, the sequence of augmented Lagrangian $\{\Upsilon_\beta(X^{k})\}$ is  bounded  and non-increasing.
\end{lem}

The proof can be found in Appendix \ref{apdx-2}.
In order to make sure that the decreasing quantity of  the augmented Lagrangian is bounded below by $\|X^{k+1}-X^k\|^2$, we need the following assumption:
\begin{assump}[Sufficient-Overlap]\label{assump2}
The sequences $\{I^k_u\}\subseteq \mathbb R_+, \{I^k_\omega\}\subseteq \mathbb R_+$, defined in \eqref{condBDB}, are bounded below by two constants $C_1$ and $C_2$, respectively,  \emph{i.e.} 
\begin{align}
&\inf\nolimits_k I^k_u\geq C_1>0, \qquad\inf\nolimits_k  I^k_\omega\geq C_2>0.\label{eqassump2}
\end{align}
\end{assump}

\begin{rem}
\label{rem4-2}
{We discuss the above assumption for standard ADMM by setting $M^k_1=0,  M_2^k=0$ in Algorithm \ref{algADMM}, since it can work  well in practice (further details can be found in  Section~\ref{sec5}).}  Assuming that the iterative sequences $\{\omega^k\}$ and $\{u^k\}$ converge to the ground truth solutions $\omega^\star$ and $u^\star$ for \eqref{eqModel}, which are also assumed to be nonzeros,  under Assumption \ref{assump2},  the following inequalities hold:
\begin{equation}\label{eqAA}
\begin{split}
\min_{0\leq t_1\leq \bar m-1}\sum\nolimits_j |(\mathcal S_j u^\star)(t_1)|^2\geq C_1,\qquad
\min_{0\leq t_2\leq n-1}\sum\nolimits_j |(\mathcal S^T_j \omega^\star)(t_2)|^2\geq C_2,
\end{split}
\end{equation}
which is a necessary condition for Assumption \ref{assump2}.
Equation~\eqref{eqAA} holds when the sample is scanned with a small enough sliding distance so that each scan position sufficiently overlaps with the next.
There must exist a positive number $k_0\geq 0$, such that for all $k\geq k_0$ the iterative sequences $\{I_u^k\}_{k\geq k_0}$ and $\{I_\omega^k\}_{k\geq k_0}$ are bounded below by some strictly positive constants. That explains the reasonability of Assumption \ref{assump2}. In practice, in order to guarantee the convergence for Ptycho-PR, scanning with smaller enough sliding distances is always needed so that enough redundancy between measurements can be attained.
However, such assumption for standard ADMM is difficult to verify.
\end{rem}

{ Due to the introduction of the proximal terms,
these  matrices can also be given  by the following strategy:
\begin{equation}
\label{eqM1}
\mathrm{diag}(M_1^k)(t_1)=
\left\{
\begin{split}
&s_1,\qquad\text{if~}h_1^k\leq s_1\\
&{r_{1}h_1^k},\quad\text{otherwise,}
\end{split}
\right.
~~
\mathrm{diag}(M_2^k)(t_2)=
\left\{
\begin{split}
&s_2, \qquad\text{if~}h_2^k\leq s_2\\
&r_{2}{h_2^k}, \quad\text{otherwise,}
\end{split}
\right.
\end{equation}
where
$h_1^k:=\big\|\sum\nolimits_j |(\mathcal S_j u^{k})|^2\big\|_\infty,$
$h_2^k:=\big\|\sum\nolimits_j |(\mathcal S^T_j \omega^{k+1})|^2\big\|_\infty,$
$s_1, s_2$ are  positive safeguard parameters, and $r_{1}, r_{2}$ are  positive constants. 
 One readily knows that  Assumption \ref{assump2} holds with
$C_j=\tfrac{2s_j\alpha_j}{\beta}\min\left\{1,r_{j}\right\}>0, j=1,2$ in \eqref{eqassump2}. This strategy also guarantees that \eqref{eqM1-Cond} and \eqref{eqM2-Cond} hold, and the sequences of these two preconditioning matrices are bounded due to the boundednesses of  $\{\omega^k\}$ and $\{u^k\}$.
}

\vskip .05in

%

The relative error condition, similarly to \eqref{condRE}, is derived in the following lemma:
\begin{lem}[Relative error condition]\label{lem3}
By letting $\beta>\tfrac{3+\sqrt{17}}{2}L$
 {and letting $\{\mathrm{diag}(M_1^k)\}, \{\mathrm{diag}(M_2^k)\}$ be bounded,} we have:
\begin{equation*}
{\mathrm{dist}(0,\partial \Upsilon_\beta (X^{k+1}))}\leq
C_3\|X^{k+1}-X^k\|~\forall  k\geq 0,
\end{equation*}
with a positive constant $C_3$.
\end{lem}

The corresponding proof can be found in Appendix \ref{apdx1-4}.
At this point, the convergence of Algorithm~\ref{algADMM} can be derived in the following theorem:
\begin{thm}\label{thm2}
Under Assumption \ref{assump2} and  $\beta>4L$, {and letting $\{\mathrm{diag}(M_1^k)\},$ $\{\mathrm{diag}(M_2^k)\}$ be bounded,}
 the iterative sequence $\{X^k\}$ generated by Algorithm \ref{algADMM}  converges to a stationary point $X^\star$ of the augmented Lagrangian $\Upsilon_\beta$ in \eqref{eqAL} for Model I in \eqref{eqModel} with wIGM,
 \emph{i.e.}
 $
 \lim_{k\rightarrow +\infty} X^k=X^\star,  \mbox{with~} {0\in \partial \Upsilon_\beta(X^\star)}.
 $
\end{thm}

The proof can be found in Appendix \ref{apdx1-5}.
 We remark that Assumption \ref{assump2} can be verified when setting the preconditioning matrices as in \eqref{eqM1}, as well as the boundedness of the sequences of preconditioning matrices.
Therefore, with suitable stepsizes and careful selection of the preconditioning matrices,  one can derive the convergence of the proposed algorithms for Model I.

%

%

\subsection{Convergence of Algorithm \ref{algADMM-II}}\label{sec4-2}
Similarly to Assumption~\ref{assump2}:
\begin{assump}[Sufficient-Overlap]\label{assumpII-2}
There exists a constant $\widehat C_1$ independent of $k$, such that $I_\omega^k$ defined in \eqref{condBDB} is bounded below, \emph{i.e.}
$I^k_\omega\geq \widehat C_1>0.$
\end{assump}

Based on the above two assumptions, we
list the theorem below without proof, since it directly follows the proof of Theorem \ref{thm2}.
\begin{thm}\label{thm3}
Under Assumptions \ref{assumpII-2}, {and letting $\{\mathrm{diag}(\bar M_1^k)\},$ $\{\mathrm{diag}(\bar M_2^k)\}$ be bounded,} the iterative sequence  generated by Algorithm \ref{algADMM-II} converges to a stationary point of the augmented Lagrangian $\widehat \Upsilon_{\beta_1,\beta_2}$ in \eqref{eqALII}, provided that the parameters $\beta_1, \beta_2$ are sufficiently large.
 \end{thm}

In order to derive the convergence of Algorithm \ref{algADMM-II}, only the positive lower bound of $I_\omega^k$ is assumed. Hence,  the convergence analysis of Algorithm \ref{algADMM-II} needs  weaker conditions, compared to  Algorithm \ref{algADMM}. For the convergence guarantee of Algorithm \ref{algADMM-II}, the boundedness of $I_\omega^k$  can be verified after the careful selection of the preconditioning matrix $\bar M_2^k$, similarly to \eqref{eqM1}.

%

%

\vskip .2in
\section{Numerical Experiments}\label{sec5}

In the following experiments we consider that the probe scans the image of interest with periodical boundary condition on square, hexagonal or random lattices. The random lattice is generated based on the square lattice with additional random offsets within one pixel. {The ground truth image and probe employed in this experiments can be found in Figure \ref{fig3} (a) and Figure \ref{fig12-0}.}

In order to evaluate performances of different algorithms, we introduce two criteria. First, the R-factor
for Algorithms \ref{algADMM} and \ref{algADMM-II}, defined
as:
\[
\text{R-factor}^k:= \tfrac{\sum\nolimits_j\big\||\mathcal A_j (\omega^k, u^k)|-\sqrt {f_j} \big\|_1}{\|\sqrt{f}\|_1}.\!
\]
If the R-factor of the iteration solutions of Algorithm~\ref{algADMM} or \ref{algADMM-II} tends to zero as $k\rightarrow \infty$ for the noiseless case,  they converge to  global minimizers of the proposed models, which are also   the stationary points. Therefore,  in the case of noiseless data, R-factors can be used to verify the convergence.
Secondly, the SNR (Signal-to-Noise Ratio) of $u^k$ is also employed, denoted as:
 \[
 \mathrm{SNR}(u^k, u_g)=-10\log_{10}{\sum\limits_{t=0}^{n-1}|\zeta^* u^k(t+T^*)-u_g(t)|^2}/{\|\zeta^* u^k\|^2},
 \]
with  the ground truth image $u_g$, where the error is computed up to the translation $T^*$, phase shift and scaling factor $\zeta^*$ determined by
 \[
 (\zeta^*,T^*):=\arg\min\limits_{\zeta\in \mathbb C, T\in \mathbb Z}\sum_{t}|\zeta u^k(t+T)-u_g(t)|^2.
 \]
 Similarly, we can denote  the $\mathrm{SNR}(\omega^k, \omega_g)$ for the probe, with $\omega_g$ being the ground truth.
We also consider measurements contaminated by Poisson noise. {In order to simulate different peak values,  a factor $\eta>0$ is introduced \cite{chang2016general} on the image, such that the Poisson noise corrupts the clean intensity as $f(t)\stackrel{\mathrm {i.d.d.}}{\sim} \mathrm{Poisson} ((\bm a^\eta)^2(t))$, where $\bm a^\eta:=|\mathcal A(\omega_g,\eta u_g)|$ with ground truth $\omega_g$ and $u_g.$
One readily knows that noise level gets higher if $\eta$ gets smaller. For the noisy measurements,  $\mathrm{SNR}_{intensity}$ \cite{chang2016Total} is used to measure the noise level, which is denoted below:
\[
\mathrm{SNR}_{intensity}:=-10\log_{10}{\| f-(\bm a^\eta)^2}\|^2/{\|(\bm a^\eta)^2\|^2}.
\]
}

In the following tests, we set $u^0=\bm 1_{n}$ and $\omega^0=\tfrac{1}{J}\mathcal F^{-1}(\sum_j\sqrt{f_j})$, where the initial guess of the probe has zero phase for Algorithm \ref{algADMM}-\ref{algADMM-II}.
{The parameters $C_\omega,$ $C_u$ need to be sufficiently large, and we employ $C_\omega=C_u=1.0\times 10^8$ for the constraint sets in Models I and II.} The same initial values of $u$ and $\omega$ are used for other compared algorithms. Other variables including auxiliary variables and multipliers  are initialized as in the initialization stage in  Algorithms \ref{algADMM}-\ref{algADMM-II}.
The parameters for the proposed algorithms are problem-dependent, which will be specified for different experiments.
The codes are implemented in \textsc{Matlab}.

{ We remark that no obvious acceleration of the proposed algorithms can be observed when using more inner iterations for the subproblems of $z^k$ in Algorithm \ref{algADMM}  and $z^k_1, z_2^k$ in Algorithm \ref{algADMM-II}. Due to the page limit, we do not report further details. Therefore, to reduce the overall computation cost, only one inner iteration is adopted.}
\subsection{Performance of Algorithm \ref{algADMM}}
We show the performance of Algorithm \ref{algADMM} with pAGM ($\varepsilon=1.0\times 10^{-8}\times\|f\|_\infty$), and then  compare it with the state-of-the-art algorithms: DR \cite{thibault2009probe}, ePIE \cite{maiden2009improved}  and  PALM \cite{hesse2015proximal}. The comparison algorithms are tuned and executed with the parameters that maximize their performance.

\subsubsection{Comparison with  DR, ePIE  and PALM}
We conduct experiments in the case of both noiseless and noisy measurements.

\textbf{i) Noiseless case.}
We conduct the experiments with two different sliding distances $D=24, 16$.
 Two different scanning lattices including square and random lattices are considered.
 {We show the performances of Algorithm \ref{algADMM} in the following two cases: the standard ADMM with $M_1^k=M_2^k=0$ (referred below as ADMM) and the generalized version \eqref{eqM1} with two additional proximal terms (ADMM-Prox), with
 $r_{1}=1.0\times 10^{-6}, r_{2}=1.0\times 10^{-3}, s_1=s_2=1.0\times 10^{-6}$, and $\alpha_1=\alpha_2=\beta$.
}
{First of all,
in order to verify the sufficient-overlap condition,   we plot  the histories of  $I^k_u$ and $I_\omega^k$ in Figure \ref{fig0} ($D=24$). One readily knows that they quickly become stable, without approaching infinity or zero. Remarkably, without the proximal terms, the iterative sequences $\{I^k_u\}$ and $\{I_\omega^k\}$ by the standard ADMM also quickly tend to non-zero values. With the additional proximal terms, the sequences of ADMM-Prox get stable faster than those of the standard ADMM.
}

\begin{figure}[h!]
\begin{center}
\subfigure[]{\includegraphics[width=.24\textwidth]{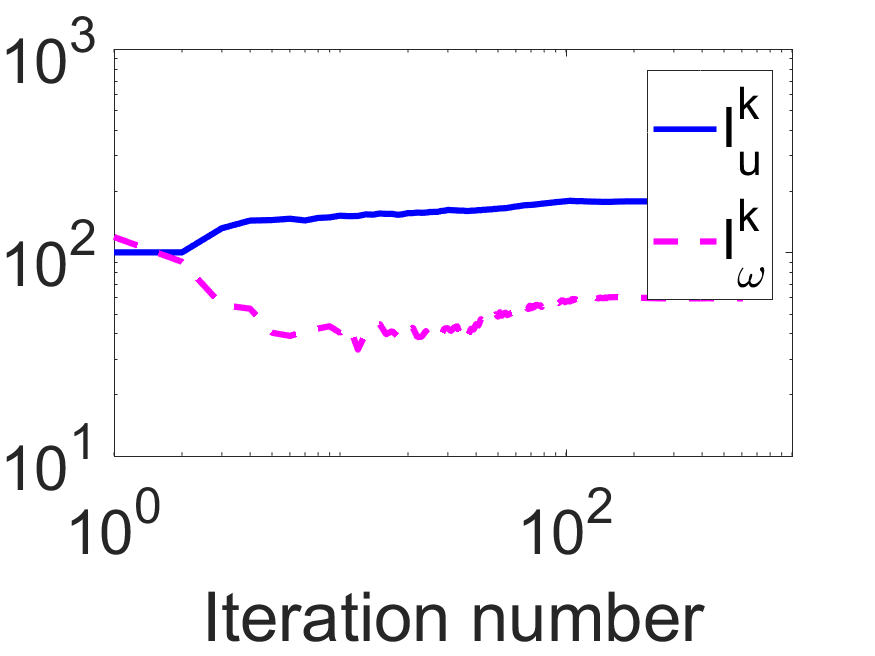}}
\subfigure[]{\includegraphics[width=.24\textwidth]{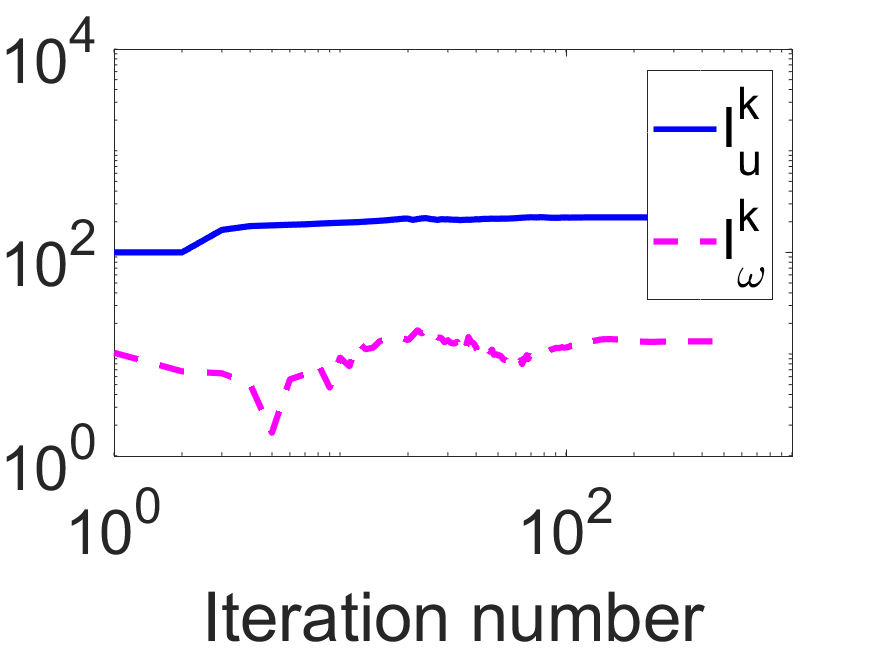}}
\subfigure[]{\includegraphics[width=.24\textwidth]{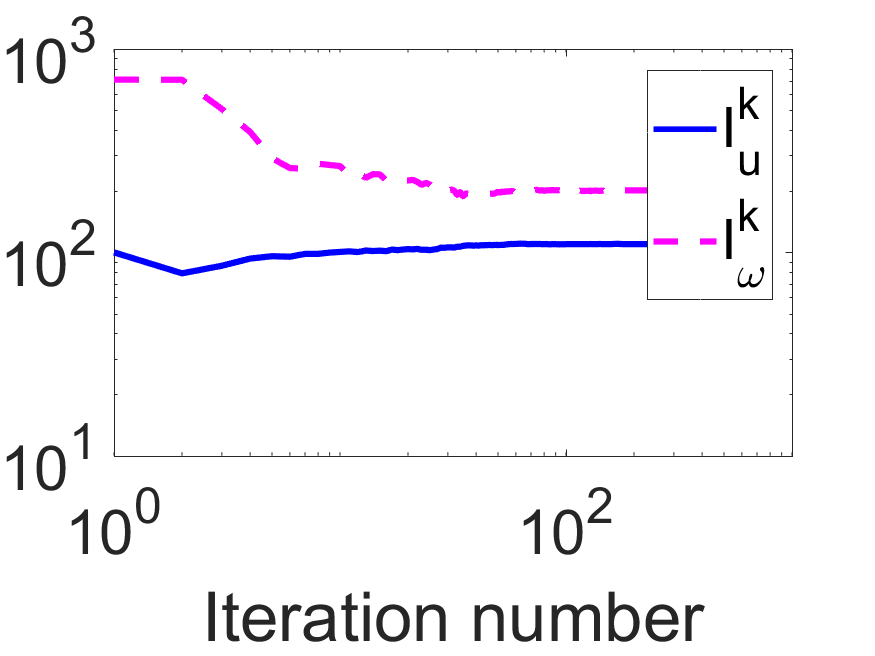}}
\subfigure[]{\includegraphics[width=.24\textwidth]{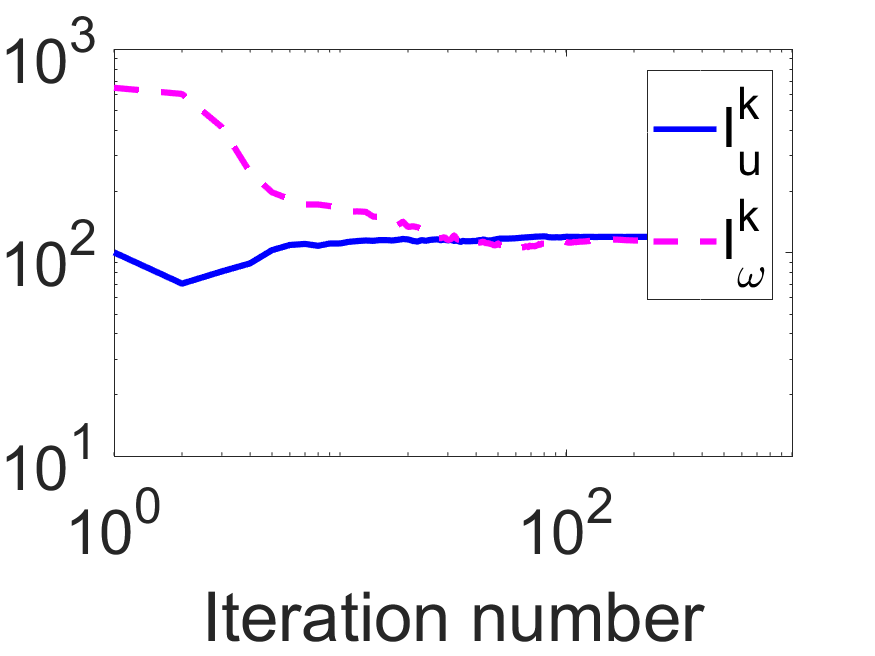}}
\end{center}
\caption{Histories of $I_u^k$ and $I_\omega^k$ for ADMM (a, b) and ADMM-Prox (c, d) scanning on square-lattice (a, c) and random-lattice (b, d). $D=24$. }
\label{fig0}
\end{figure}

\begin{figure}[]
\begin{center}
\subfigure[Truth]{\includegraphics[width=.15\textwidth]{test0_1/3-10-16uG-eps-converted-to.pdf}\includegraphics[width=.04\textwidth]{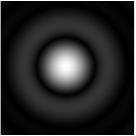}}
\hskip -.02in
\subfigure[DR \cite{thibault2009probe}]{\includegraphics[width=.15\textwidth]{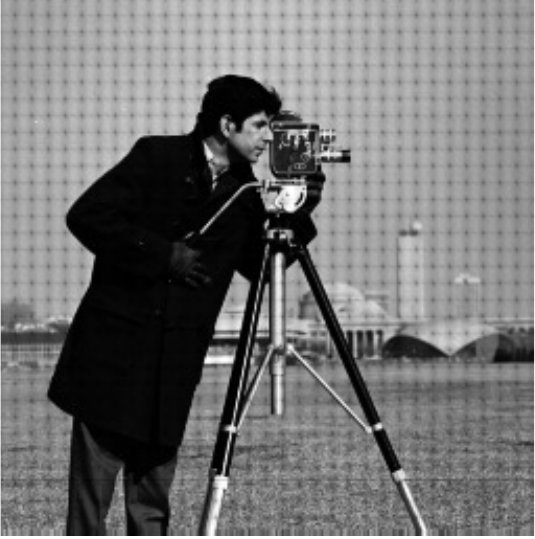}\includegraphics[width=.04\textwidth]{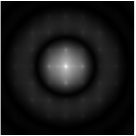}}
\hskip -.02in
\subfigure[PALM  \cite{hesse2015proximal}]{\includegraphics[width=.15\textwidth]{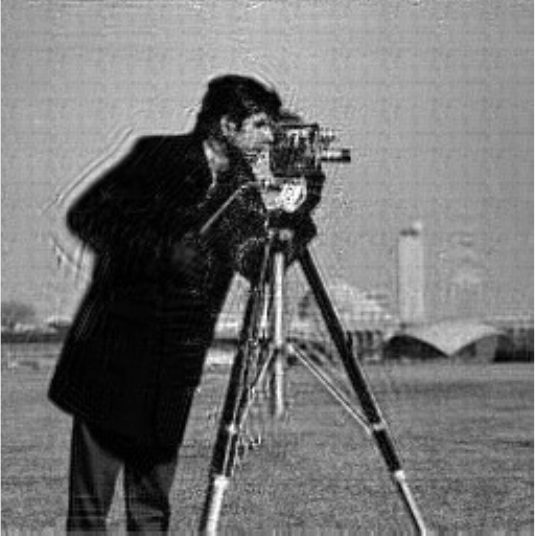}\includegraphics[width=.04\textwidth]{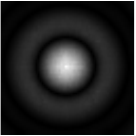}}
\hskip -.02in
\subfigure[ADMM]{\includegraphics[width=.15\textwidth]{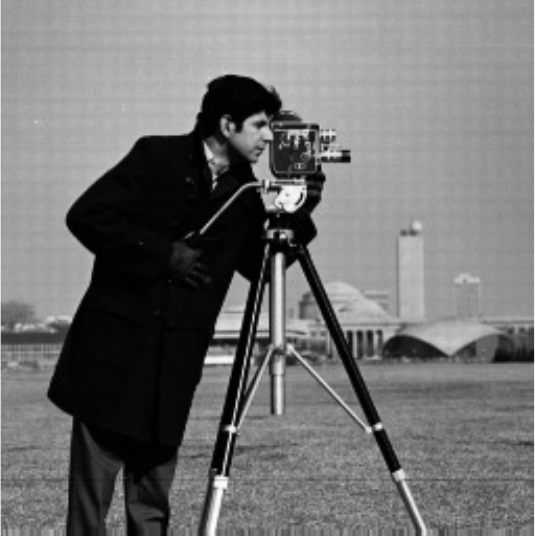}\includegraphics[width=.04\textwidth]{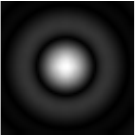}}
\hskip -.02in
\subfigure[ADMM-Prox]{\includegraphics[width=.15\textwidth]{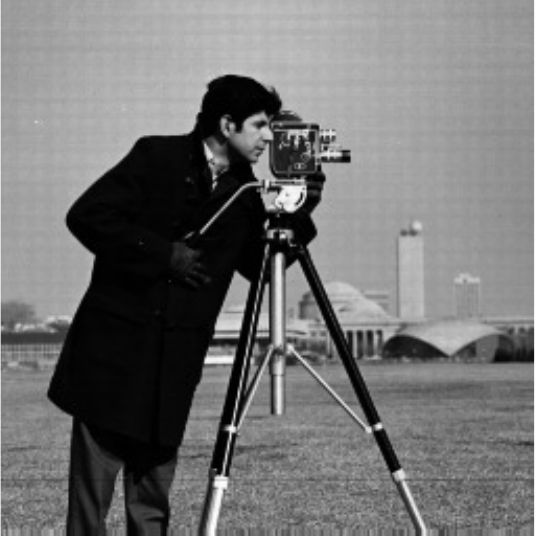}\includegraphics[width=.04\textwidth]{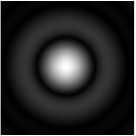}}\\
\hskip .19\textwidth
\subfigure[DR \cite{thibault2009probe}]{\includegraphics[width=.15\textwidth]{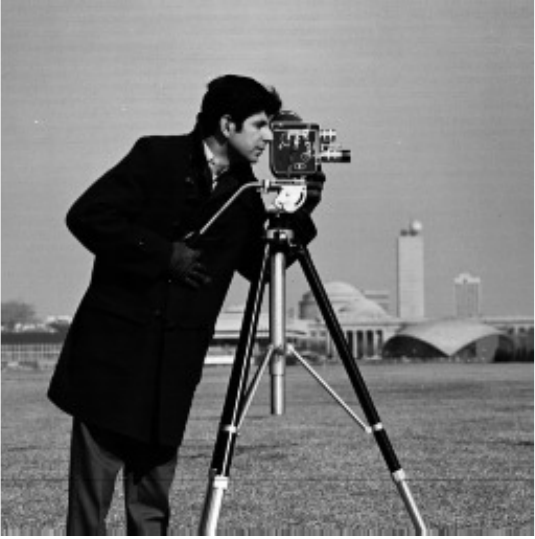}\includegraphics[width=.04\textwidth]{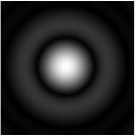}}
\hskip -.02in
\subfigure[PALM  \cite{hesse2015proximal}]{\includegraphics[width=.15\textwidth]{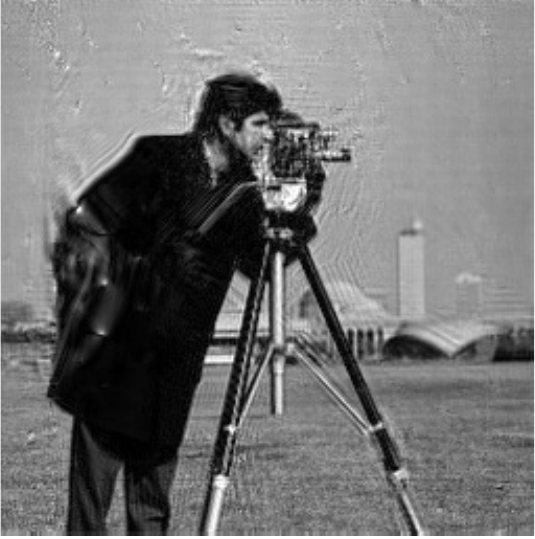}\includegraphics[width=.04\textwidth]{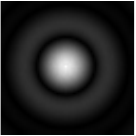}}
\hskip -.02in
\subfigure[ADMM]{\includegraphics[width=.15\textwidth]{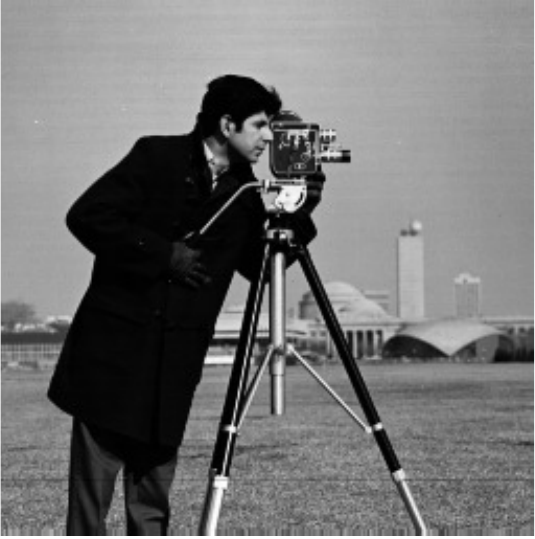}\includegraphics[width=.04\textwidth]{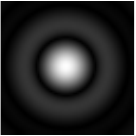}}
\hskip -.02in
\subfigure[ADMM-Prox]{\includegraphics[width=.15\textwidth]{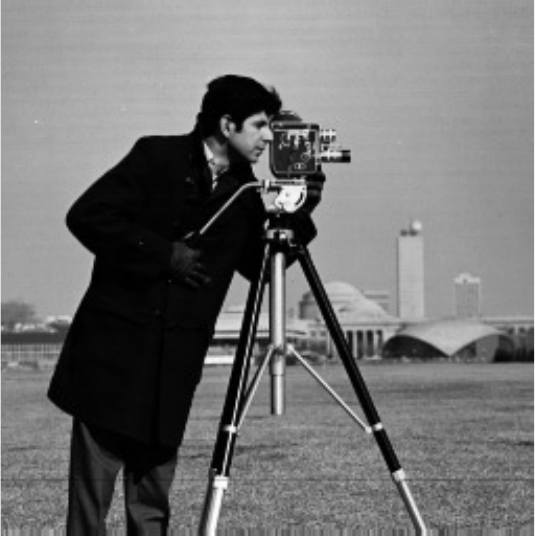}\includegraphics[width=.04\textwidth]{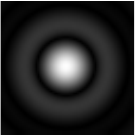}}\\
\subfigure[ePIE \cite{maiden2009improved}]{\includegraphics[width=.15\textwidth]{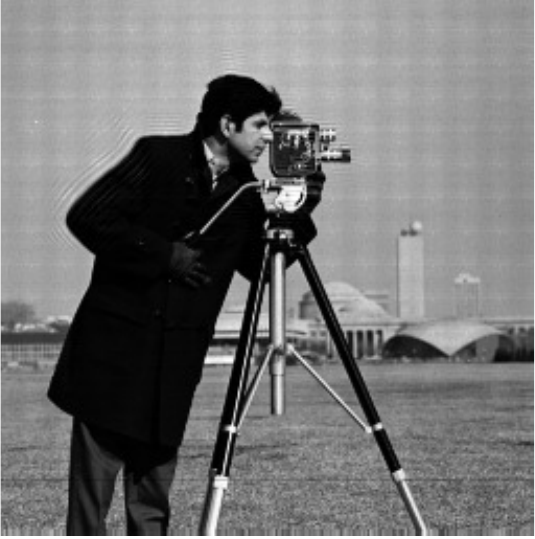}\includegraphics[width=.04\textwidth]{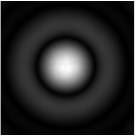}}
\hskip -.02in
\subfigure[DR \cite{thibault2009probe}]{\includegraphics[width=.15\textwidth]{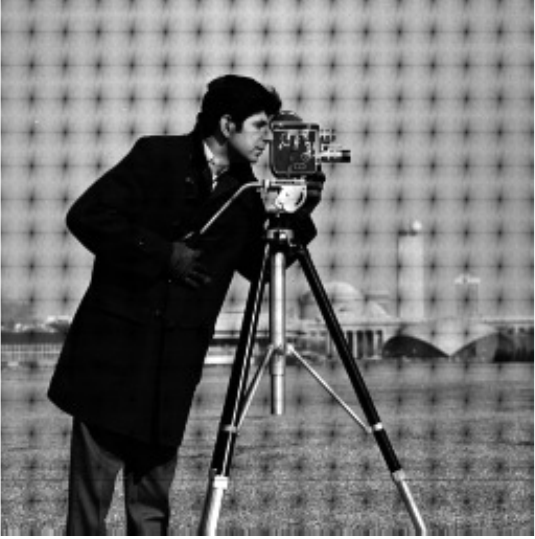}\includegraphics[width=.04\textwidth]{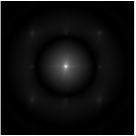}}
\hskip -.02in
\subfigure[PALM  \cite{hesse2015proximal}]{\includegraphics[width=.15\textwidth]{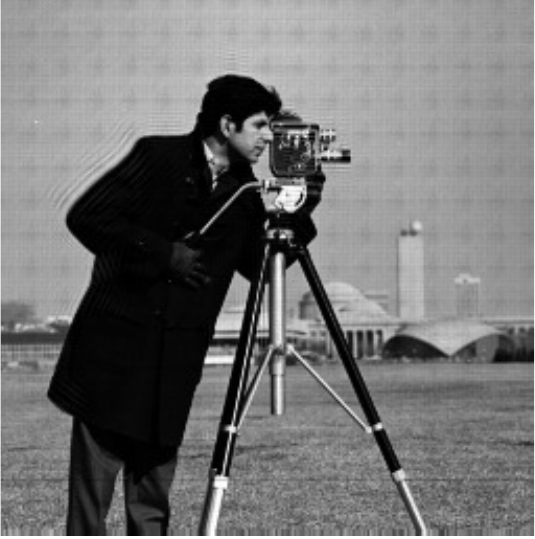}\includegraphics[width=.04\textwidth]{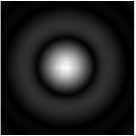}}
\hskip -.02in
\subfigure[ADMM]{\includegraphics[width=.15\textwidth]{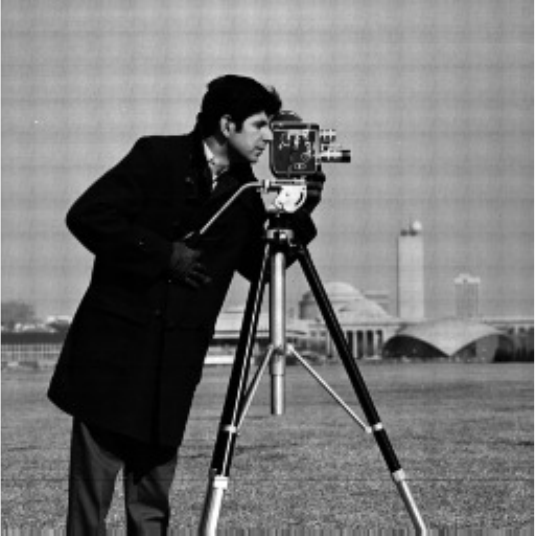}\includegraphics[width=.04\textwidth]{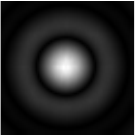}}
\hskip -.02in
\subfigure[ADMM-Prox]{\includegraphics[width=.15\textwidth]{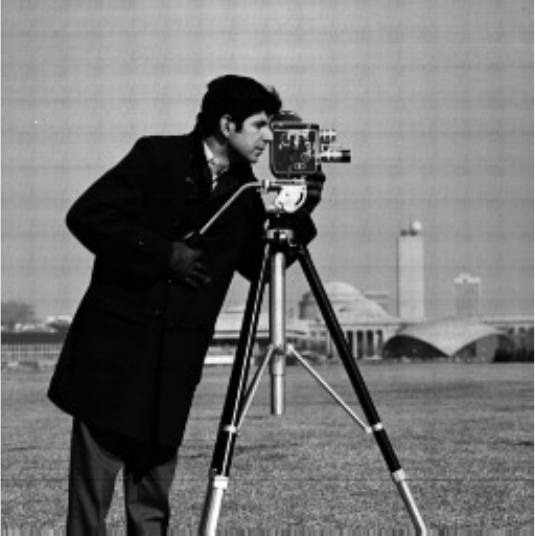}\includegraphics[width=.04\textwidth]{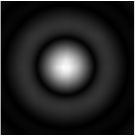}}\\
\subfigure[ePIE \cite{maiden2009improved}]{\includegraphics[width=.15\textwidth]{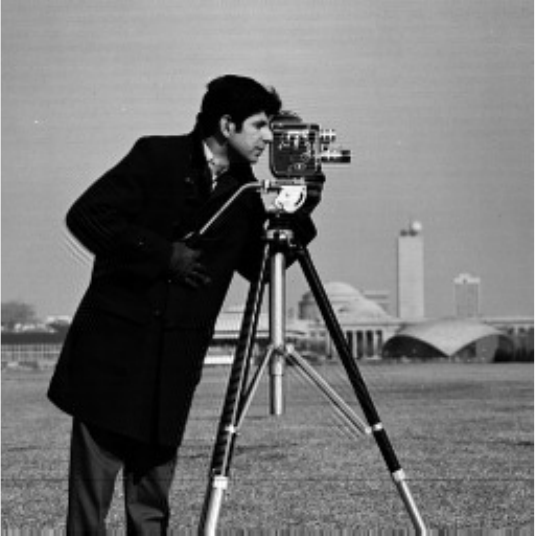}\includegraphics[width=.04\textwidth]{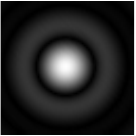}}
\hskip -.02in
\subfigure[DR \cite{thibault2009probe}]{\includegraphics[width=.15\textwidth]{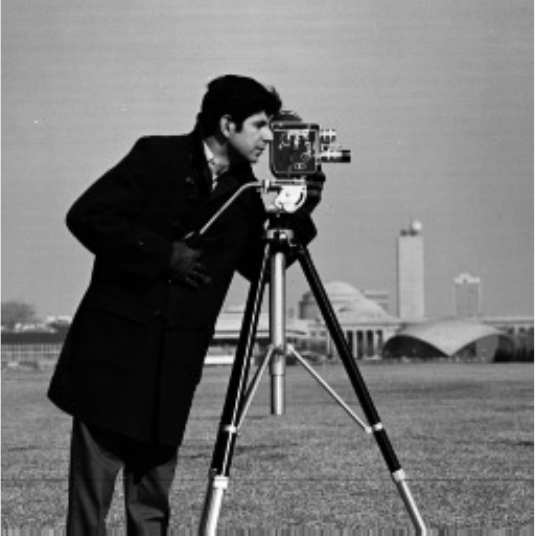}\includegraphics[width=.04\textwidth]{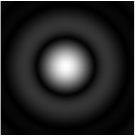}}
\hskip -.02in
\subfigure[PALM  \cite{hesse2015proximal}]{\includegraphics[width=.15\textwidth]{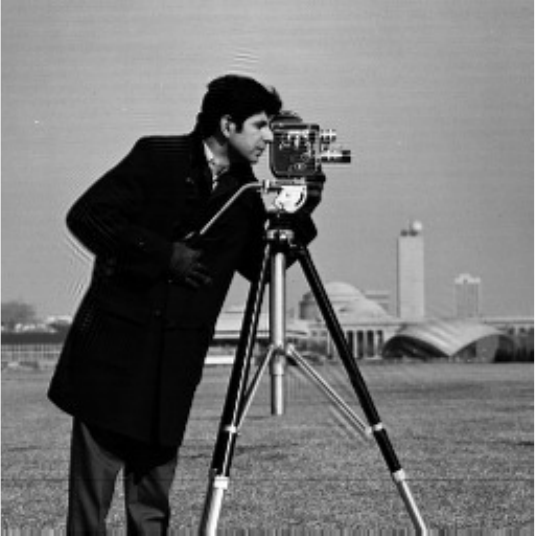}\includegraphics[width=.04\textwidth]{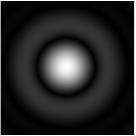}}
\hskip -.02in
\subfigure[ADMM]{\includegraphics[width=.15\textwidth]{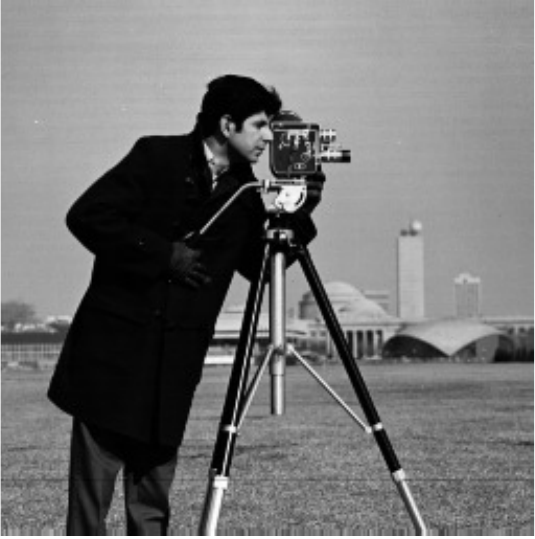}\includegraphics[width=.04\textwidth]{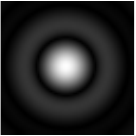}}
\hskip -.02in
\subfigure[ADMM-Prox]{\includegraphics[width=.15\textwidth]{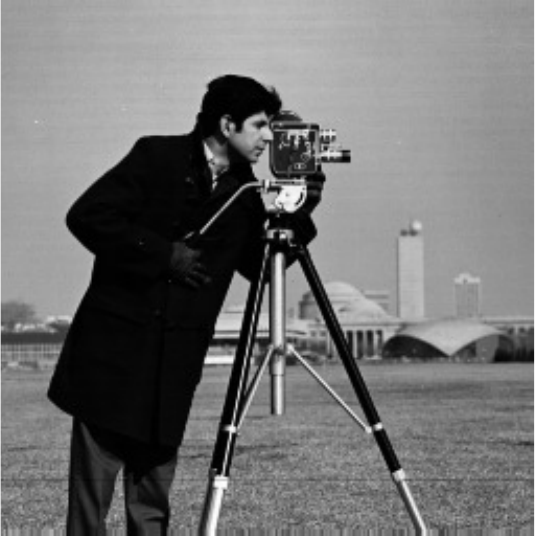}\includegraphics[width=.04\textwidth]{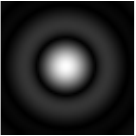}}
\end{center}
\caption{
Results  of recovered images and probes by  ePIE \cite{maiden2009improved}, DR \cite{thibault2009probe},  PALM \cite{hesse2015proximal}  and ADMM (Algorithm \ref{algADMM}) from left to right. D=24 ($m=6.25n$) in the 1st-2nd rows; D=16 ($m=16n$) in 3rd-4th rows. Scanning on  square-lattice  in 1st and 3rd rows and random-lattice in 2nd and 4th rows. (a): Ground truth for the image and probe with resolutions $n=256\times 256$ and $\bar m=64\times 64$ respectively.
}
\vskip -.2in
\label{fig3}
\end{figure}

\begin{figure}[]
\begin{center}
\subfigure[Truth]{\includegraphics[width=.13\textwidth]{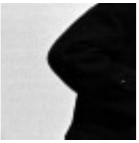}}\qquad
\subfigure[DR \cite{thibault2009probe}]{\includegraphics[width=.13\textwidth]{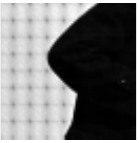}}\qquad
\subfigure[\!\!PALM\!\cite{hesse2015proximal}\!\!]{\includegraphics[width=.13\textwidth]{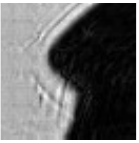}}\qquad
\subfigure[ADMM]{\includegraphics[width=.13\textwidth]{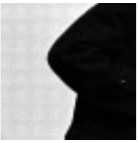}}\qquad
\subfigure[ADMM-Prox]{\includegraphics[width=.13\textwidth]{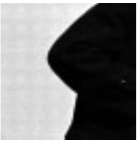}}\\
\hskip .13\textwidth\qquad
\subfigure[DR \cite{thibault2009probe}]{\includegraphics[width=.13\textwidth]{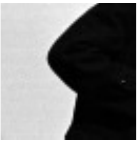}}\qquad
\subfigure[PALM \cite{hesse2015proximal}]{\includegraphics[width=.13\textwidth]{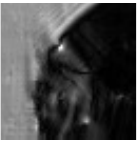}}\qquad
\subfigure[ADMM]{\includegraphics[width=.13\textwidth]{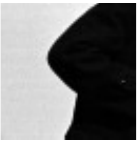}}\qquad
\subfigure[ADMM-Prox]{\includegraphics[width=.13\textwidth]{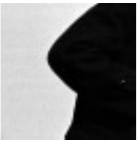}}\\
\subfigure[ePIE \cite{maiden2009improved}]{\includegraphics[width=.13\textwidth]{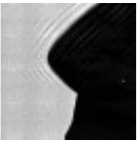}}\qquad
\subfigure[DR \cite{thibault2009probe}]{\includegraphics[width=.13\textwidth]{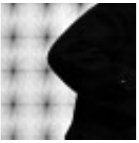}}\qquad
\subfigure[PALM \cite{hesse2015proximal}]{\includegraphics[width=.13\textwidth]{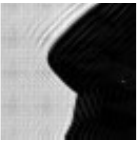}}\qquad
\subfigure[ADMM]{\includegraphics[width=.13\textwidth]{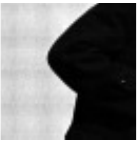}}\qquad
\subfigure[ADMM-Prox]{\includegraphics[width=.13\textwidth]{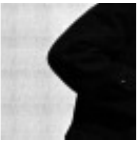}}\\
\subfigure[ePIE \cite{maiden2009improved}]{\includegraphics[width=.13\textwidth]{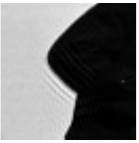}}\qquad
\subfigure[DR \cite{thibault2009probe}]{\includegraphics[width=.13\textwidth]{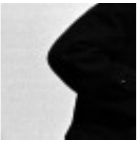}}\qquad
\subfigure[PALM \cite{hesse2015proximal}]{\includegraphics[width=.13\textwidth]{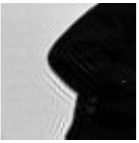}}\qquad
\subfigure[ADMM]{\includegraphics[width=.13\textwidth]{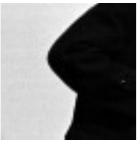}}\qquad
\subfigure[ADMM-Prox]{\includegraphics[width=.13\textwidth]{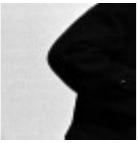}}
\end{center}
\caption{
Zoom-in views of the recovered images of Figure \ref{fig3}. D=24 in the 1st-2nd rows; D=16 in 3rd-4th rows. Scanning on  square-lattice  in 1st and 3rd rows and random-lattice in 2nd, 4th rows.
}
\label{fig66}
\end{figure}

The results of the recovered images and probes are shown in Figure  \ref{fig3}, and their corresponding zoom-in views in Figure \ref{fig66}.
The recovery results demonstrate how ePIE and PALM  are more sensitive to the number of measurements.
With $D=24$, ePIE fails due to blow-up of its iterative solutions. The recovered images using PALM (Figure \ref{fig3} (c, g), and its zoom-in Figure \ref{fig66} (c, g)) present obvious stripe structures and  ringing artifacts. With fewer number of measurements (or larger sliding distances), DR and Algorithm \ref{algADMM} work better than ePIE and PALM.
In the first and third rows of Figures~\ref{fig3} and \ref{fig66} (square-lattice scanning), periodical structural artifacts can be seen in all compared  reconstructions.
In the case of DR, such artifacts are specially severe.
In the case of ADMM, ADMM-Prox, PALM and ePIE, those artifacts can be attenuated with $D = 16$ (Figure \ref{fig3} and \ref{fig66} third row).
The results from Figures \ref{fig3} and \ref{fig66} (first row) demonstrate how both ADMM and ADMM-Prox can still attenuate these artifacts with fewer number of measurements ($D = 24$).
In the case of random-lattice scanning (Figure \ref{fig3} and \ref{fig66}, fourth row), all algorithms produce higher quality reconstructions. However, it can be noted in Figure \ref{fig66} (zoom-in views) how ADMM, ADMM-Prox and DR can recover much sharper edges than ePIE and PALM.

The convergence analysis is presented in Figure~\ref{fig4}.
The R-factor in the case of ADMM and ADMM-Prox decreases faster than with the other compared algorithms \emph{w.r.t.} the number of iterations.
Furthermore, within one thousand iterations, the proposed algorithms reach the given tolerance. Those results experimentally demonstrate the convergence of the proposed algorithms, which is consistent with the previous convergence analysis.
Specifically, inferred from the first row of Figure \ref{fig4}, ADMM and ADMM-Prox converge the fastest among all the compared algorithms: neither ePIE, PALM nor DR reach the desired tolerance within one thousand iterations.
 Due to the property of  fast convergence of ADMM and ADMM-Prox, the SNRs of recovered images and probes  are much higher than those by other algorithms. {Execution times for the compared algorithms can be found in Table  \ref{tab0}.  Due to the additional update of the multipliers,  the single step complexity of  ADMM and ADMM-Prox  is slightly higher than the compared algorithms.
Because the proposed ADMM algorithms require fewer iterations, they achieve considerable lower execution times compared with DR, ePIE and PALM. ADMM and ADMM-Prox  are 1.7 and 2 times faster, on average, for square- and random-lattice scanning, respectively.
 }

{
The recovery images of ADMM-Prox with a careful selection of the preconditioning matrices have higher quality than those of standard ADMM: about 0.6 and 2.0 dB increase of SNRs from the images in Figure \ref{fig4} (e, g) is gained. For the random-lattice scanning, they have very similar performances. In the following results, we run  proposed algorithms simply setting the preconditioning matrices to zeros. Note that ADMM-Prox require more parameters to be tuned to gain the best performance but ADMM produces comparable results by tuning a single parameter.
}

\begin{figure}[]
\begin{center}
\subfigure[]{\includegraphics[width=.24\textwidth]{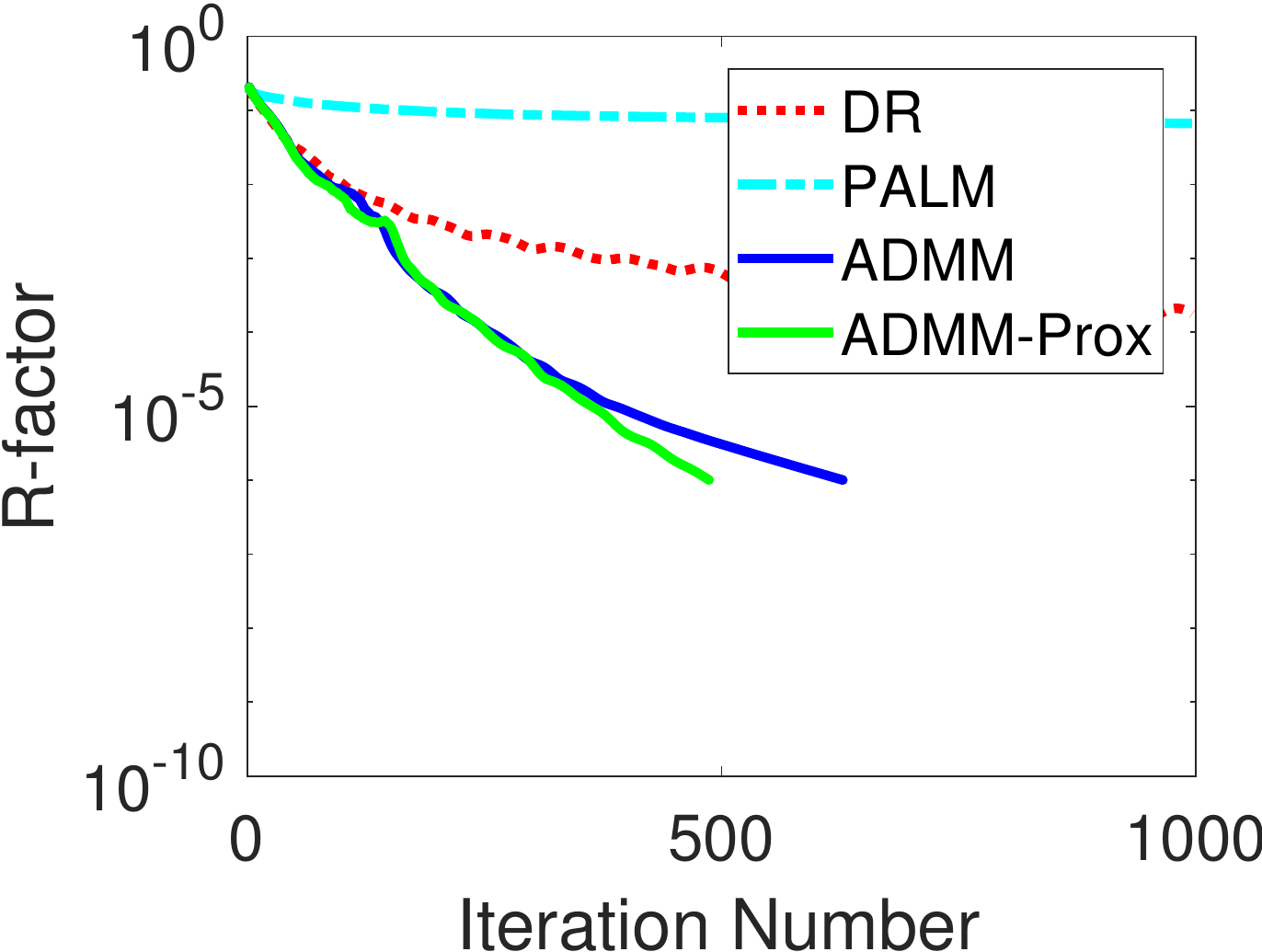}}
\subfigure[]{\includegraphics[width=.24\textwidth]{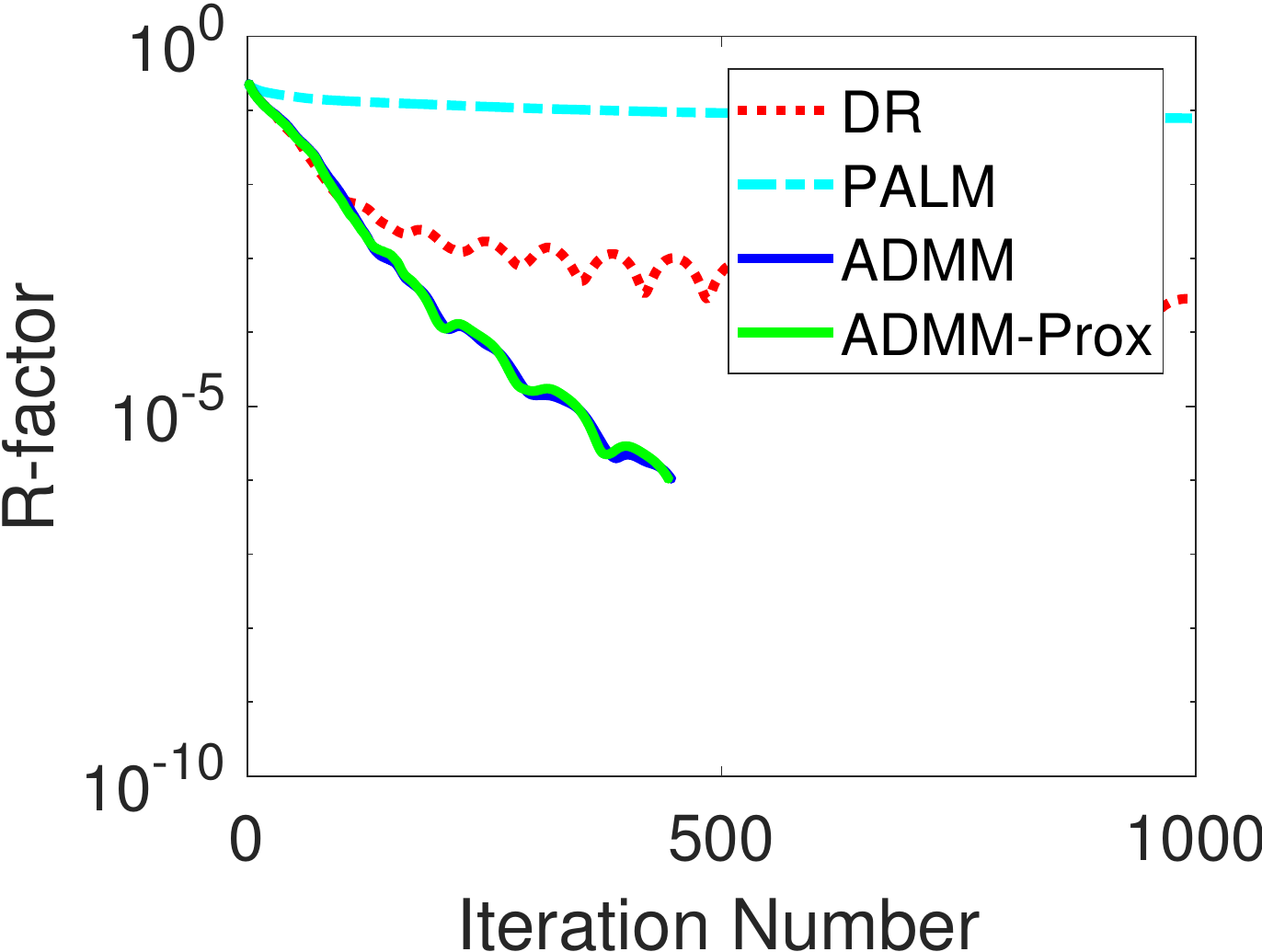}}
\subfigure[]{\includegraphics[width=.24\textwidth]{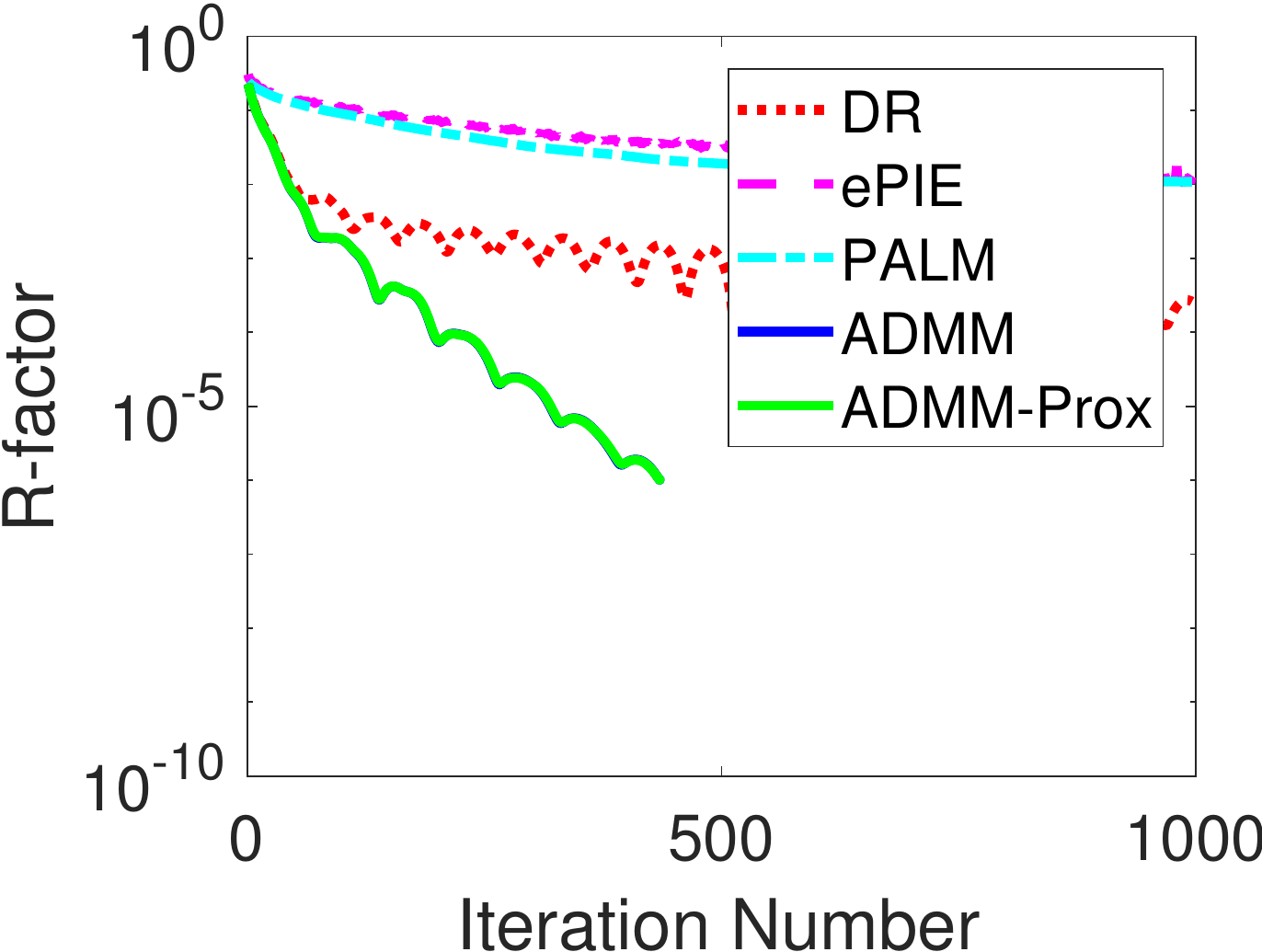}}
\subfigure[]{\includegraphics[width=.24\textwidth]{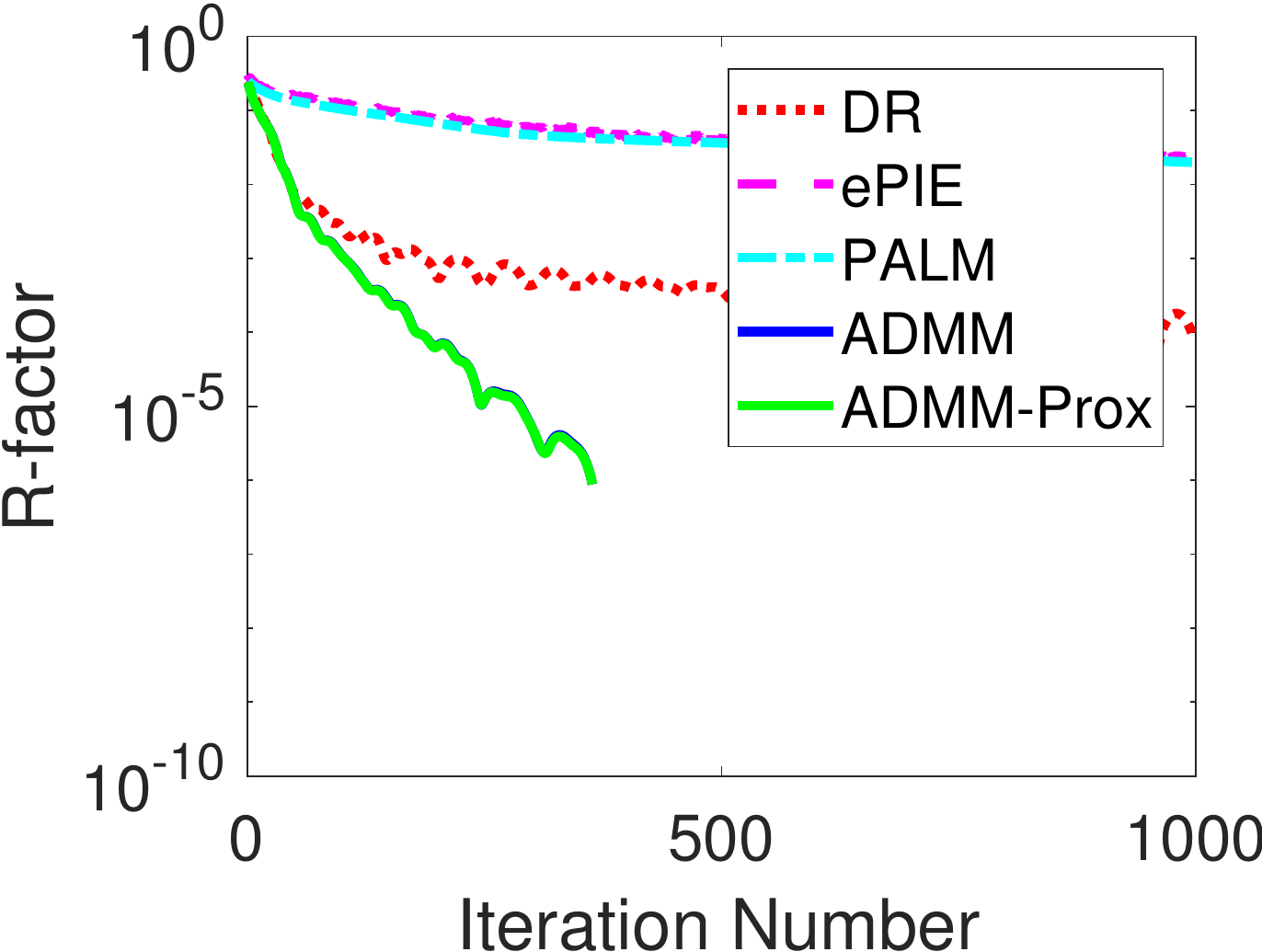}}
\subfigure[]{\includegraphics[width=.24\textwidth]{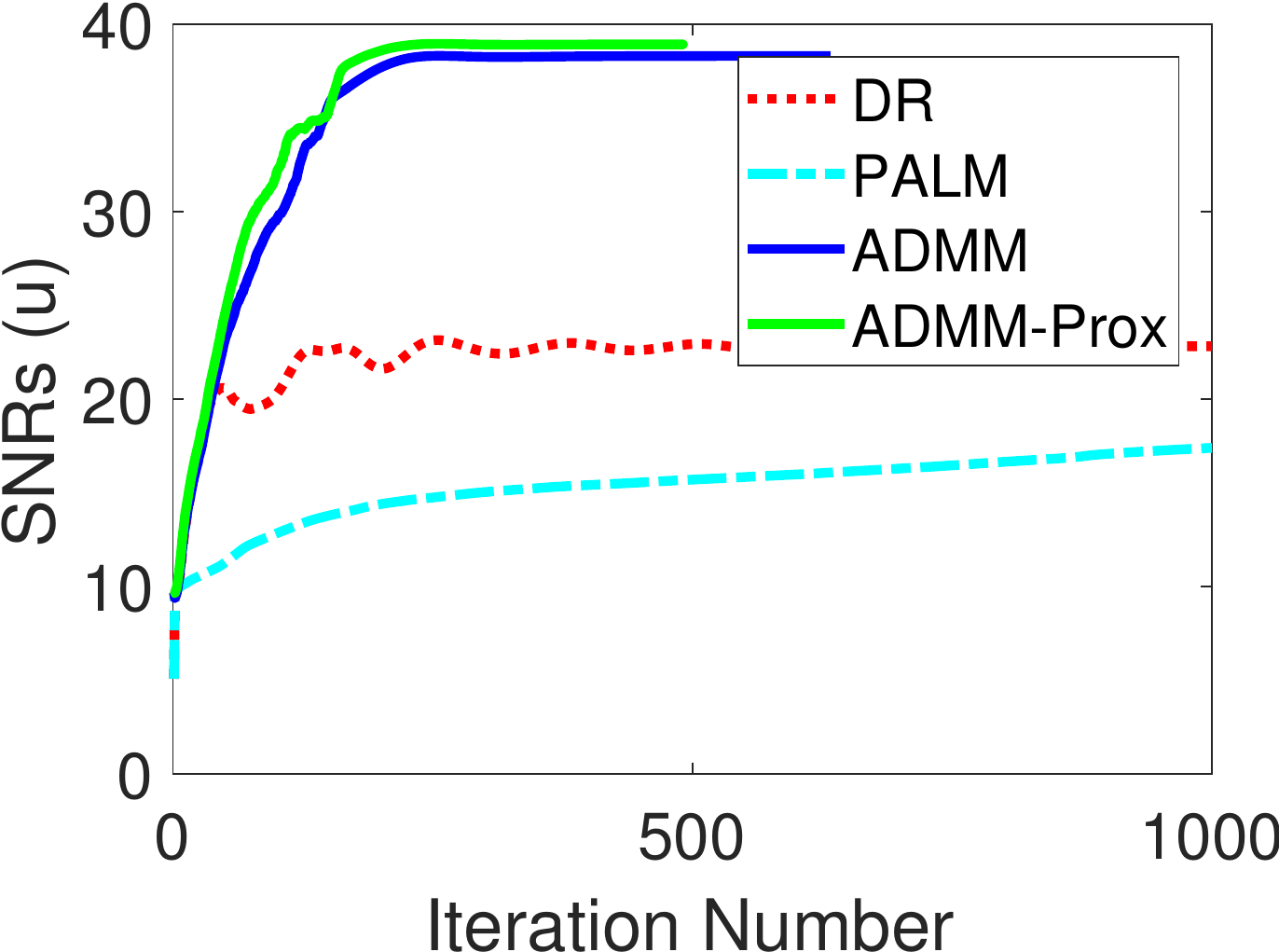}}
\subfigure[]{\includegraphics[width=.24\textwidth]{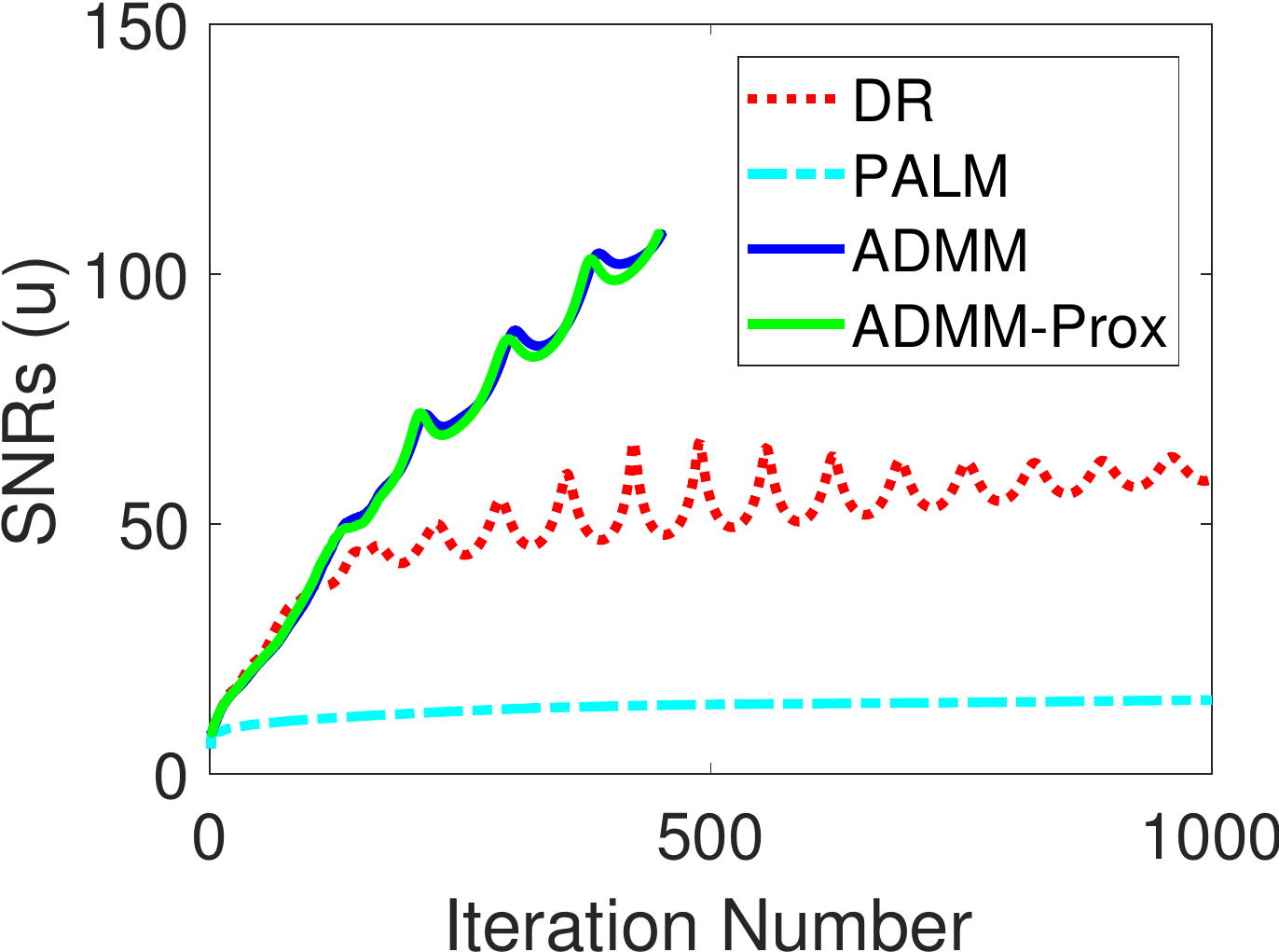}}
\subfigure[]{\includegraphics[width=.24\textwidth]{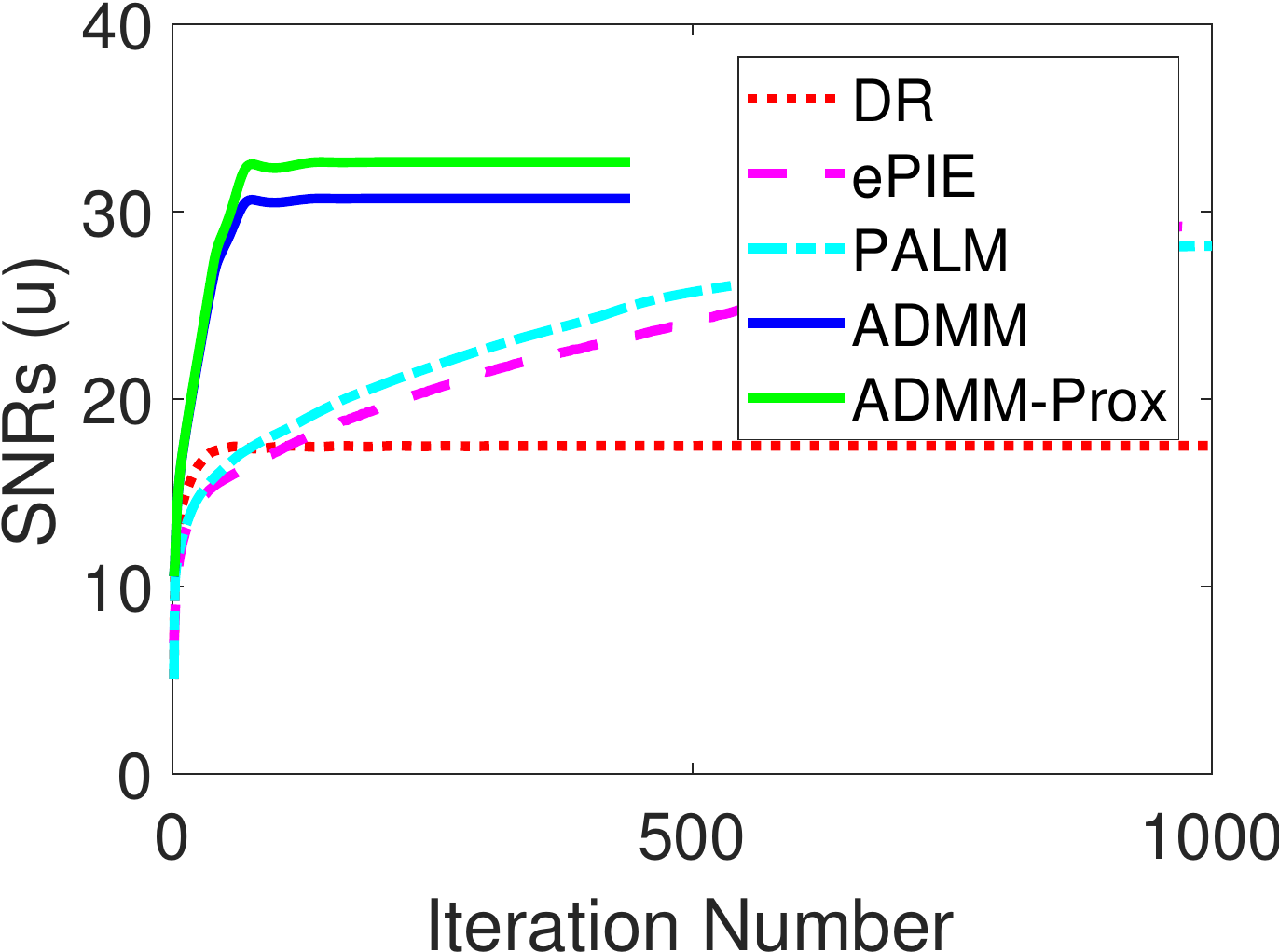}}
\subfigure[]{\includegraphics[width=.24\textwidth]{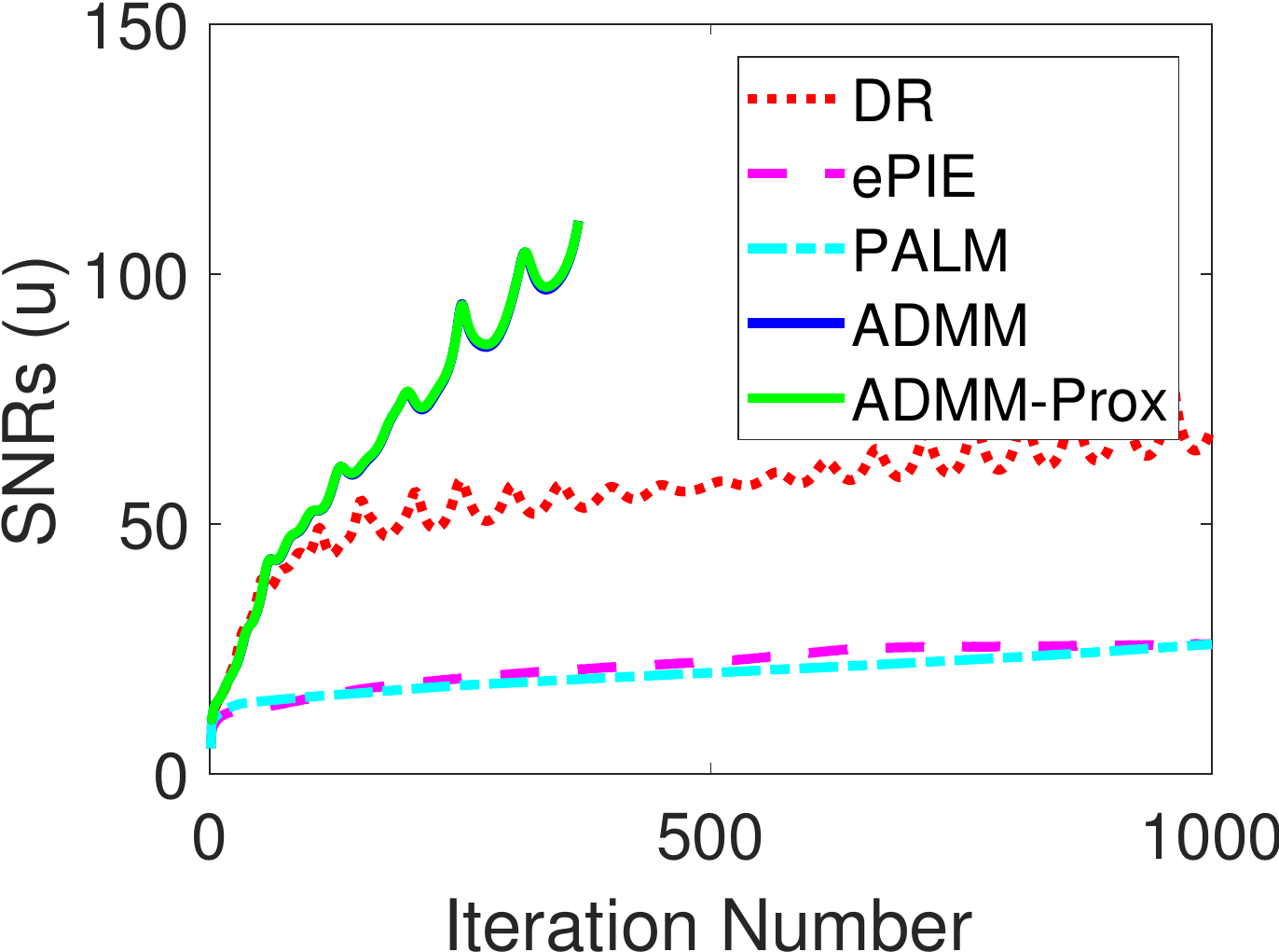}}
\subfigure[]{\includegraphics[width=.24\textwidth]{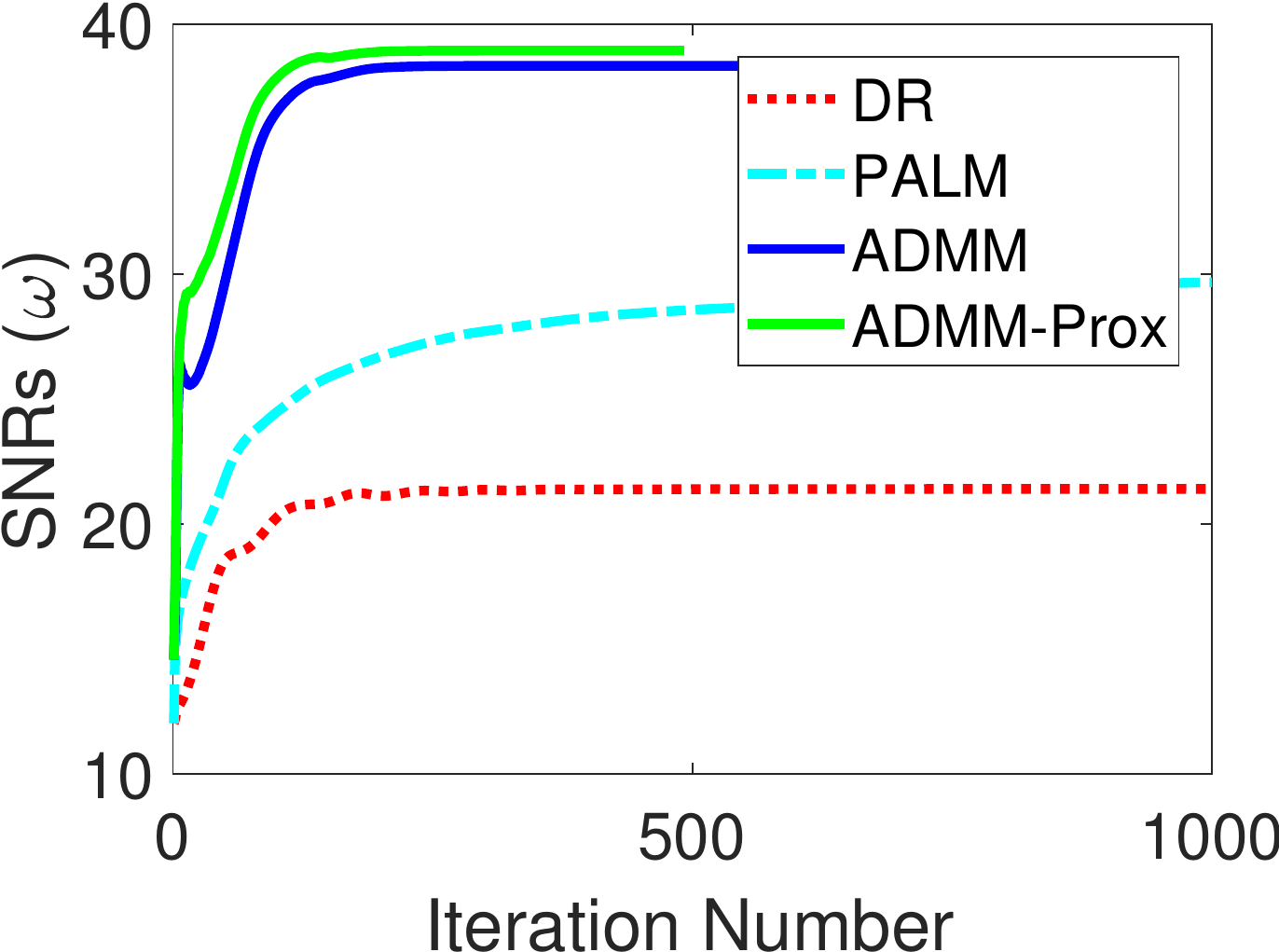}}
\subfigure[]{\includegraphics[width=.24\textwidth]{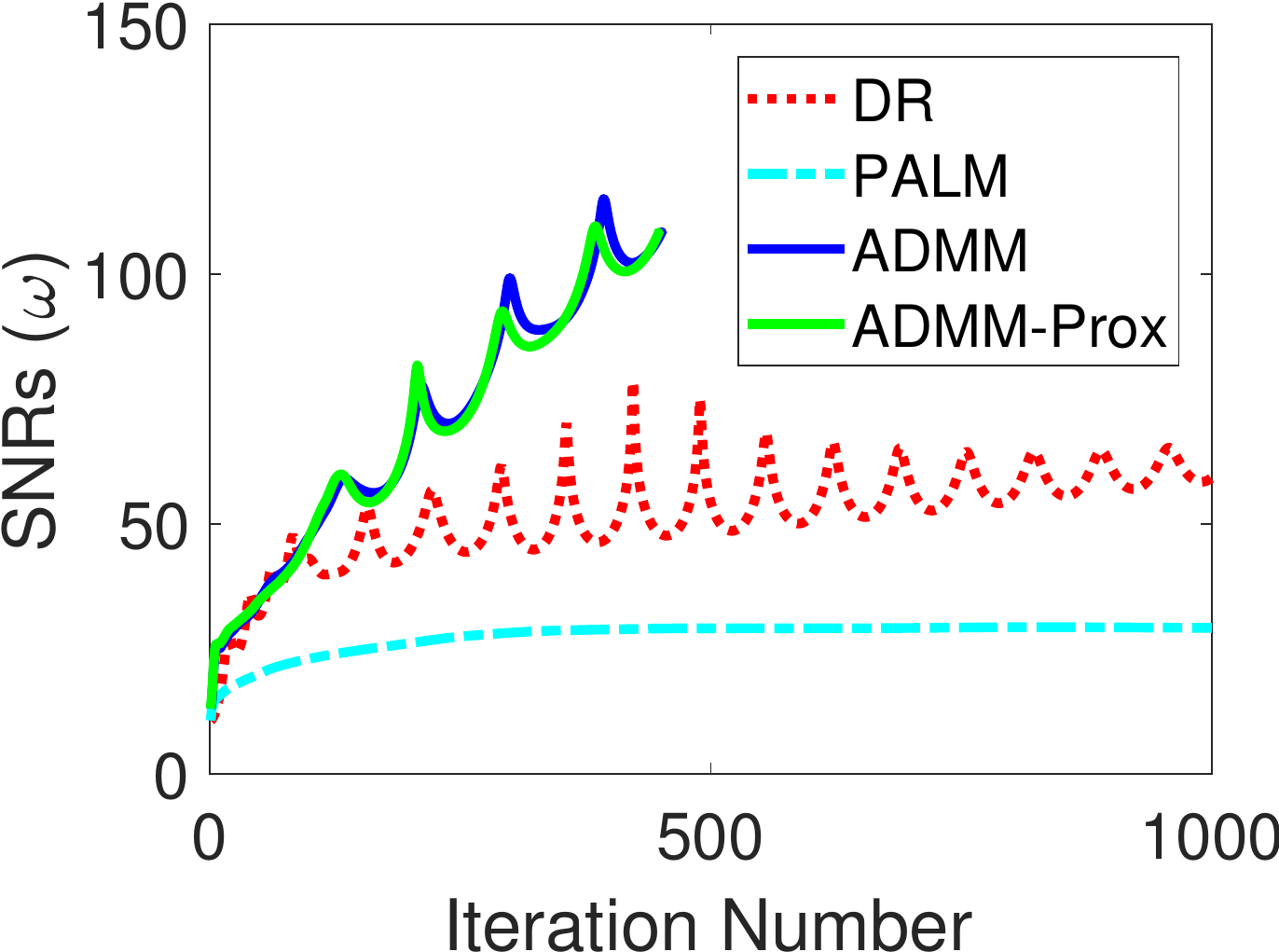}}
\subfigure[]{\includegraphics[width=.24\textwidth]{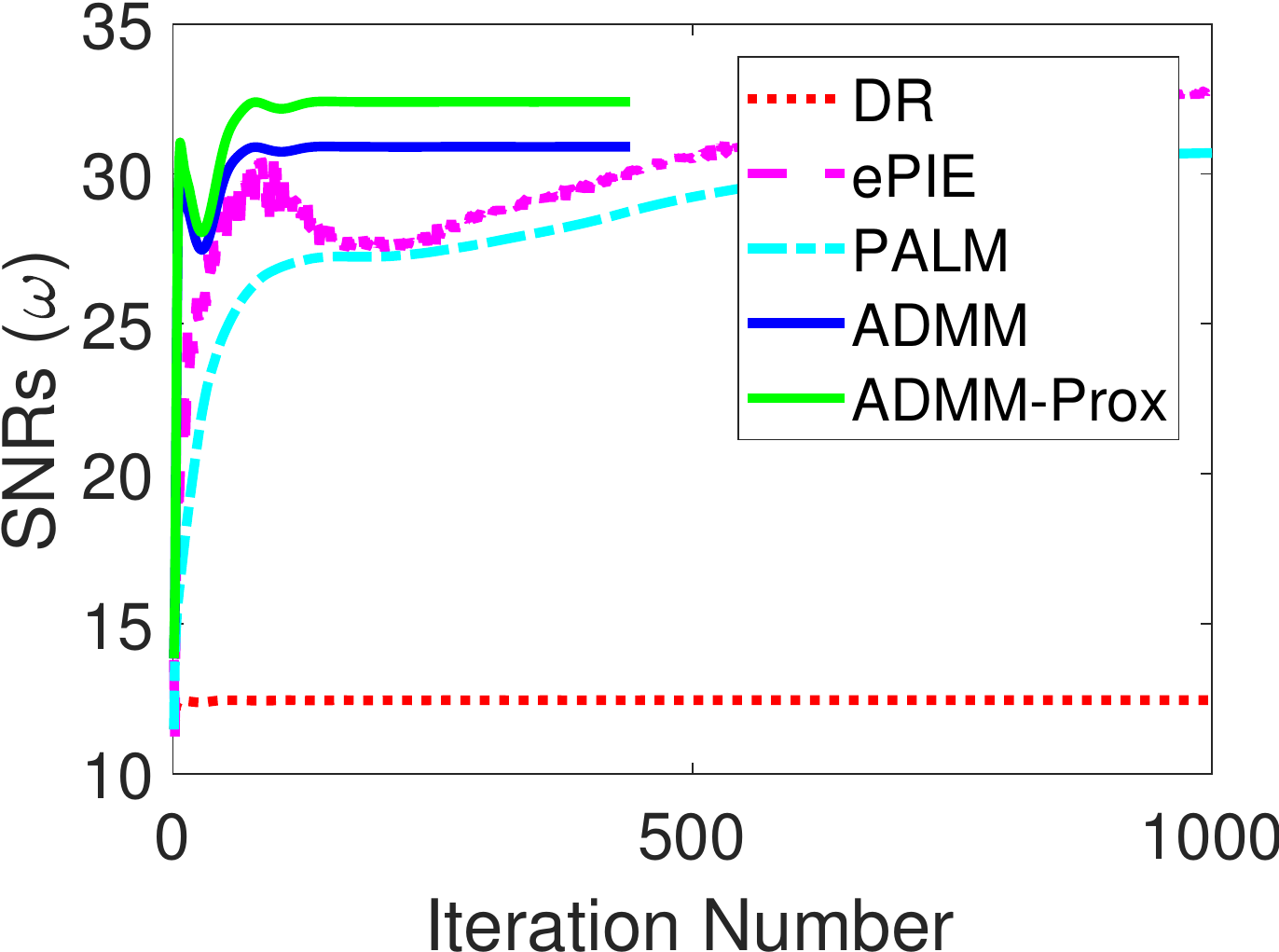}}
\subfigure[]{\includegraphics[width=.24\textwidth]{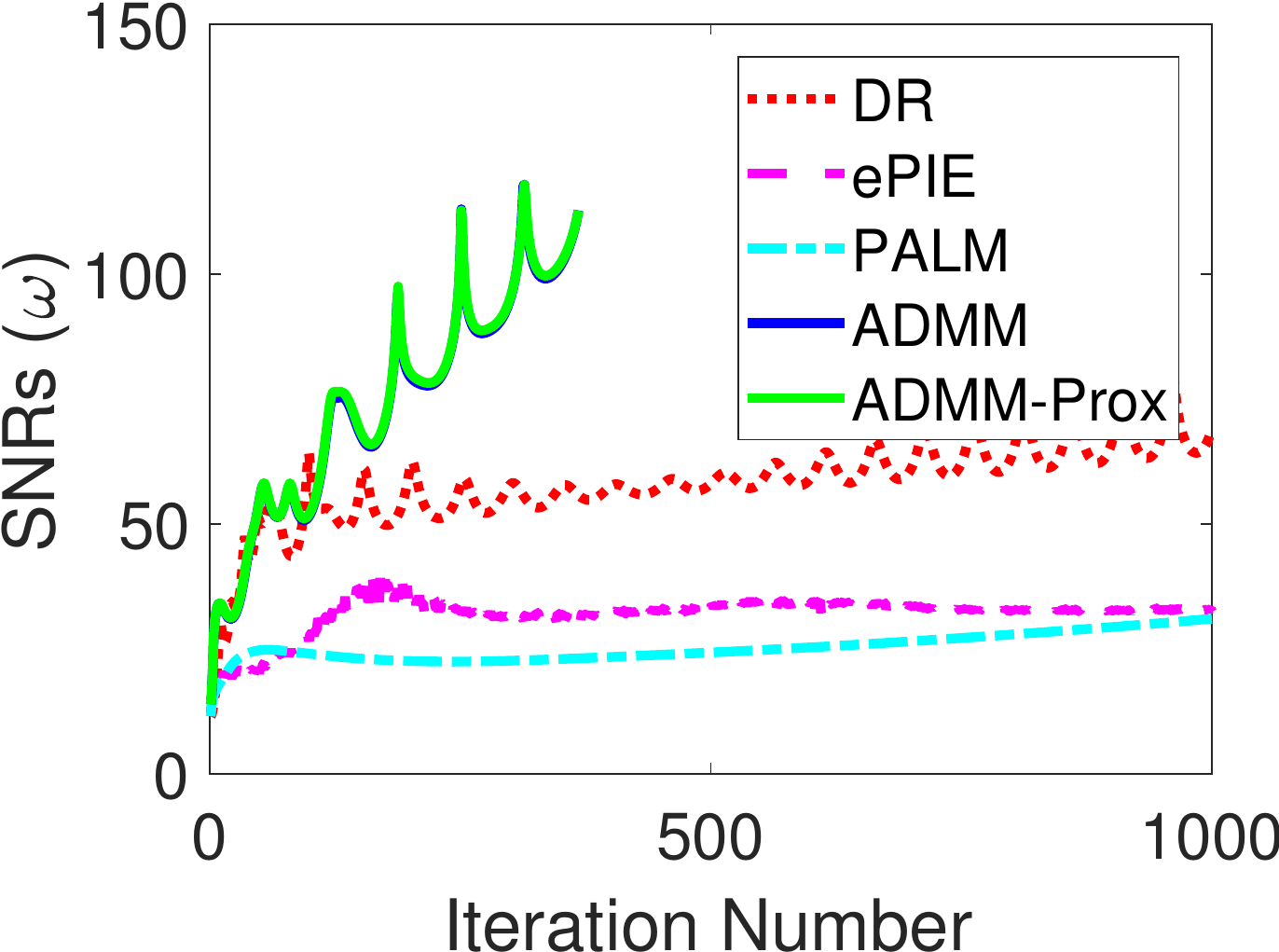}}
\end{center}
\caption{ Convergence histories of R-factor, SNRs of  images and probes from up to down, respectively by ePIE \cite{maiden2009improved}, DR \cite{thibault2009probe},  PALM \cite{hesse2015proximal} and ADMM (Algorithm \ref{algADMM}). D=24 in the 1st-2nd columns; D=16 in 3rd-4th columns. Scanning on  square-lattice  in 1st and 3rd columns and random-lattice in 2nd and 4th columns. All compared algorithms stop if  $\mathrm{R-factor}^k\leq 1.0\times 10^{-6}$ or iteration number reaches $1000.$ }
\label{fig4}
\vskip -.1in
\end{figure}

\begin{table}[h!]
\caption{{Runtimes of all compared algorithms.} The format of the data ``$\cdot/\cdot(\cdot)$'' denotes ``Elapsed time (in seconds)/Number of iterations (elapsed time per iteration)''. ``-/-'' means that the iterative solutions blow up. Note that the iteration number of ePIE means the number of cycles.}
\begin{center}
\scalebox{0.6}{
\begin{tabular}{|c|c||c|c|c|c|c|}
\hline
\multicolumn{2}{|c||}{Cases}&DR &ePIE &PALM &ADMM &ADMM-Prox\\
\hline
\multirow{2}{*}{Square-lattice}&D=24&170/1000(0.17) &-/- &115/1000(0.12) &93/633(0.15)&72/492(0.15) \\
&D=16&218/1000(0.22) &200/1000(0.20) &168/1000(0.17) & 124/444(0.28)&124/444(0.28)  \\
\hline
\multirow{2}{*}{Random-lattice}
&D=24&150/1000(0.15)&-/-&108/1000(0.11)& 65/452(0.14)&65/449(0.14)\\
&D=16&219/1000(0.22)&200/1000(0.20)&169/1000(0.17)& 102/368(0.28)&102/368(0.28)\\
\hline
\end{tabular}
}
\end{center}
\label{tab0}
\end{table}
\textbf{ ii) Noisy case.}
Poisson noisy measurements are considered with peak levels $\eta=0.1$ and $1$,  $D=16$, and random-lattice scanning. {$\mathrm{SNR}_{intensity}=38.33, 58.23$ for $\eta=0.1, 1$ respectively.}  Since the iterative solutions of ePIE blow up,  we only show results for ADMM, ADMM-Prox, DR and PALM.
The reconstruction results can be found in Figure~\ref{fig5}.

One should notice the obvious drift existing in the probe by DR in Figure \ref{fig5} (a), with peak level $\eta=0.1$. For this reason, the quality  of the recovered images of DR is the lowest among all compared algorithms, since  DR does not have fixed points for an inconsistent feasibility problem \cite{bauschke2004finding}. 
  At this noise level,  ADMM and PALM can  produce visually acceptable results. When the noise gets weak (peak level $\eta=1$) all algorithms can achieve high quality reconstructions, Figure~\ref{fig5} (second row).
 The zoom-in views of this experiment are presented in Figure~\ref{fig6}.
In this view we can see how the proposed ADMM algorithm, Figure \ref{fig6} (c), presents much weaker ringing artifacts compared to PALM and DR, Figure \ref{fig6} (a, b). At peak level $\eta=1,$ the recovered image from ADMM is nearly artifact-free, Figure \ref{fig6} (g), while the results by the other two algorithms contain evident ringing artifacts, Figure \ref{fig6} (e, f).

Convergence histories for this experiment can be found in Figure~\ref{fig7}.
The results of Figure~\ref{fig7} (a, c) demonstrate how DR is not stable with strong noise ($\eta=0.1$). The R-factors of ADMM are smaller than those obtained by PALM and DR, and the SNRs of the recovery results of ADMM are also higher than those obtained by PALM, which is consistent with the results shown in Figures~\ref{fig5} and \ref{fig6}. In summary, ADMM and PALM are more stable than DR and ePIE, and  ADMM can recover images with higher quality than all the other compared algorithms.

\begin{figure}[h!]
\hskip .15\textwidth DR \cite{thibault2009probe} \hskip .14\textwidth PALM \cite{hesse2015proximal}  \hskip .2\textwidth ADMM
\begin{center}
\subfigure[]{\includegraphics[width=.25\textwidth]{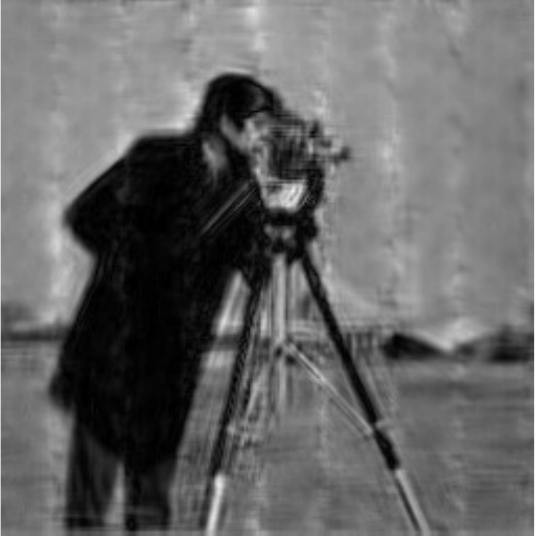}\includegraphics[width=.04\textwidth]{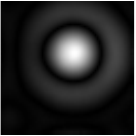}}
\subfigure[]{\includegraphics[width=.25\textwidth]{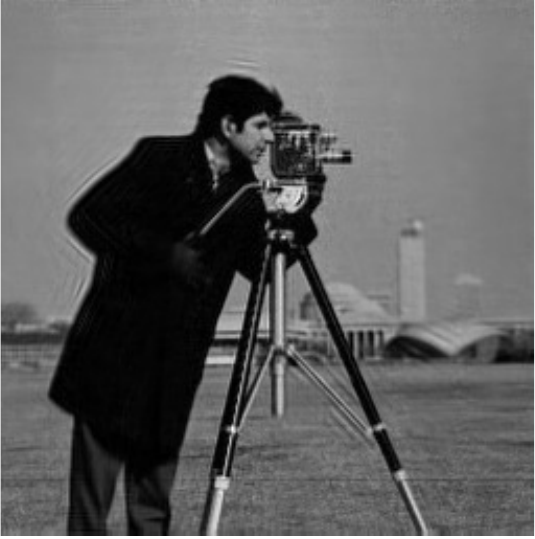}\includegraphics[width=.04\textwidth]{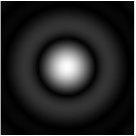}}
\subfigure[]{\includegraphics[width=.25\textwidth]{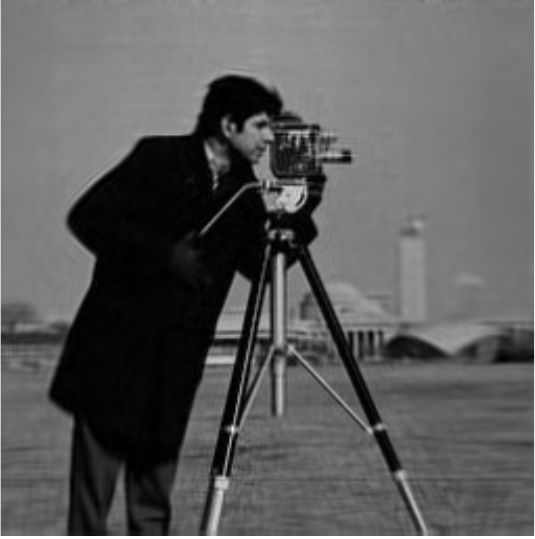}\includegraphics[width=.04\textwidth]{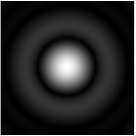}}\\
\subfigure[]{\includegraphics[width=.25\textwidth]{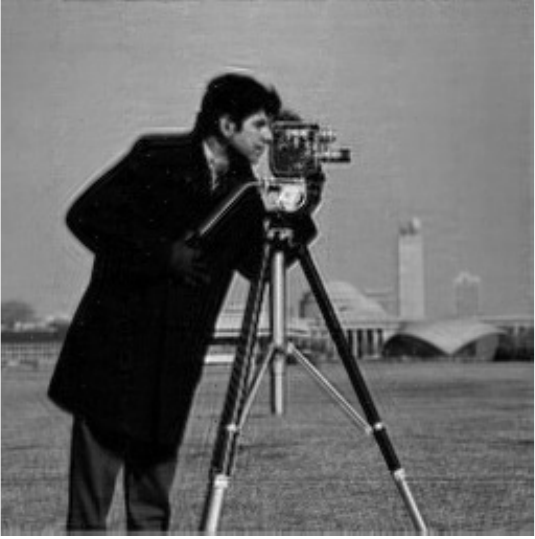}\includegraphics[width=.04\textwidth]{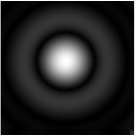}}
\subfigure[]{\includegraphics[width=.25\textwidth]{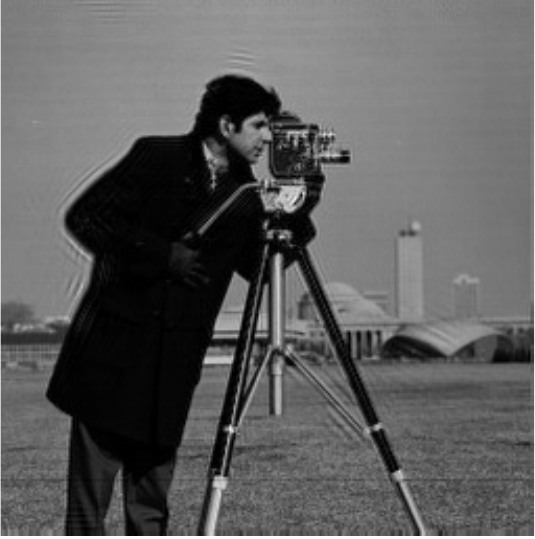}\includegraphics[width=.04\textwidth]{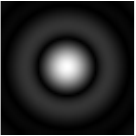}}
\subfigure[]{\includegraphics[width=.25\textwidth]{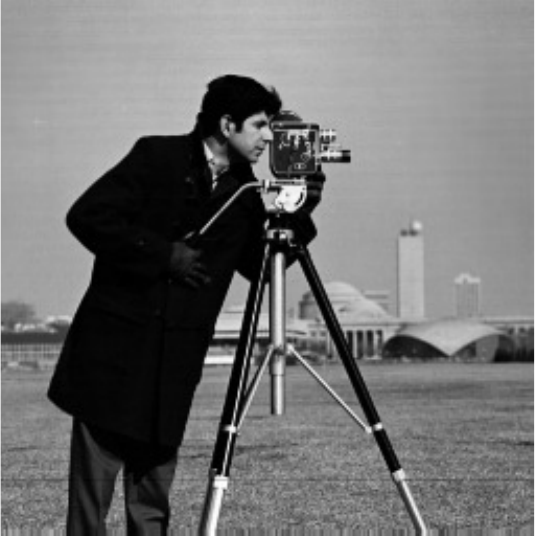}\includegraphics[width=.04\textwidth]{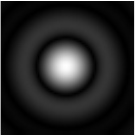}}
\end{center}
\caption{
Results of DR \cite{thibault2009probe}, PALM \cite{hesse2015proximal}  and ADMM (Algorithm \ref{algADMM}) with random-lattice scanning for Poisson noisy data. Peak level $\eta=0.1$ (top), $\eta=1$ (bottom). $\beta=0.5, 0.2$ for Algorithm \ref{algADMM} with $\eta=0.1, 1$, respectively.
All compared algorithms stop when the iteration number reaches $Iter_{Max}=300$.
}
\label{fig5}
\end{figure}

\begin{figure}[]
\begin{center}
\subfigure[DR \cite{thibault2009probe}]{\includegraphics[width=.135\textwidth]{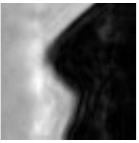}}
\subfigure[\hskip -.05in PALM\cite{hesse2015proximal}]{\includegraphics[width=.135\textwidth]{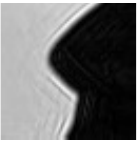}}
\subfigure[ADMM]{\includegraphics[width=.135\textwidth]{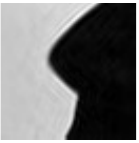}}
\subfigure[Truth]{\includegraphics[width=.135\textwidth]{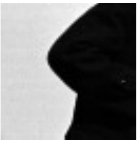}}
\subfigure[DR \cite{thibault2009probe}]{\includegraphics[width=.135\textwidth]{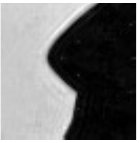}}
\subfigure[\hskip -.05in PALM\cite{hesse2015proximal}]{\includegraphics[width=.135\textwidth]{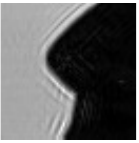}}
\subfigure[ADMM]{\includegraphics[width=.135\textwidth]{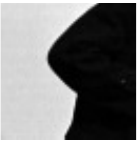}}
\end{center}
\caption{
Zoom-in views of Figure \ref{fig5}  with random-lattice scanning. (d): zoom-in views of ground truth.
(a, c), (e, g) are generated by magnifying Figure \ref{fig5} (a, c), (d, f), respectively.
(a, c) with $\eta=0.1$; (e, g) with $\eta=1.$
}
\label{fig6}
\end{figure}

\begin{figure}[]
\begin{center}
\subfigure[]{\includegraphics[width=.25\textwidth]{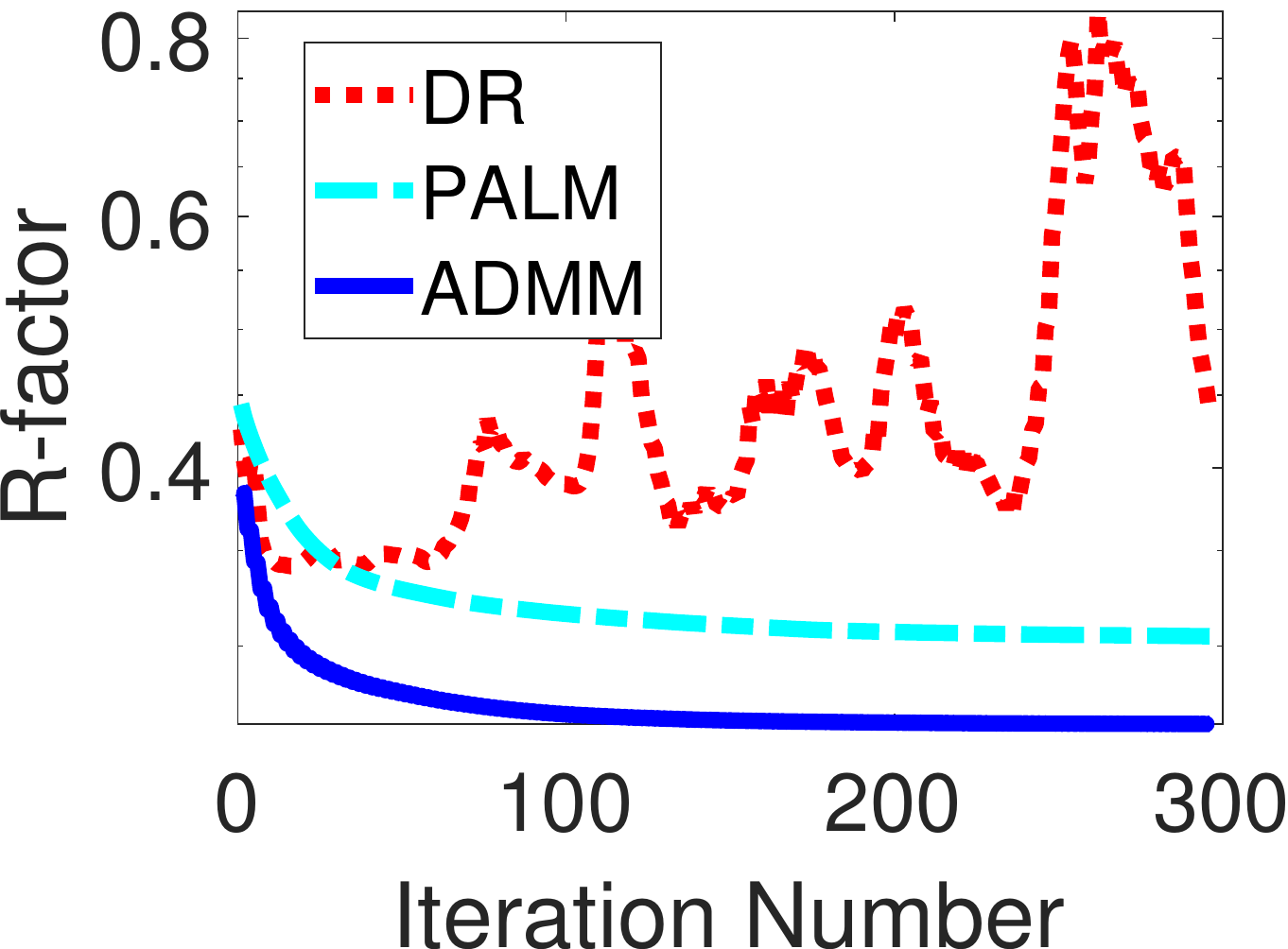}}
\subfigure[]{\includegraphics[width=.25\textwidth]{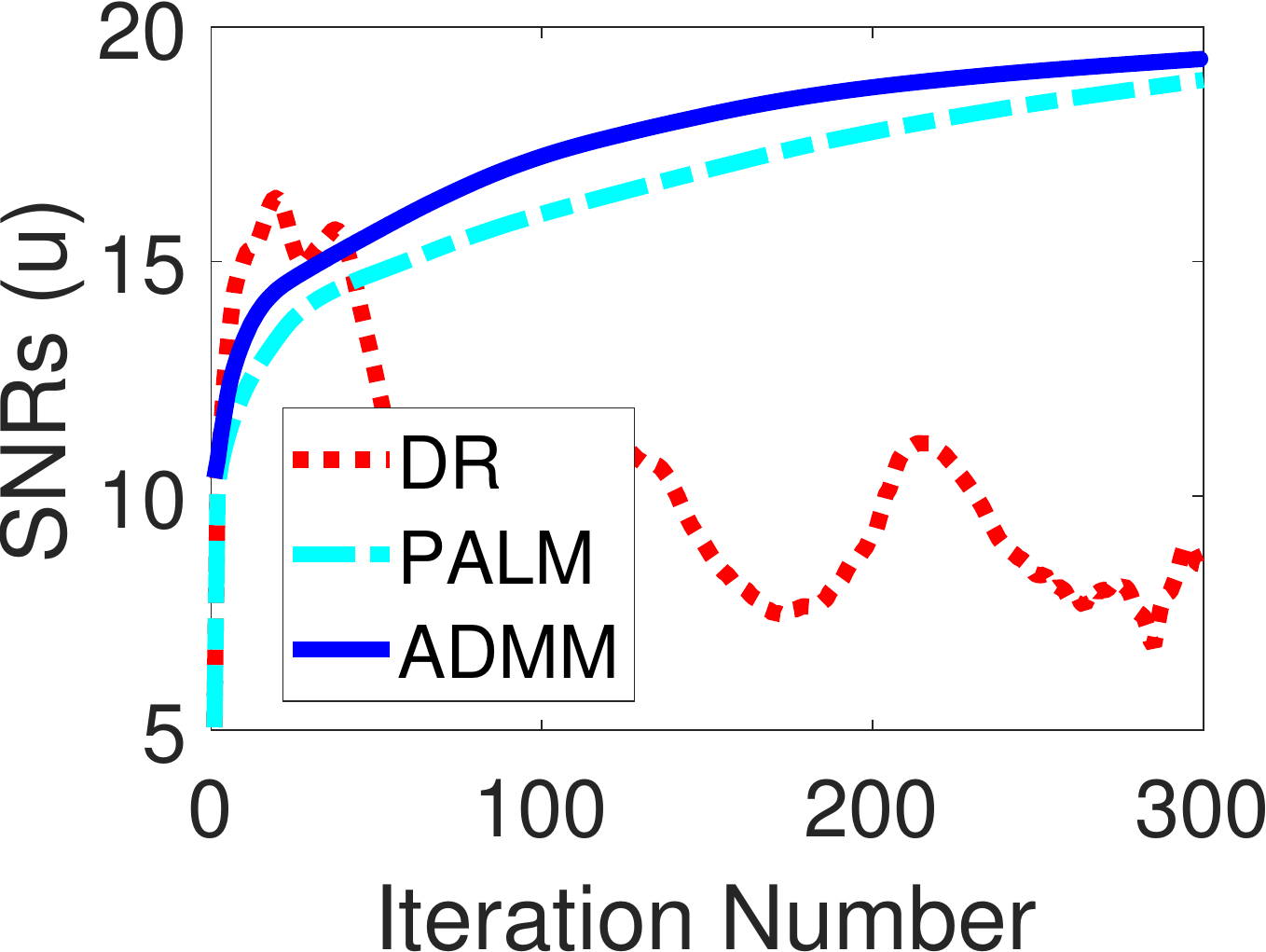}}
\subfigure[]{\includegraphics[width=.25\textwidth]{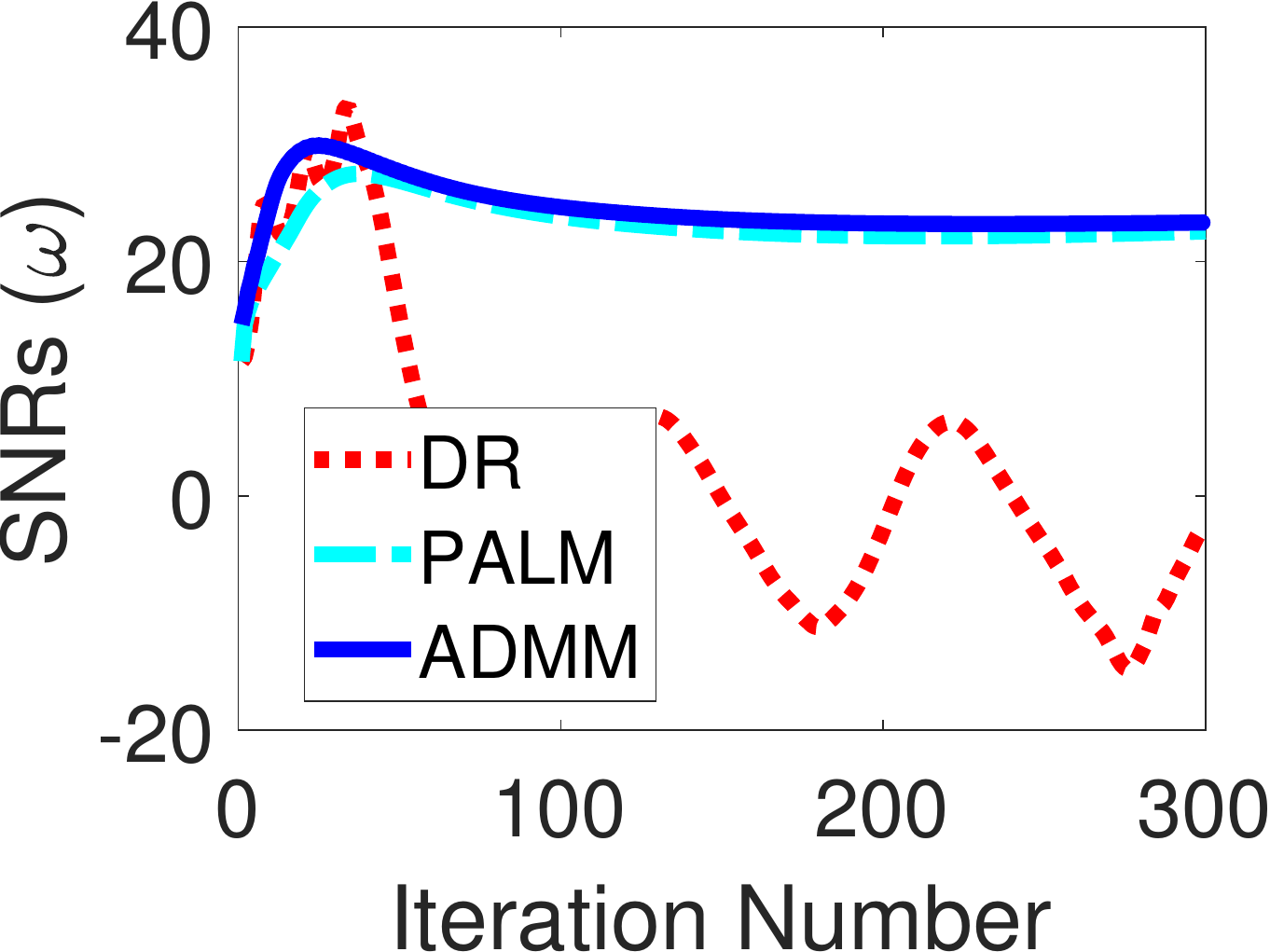}}
\subfigure[]{\includegraphics[width=.25\textwidth]{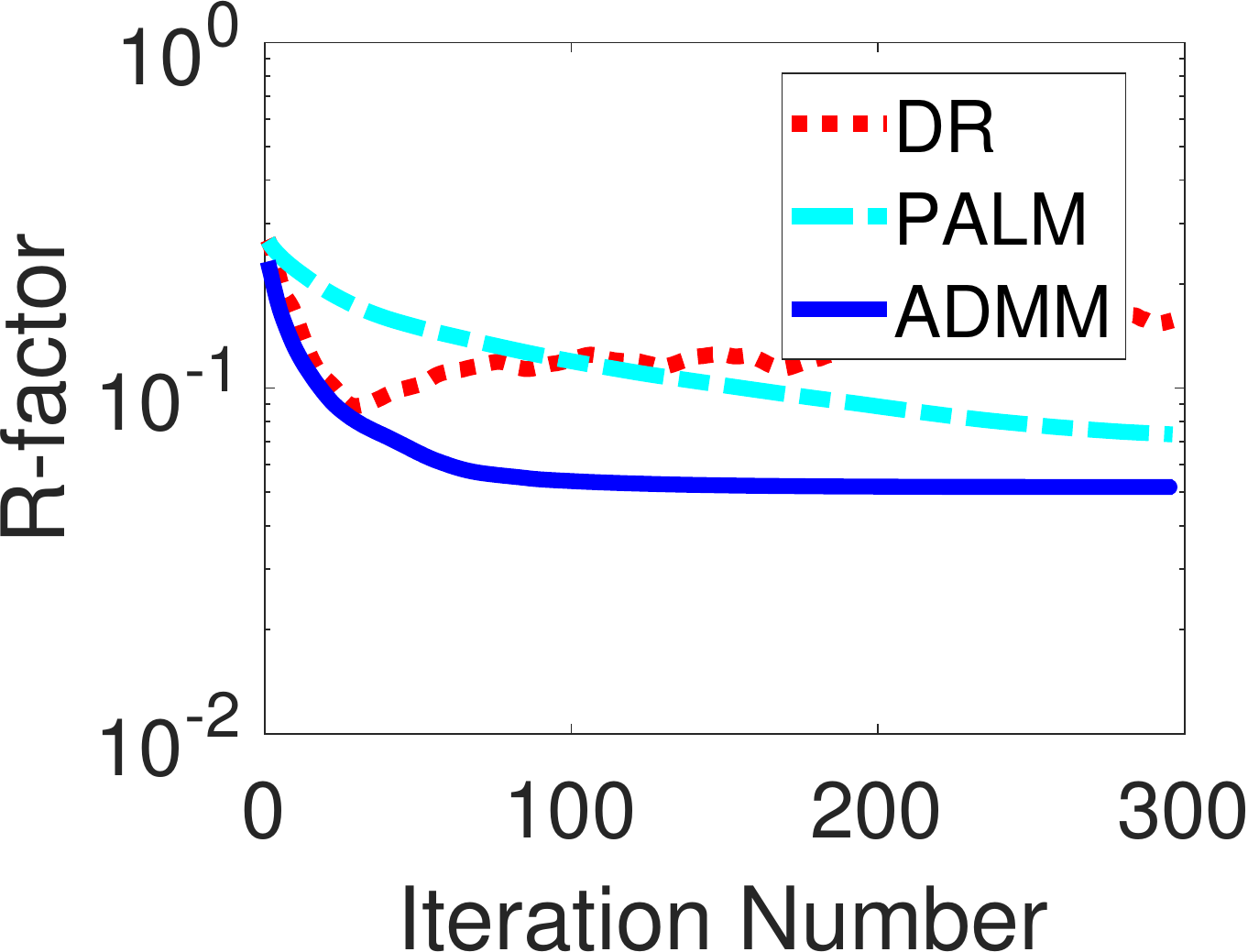}}
\subfigure[]{\includegraphics[width=.25\textwidth]{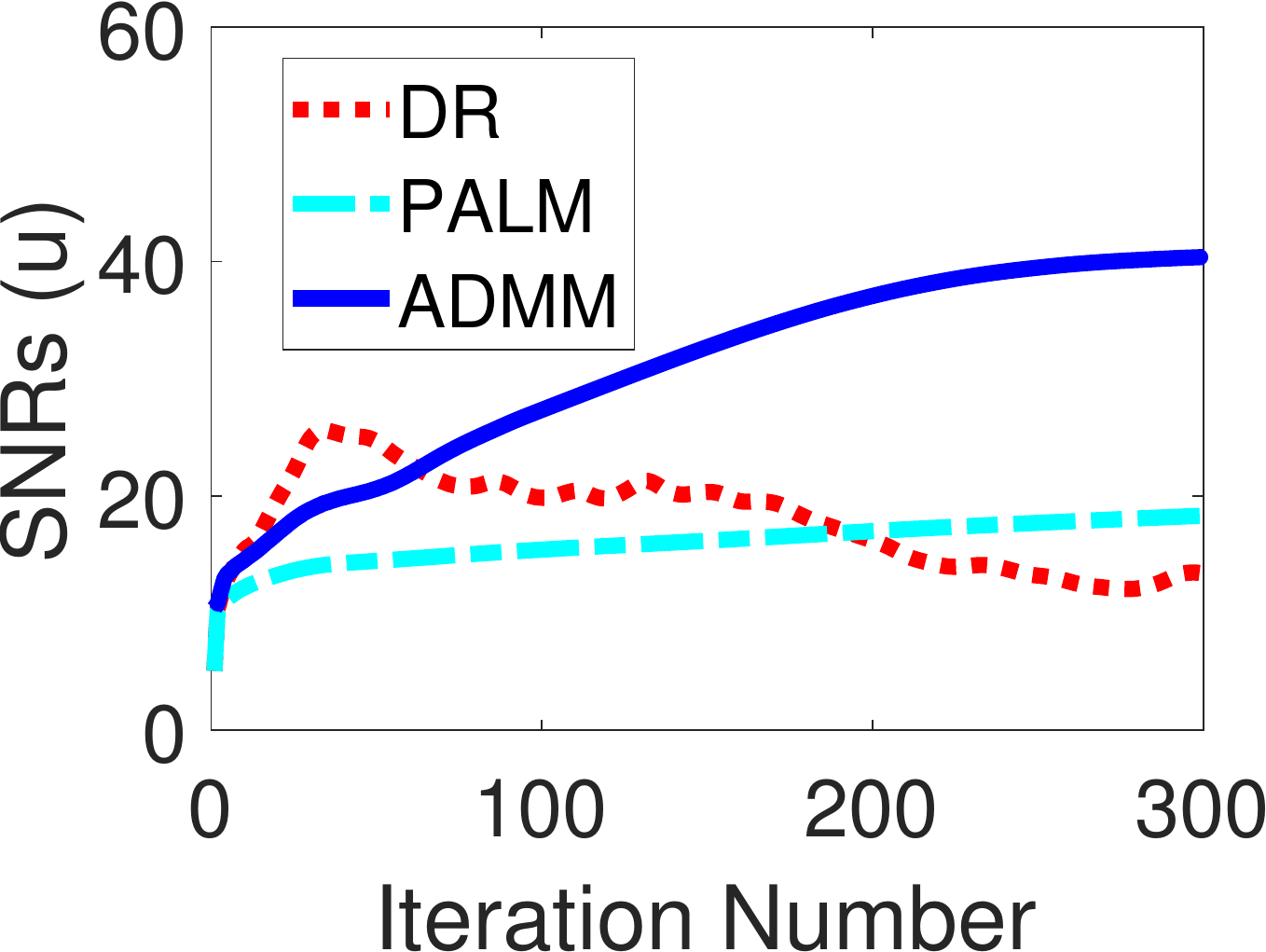}}
\subfigure[]{\includegraphics[width=.25\textwidth]{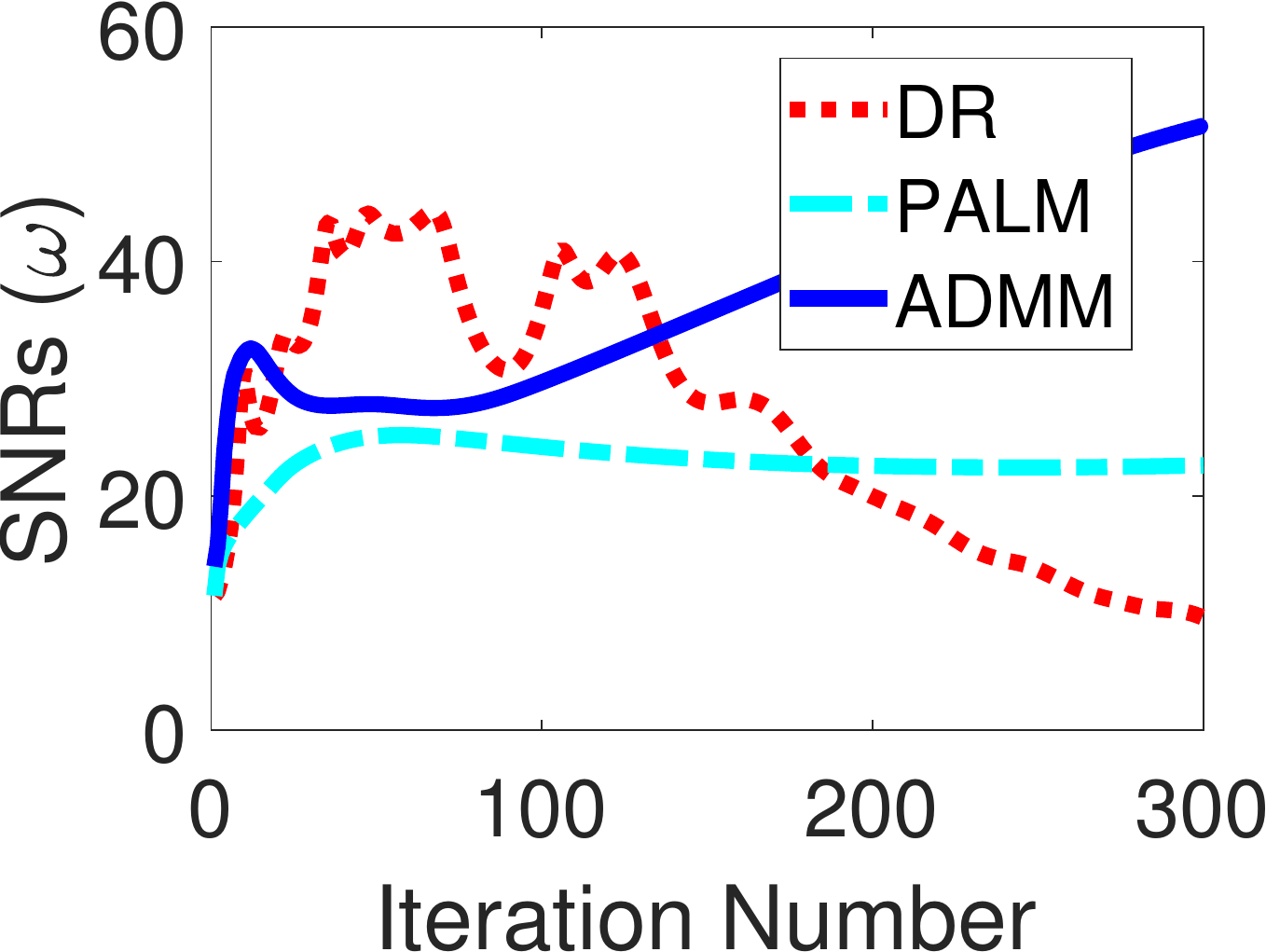}}
\end{center}
\caption{Convergence histories (random-lattice scanning) of DR \cite{thibault2009probe}, PALM \cite{hesse2015proximal} and ADMM (Algorithm \ref{algADMM}). Histories for R-factor, SNRs for the image and probe from left to right, respectively. $D=16$. Peak level $\eta=0.1$ (top), $\eta=1$ (bottom).}
\label{fig7}
\end{figure}

\subsubsection{Different metrics}
We  show the performance of Algorithm \ref{algADMM} with {four different metrics}: pAGM, pIPM, IGM and wIGM (the latter two are defined in \eqref{DF}). We conduct these experiments with $D=16$ and random-lattice scanning.

First, noiseless measurements are considered, with  convergence histories shown  in Figure \ref{fig8}.
Within the first 50 iterations,  R-factors with four different metrics  decrease with almost the same speed. As the iterations go on,
Algorithm \ref{algADMM} with IGM converges the slowest, and  it is further accelerated  with the weighted norm in wIGM.  Algorithm \ref{algADMM} with pAGM/pIPM converges faster than when using the other two metrics.

We also consider the noisy measurements.
Convergence histories for measurements contaminated by Poisson noise with $\eta=1$ are shown in Figure \ref{fig8} (second row).
It can be seen in Figure~\ref{fig8} (d) how ADMM with IGM  converges at the slowest speed. When using pAGM or pIPM the convergence speed is slightly improved. Figure \ref{fig8} (e, f) shows how the SNRs of the recovered results with pIPM are higher than those with the other metrics, since pIPM derives from the MLE for Poisson noise.

\begin{figure}[]
\begin{center}
\subfigure[]{\includegraphics[width=.25\textwidth]{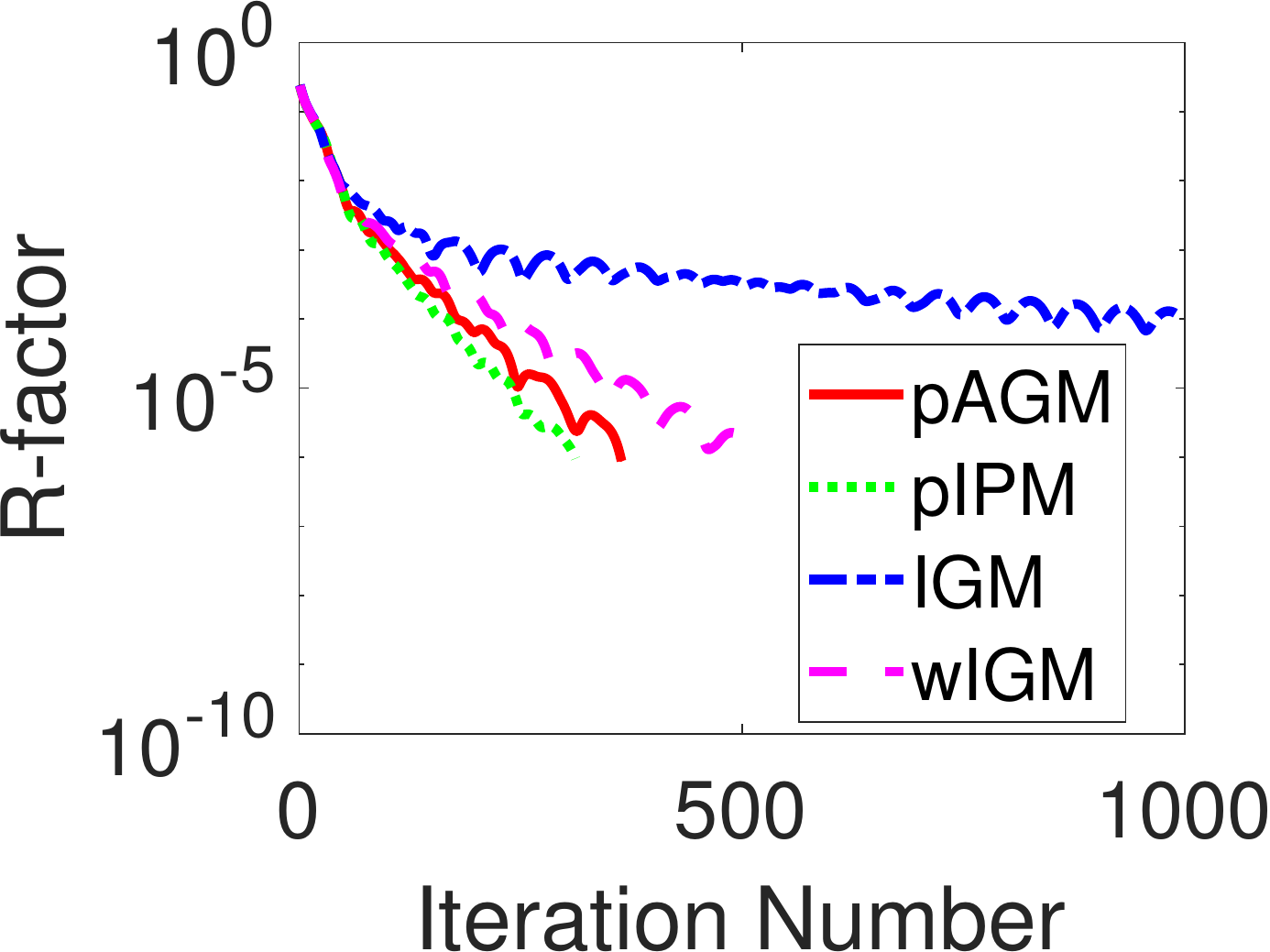}}
\subfigure[]{\includegraphics[width=.25\textwidth]{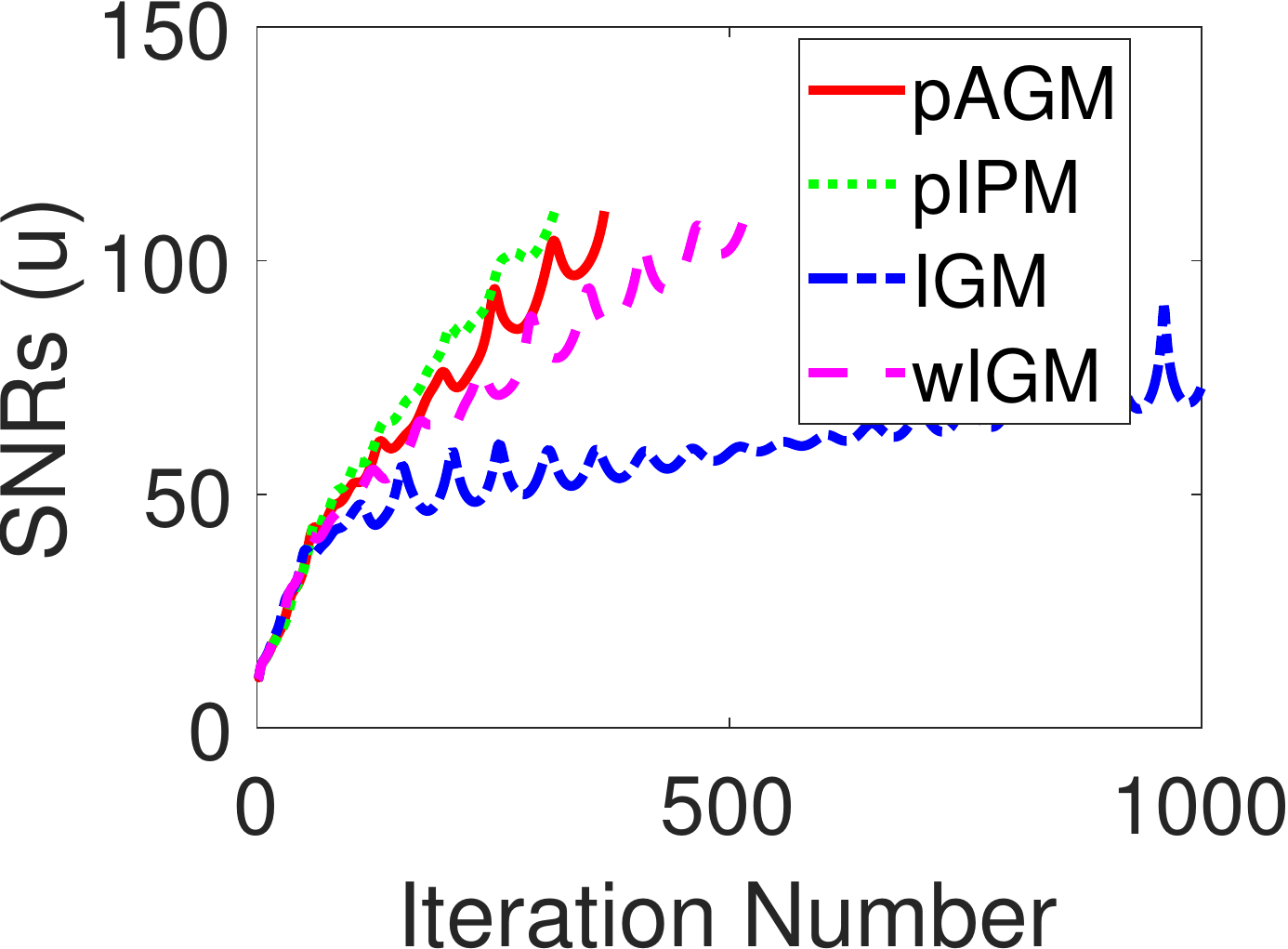}}
\subfigure[]{\includegraphics[width=.25\textwidth]{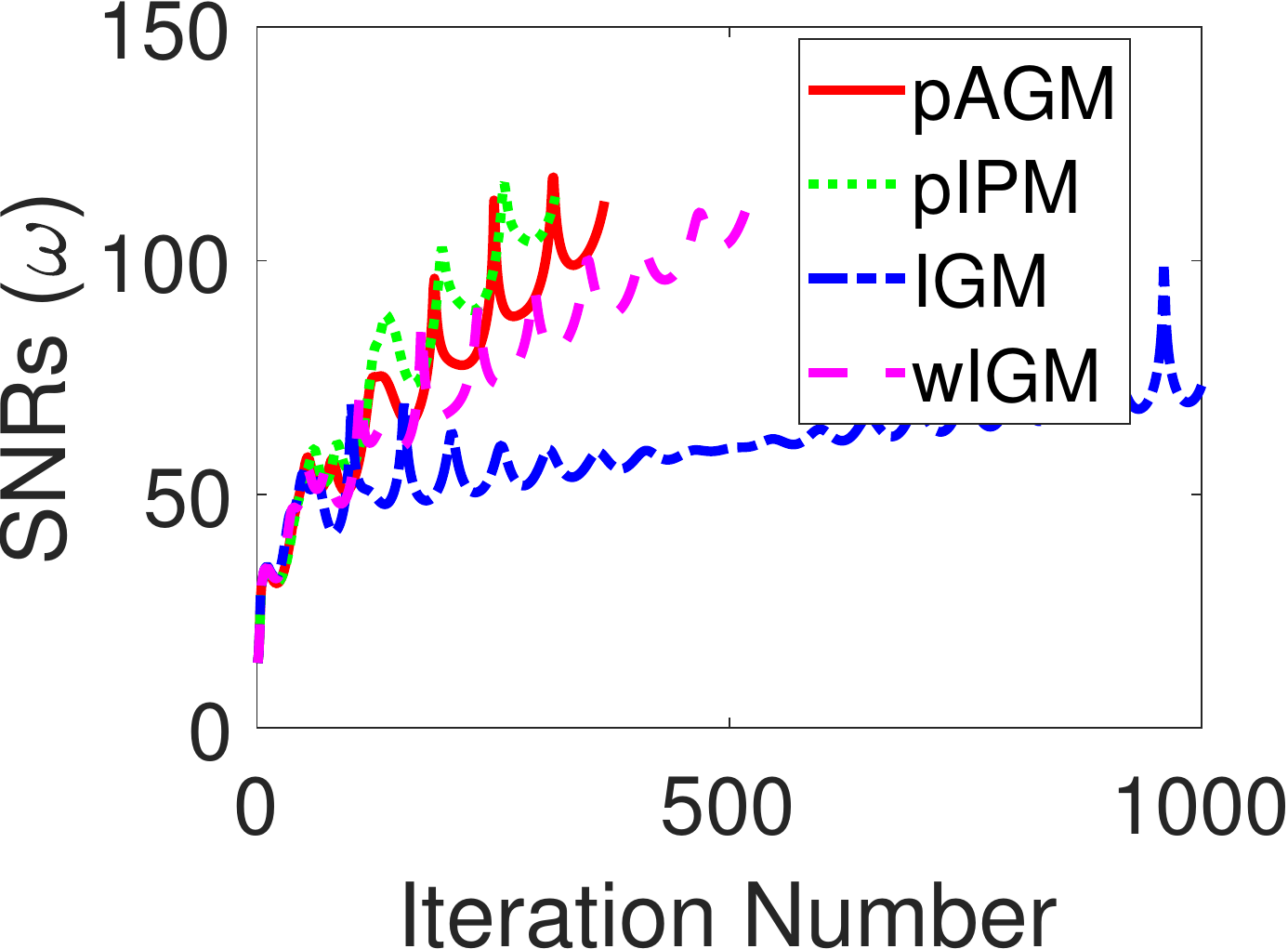}}
\subfigure[]{\includegraphics[width=.25\textwidth]{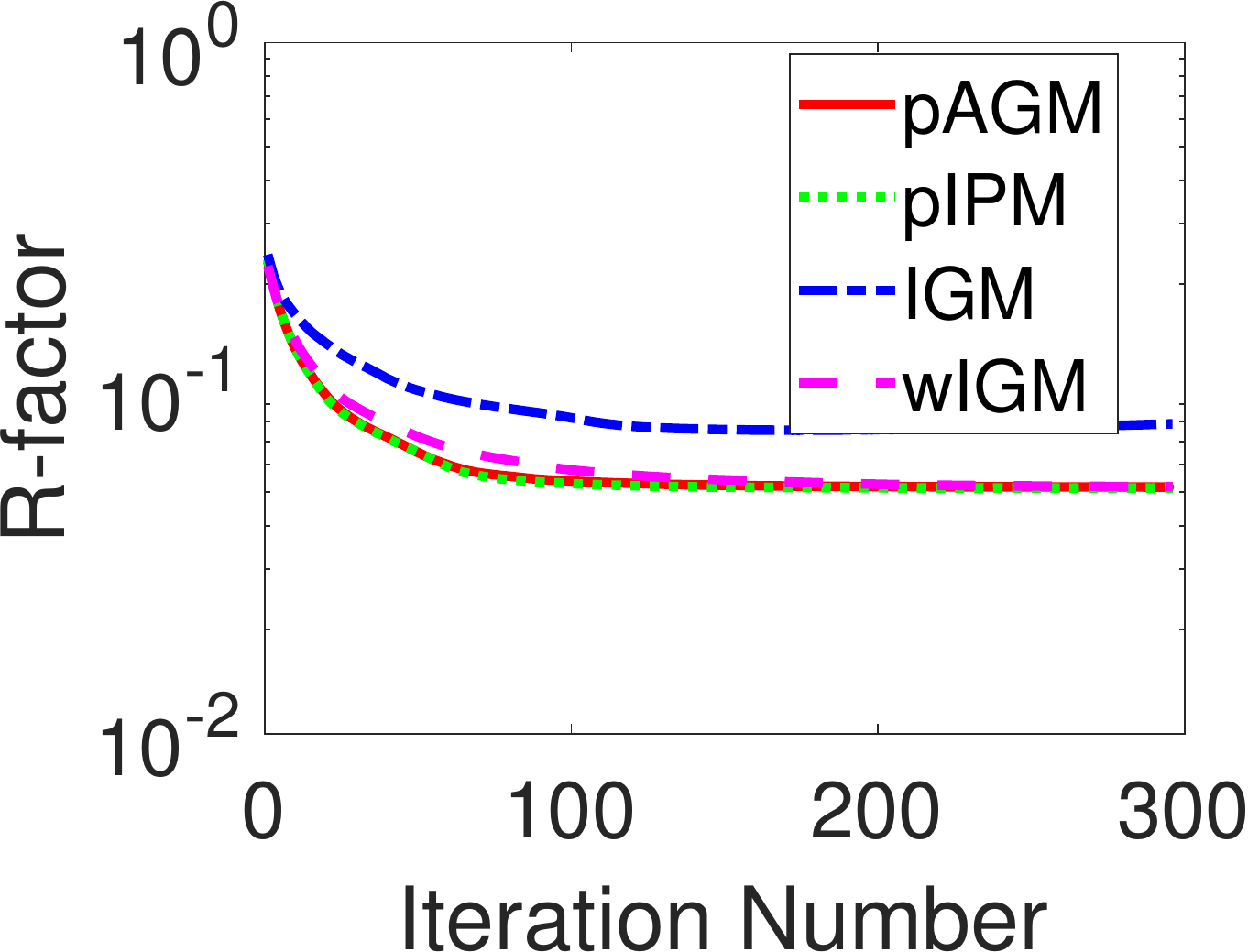}}
\subfigure[]{\includegraphics[width=.25\textwidth]{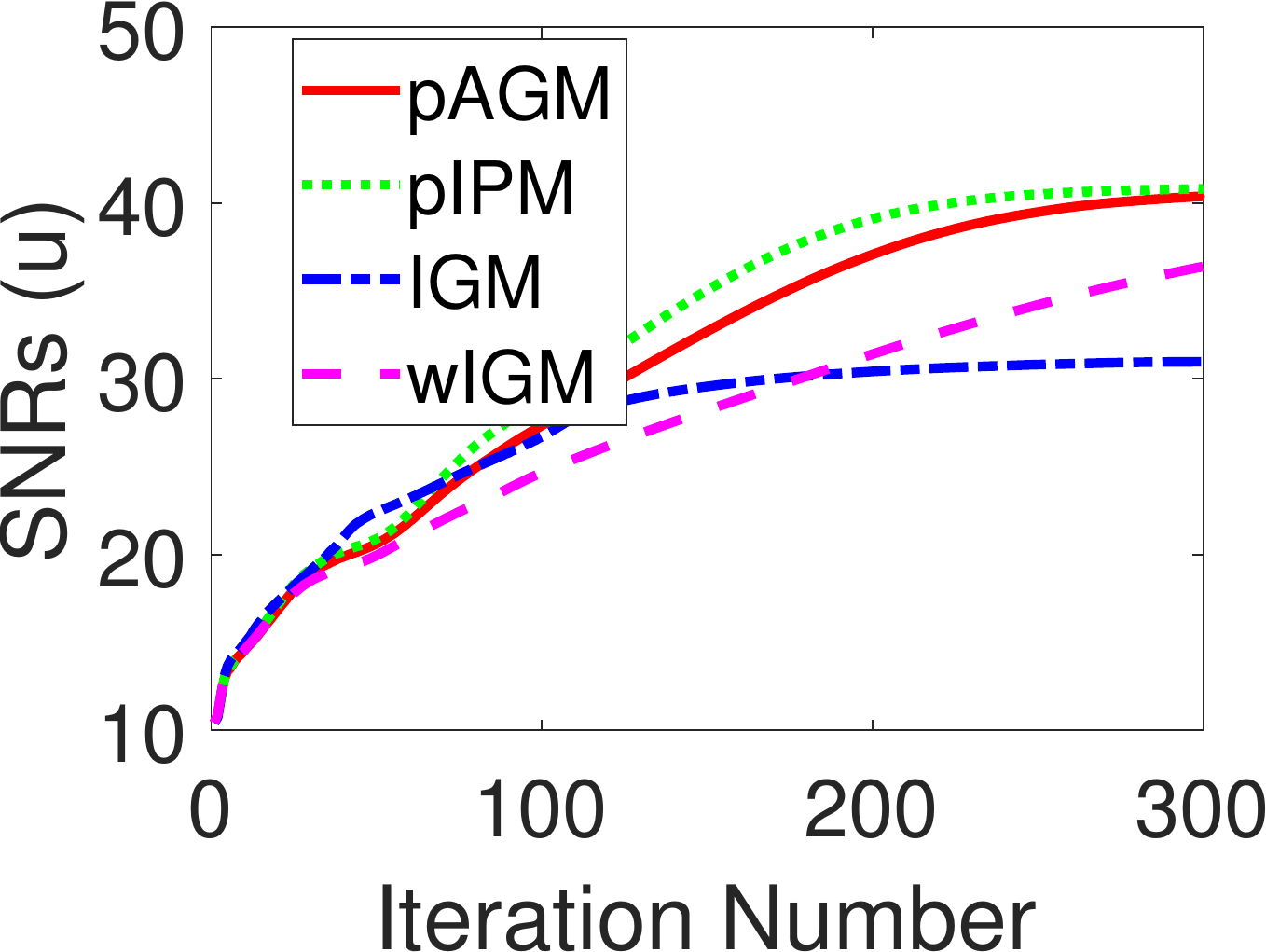}}
\subfigure[]{\includegraphics[width=.25\textwidth]{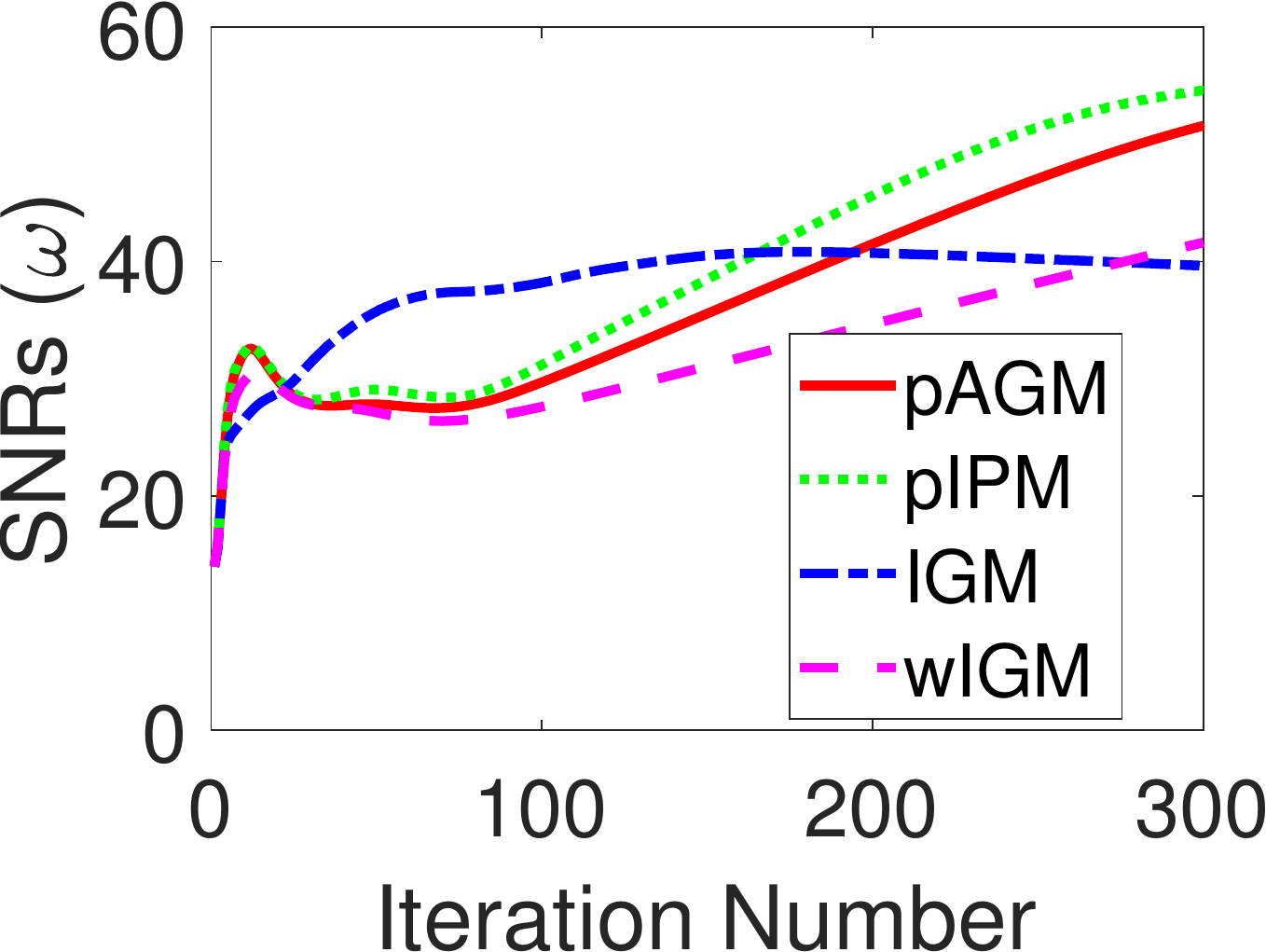}}
\end{center}
\caption{Convergence histories for Algorithm \ref{algADMM} scanning on random-lattice with different metrics.  Histories for R-factor (left), SNRs of the image (center) and probe (right). $D=16$. Noiseless data (top row) and Poisson noisy data with peak level $\eta=1$ (bottom row).  For the noiseless data: stopping condition as $\mathrm{R-factor}^k\leq 1.0\times 10^{-6}$ or iteration number reaches $Iter_{Max}=1000;$ $\beta=0.04, 0.1, 0.04$ for pAGM, pIPM and IGM, respectively, and  $\varepsilon=\beta=0.1$ for wIGM. For the noisy data: stopping condition  as iteration number reaches $Iter_{Max}=300;$  $\beta=0.2, 0.3, 100, 1$ for pAGM, pIPM, IGM and wIGM, respectively, and  $\varepsilon=10$ for wIGM.}
\label{fig8}
\end{figure}

In summary, when dealing with noiseless measurements, pAGM or pIPM based metrics are recommended in order to gain faster convergence speed. For Poisson noisy data,  pIPM should be used instead, in order to obtain higher quality reconstruction results. Due to the existence of noise, the recovered images seem blurry, which can be further improved by considering sparse prior information as in \cite{chang2016Total}. This study will be considered as future work for this research

\subsection{Performance of Algorithm \ref{algADMM-II}}
In this section only noiseless data is considered, and experiments are performed with periodical scanning on  square and hexagonal lattices.
The reconstruction results of Algorithm~\ref{algADMM-II} can be found in Figure~\ref{fig10}, and its converge histories are reported in Figure~\ref{fig11}.

With additional prior information of the probe,   Algorithm \ref{algADMM-II} reconstructs images with higher SNRs, Figure \ref{fig10} (b, d), in contrast with the results of Figure \ref{fig10} (a, c) by Algorithm \ref{algADMM}. We can see how  Algorithm \ref{algADMM-II} is able to completely remove the structural artifacts.
Figure~\ref{fig11} (a) shows how the R-factors of Algorithm~\ref{algADMM-II} decrease faster  compared with those from Algorithm \ref{algADMM}, and fewer iterations are needed to reach the given tolerance.
It is important to note that, as shown in Figure~\ref{fig11} (b, c), the SNRs of the reconstructed images and probes are greatly increased, which demonstrates the advantage of incorporating the prior information of the probe to the proposed algorithm.

\begin{figure}[h!]
\begin{center}
\subfigure[ALG1]{\includegraphics[width=.22\textwidth]{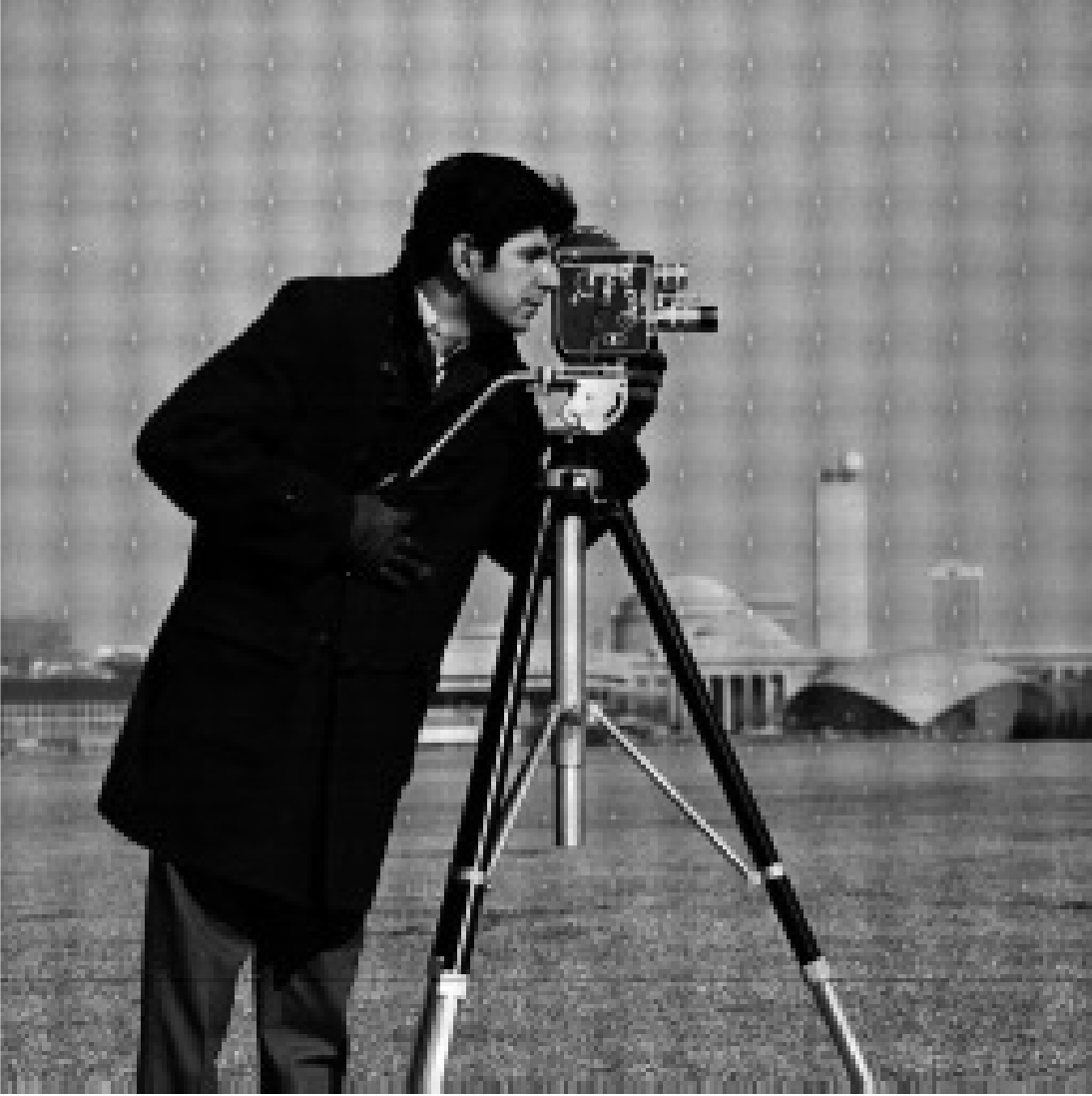}}
\subfigure[ALG2]{\includegraphics[width=.22\textwidth]{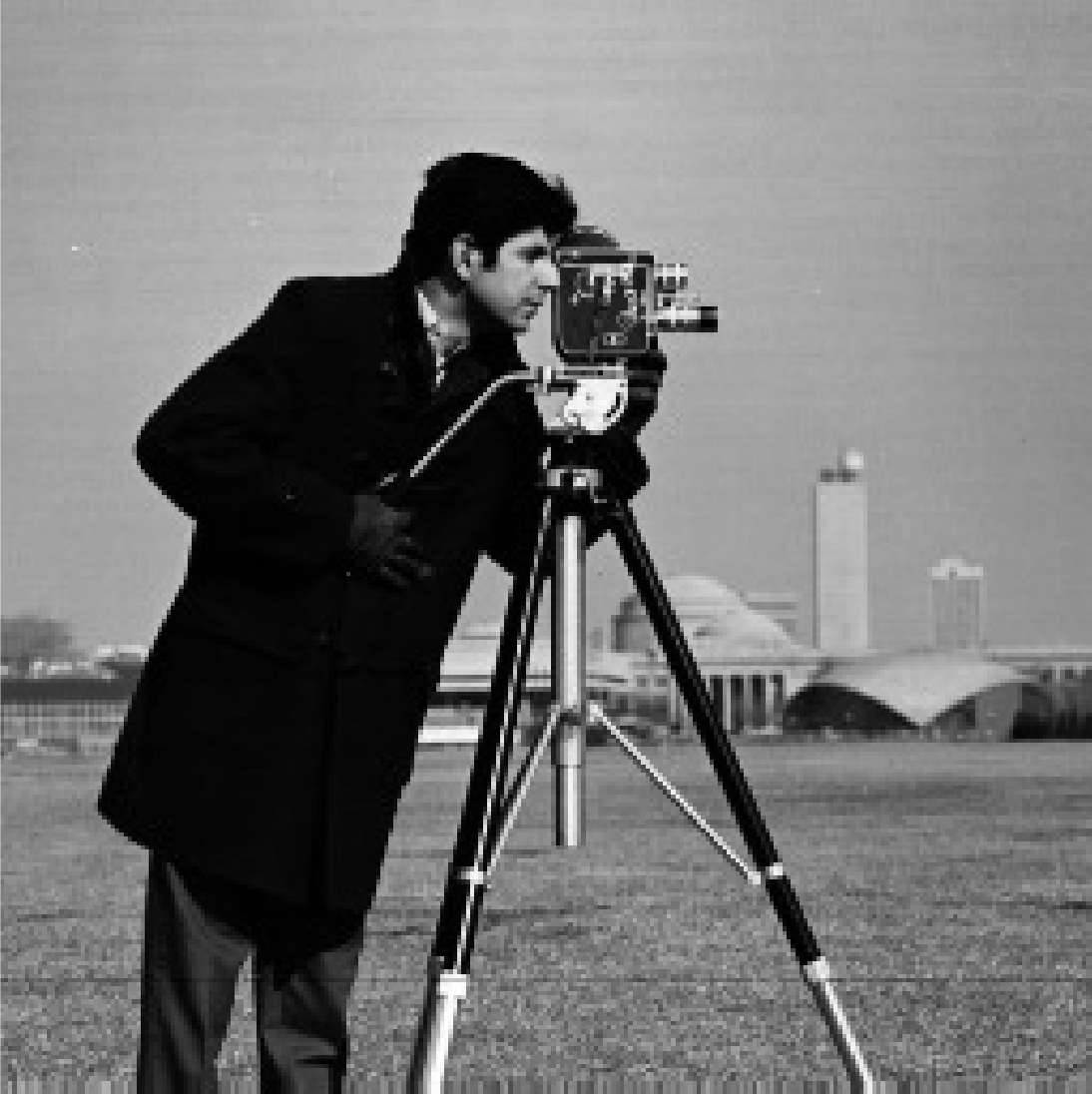}}
\qquad
\subfigure[ALG1]{\includegraphics[width=.22\textwidth]{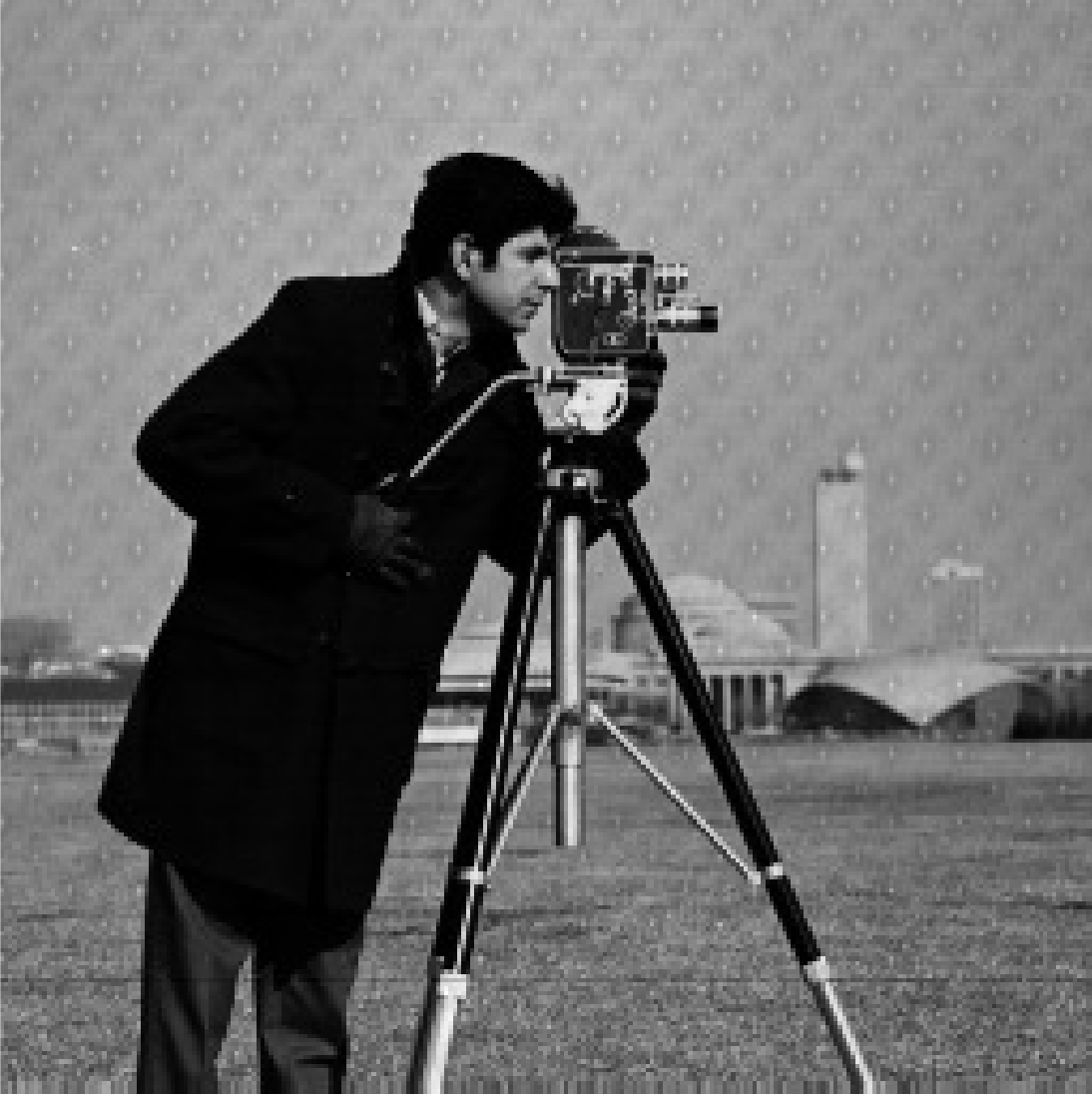}}
\subfigure[ALG2]{\includegraphics[width=.22\textwidth]{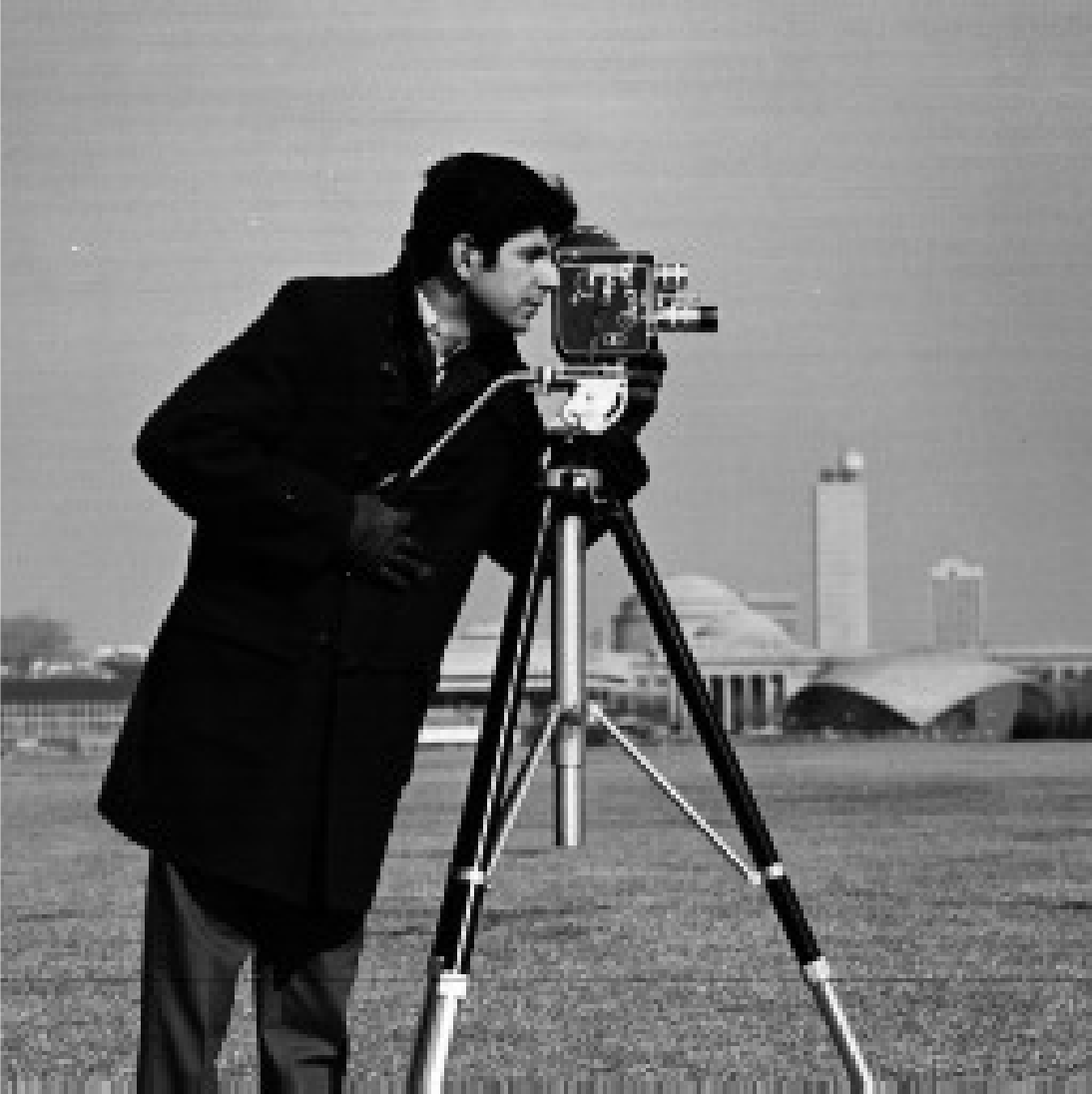}}
\end{center}
\caption{Performance of ALG1 (Algorithm \ref{algADMM}) in (a,c) and ALG2 (Algorithm \ref{algADMM-II}) in (b, d). (a, b) with square-lattice, and (c, d) with hexagonal-lattice. Algorithm \ref{algADMM-II} with AGM stops if  $\mathrm{R-factor}^k\leq 1.0\times 10^{-6}$ or if the iteration number reaches $Iter_{Max}=1000.$   $\tau=10, \beta_1=0.04, \beta_2=0.4$ for Algorithm \ref{algADMM-II}.  }
\label{fig10}
\end{figure}

\begin{figure}[]
\begin{center}
\subfigure[]{\includegraphics[width=.3\textwidth]{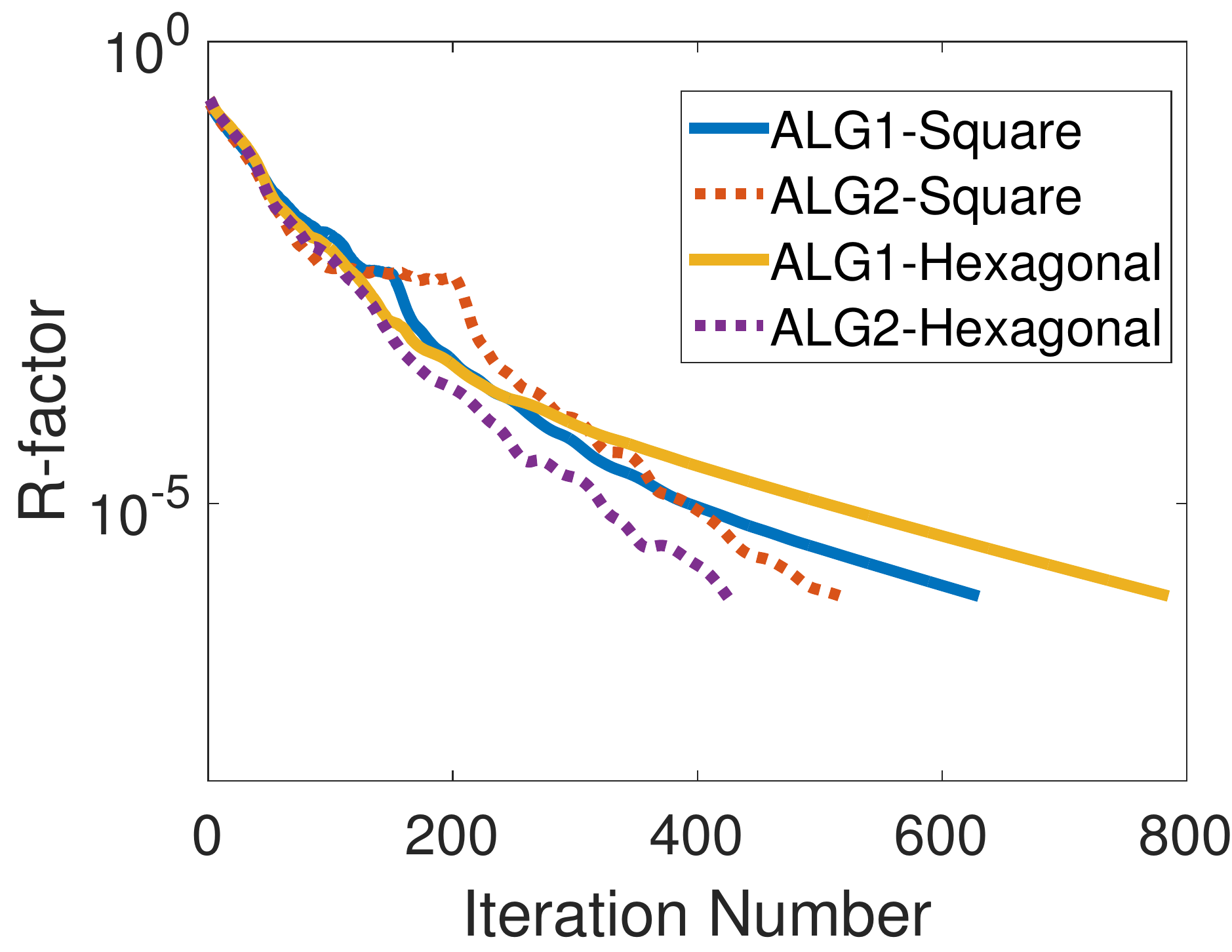}}
\subfigure[]{\includegraphics[width=.3\textwidth]{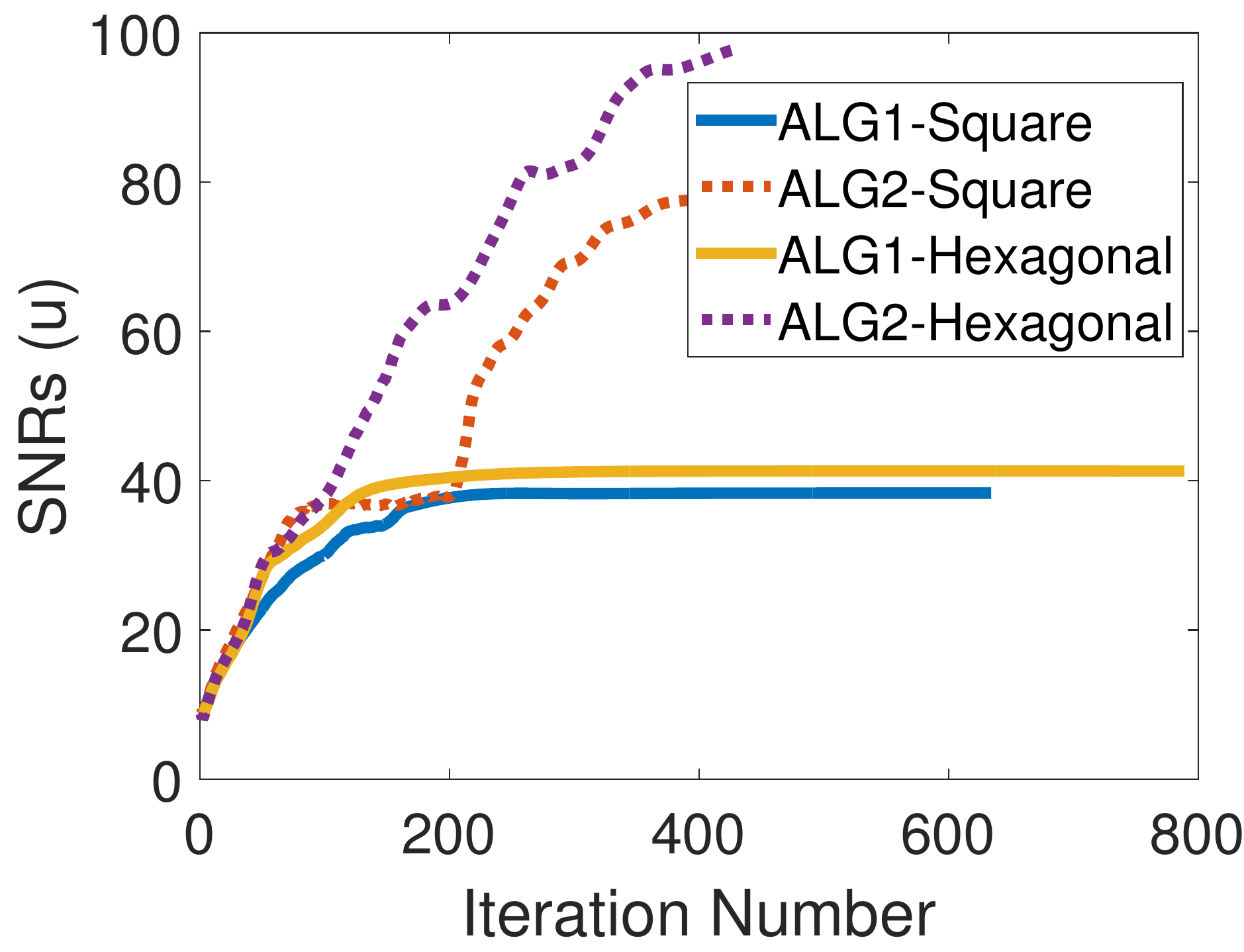}}
\subfigure[]{\includegraphics[width=.3\textwidth]{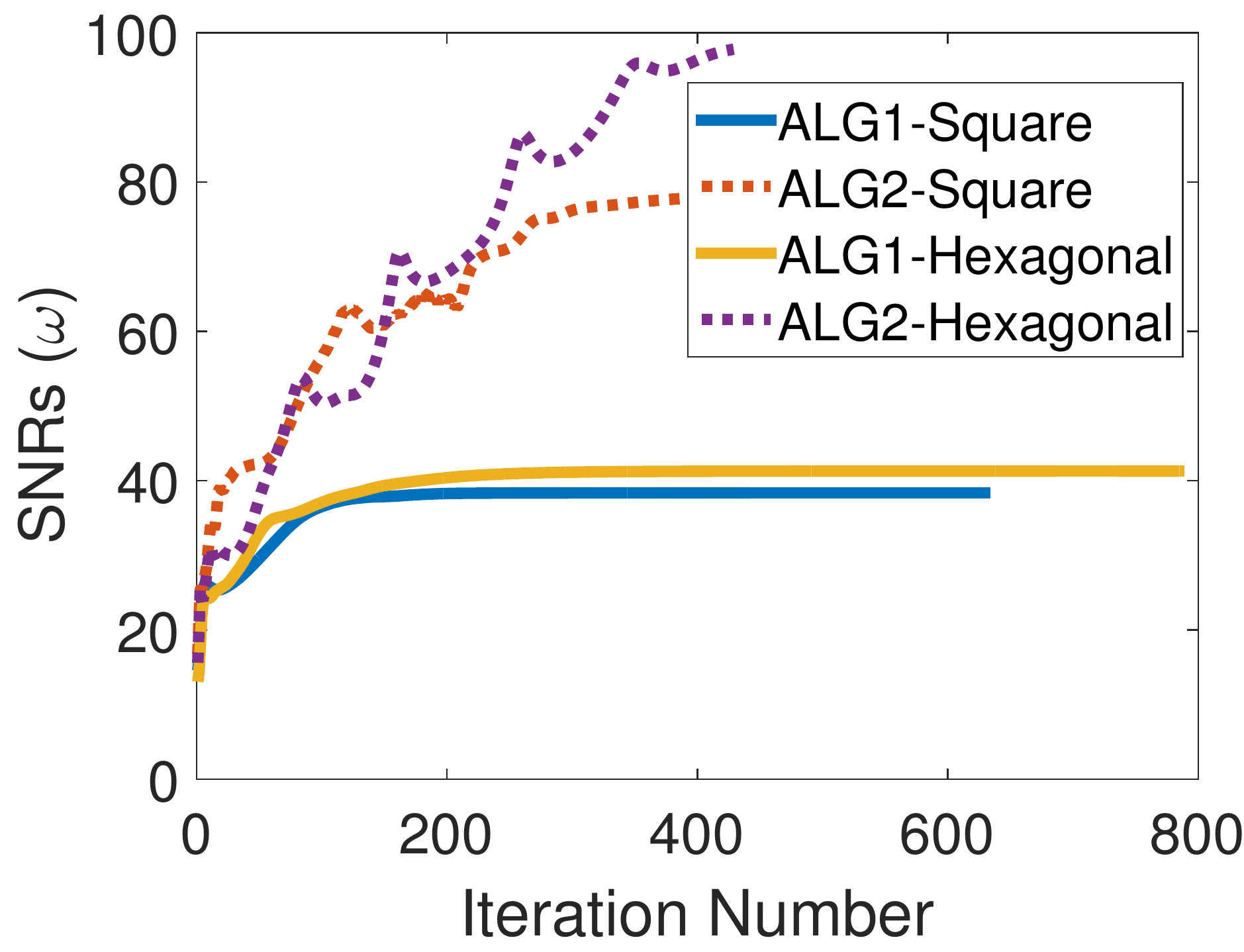}}
\end{center}
\caption{Convergence histories of ALG1 (Algorithm \ref{algADMM}) and ALG2 (Algorithm \ref{algADMM-II}). 
  Histories for R-factor (left), SNRs for the image (center) and probe (right). }
\label{fig11}
\vskip -.15in
\end{figure}


\subsection{GPU Acceleration}

We conduct this experiment using Algorithm \ref{algADMM} with different $D\in\{12,8,6,4,3\}$ to produce large-scale data, since the number of measurements $m$ satisfies
$m\approx \tfrac{n}{D^2}.$
Table~\ref{tab1} reports the corresponding execution times and GPU speedup ratios (SR) as $SR:={t_{CPU}}/{t_{GPU}}$. The GPU speedup ratios  range between $12\sim17$. Considering that the GPU \textsc{Matlab} implementation has not been optimized, the proposed algorithm shows great potential for large-scale BP-PR problems.
The number of iterations ($IterNum$) to reach the given tolerance are  about $300\sim340$ with different $D$, which demonstrates that the convergence speed is not sensitive to the number of measurements for a large-scale problem. Although it is convenient to produce more measurements experimentally, with our algorithms, using fewer number of measurements can save memory and reduce the computation cost.  One can also observe the computation cost for  CPU (or GPU) is almost linear \emph{w.r.t.} the number of measurements $m$, which is consistent with the runtime complexity  of Algorithm \ref{algADMM} as about $O(m(\log_2\bar m+\tfrac{n}{\bar m}) \times IterNum)$.

\begin{table}[h!]
\caption{Performances of Algorithm \ref{algADMM} on GPU v.s. CPU. $IterNum$ denotes the iteration number to reach the given tolerance, $t_{CPU}$ and $t_{GPU}$ denote the elapsed time in seconds for CPU and GPU, respectively. SR denotes the speedup ratio.    Algorithm \ref{algADMM} stops if the R-factor is less than $1\times 10^{-5}$ or if the maximum iteration number reaches $1000$.  The code runs on a desktop with  CPU (Intel I7-4820K, 64GB RAM) and GPU (Nvidia TITAN-X with 12 GB GDDR5X) in single precision mode. The GPU version is directly implemented using the build-in functions of \textsc{Matlab}.
}
\begin{center}{}
\scalebox{.9}{
\begin{tabular}{|c|c|c|c|c|c|c|c|}
\hline
$D$&12&8&6&4&3\\
\hline
${m}$&$64^2\times 441$&$64^2\times 1024$&$64^2\times 1764 $&$64^2\times 4096 $&$64^2\times
7225$\\
\hline
$IterNum$& 297 &332 &296&334&336\\
\hline
$t_{CPU}$  (s) &46.9&97.1&188.1&373.4&773.7\\
$t_{GPU}$ (s)&4.0&7.8&10.9&26.4&46.1 \\
\hline
SR  & 11.7& 12.4&17.2&14.1&16.8\\
\hline
\end{tabular}
}
\end{center}
\label{tab1}
\end{table}
\subsection{Extensive tests: different probe and images}
{
In order to test the robustness of proposed algorithms, we conduct more experiments on a different probe (non-circularly symmetric) with its absolute value shown in Figure \ref{fig12-0} (c).
Two other images shown in Figure \ref{fig12-0} (a)-(b) are considered with scanning on square-lattice and random-lattice. The convergence curves of the R-factor,
SNRs of recovered images and probes are reported in Figure \ref{fig12}, where one can readily see that the R-factors by proposed algorithms decrease faster, and
the recovery results have better quality (bigger SNRs). Especially the test reported in the first column of Figure \ref{fig12}  shows that ADMM-Prox, with the  help of additional proximal terms, becomes more stable and faster compared with standard ADMM.
}

\begin{figure}[]
\begin{center}
\subfigure[]{\includegraphics[width=.24\textwidth]{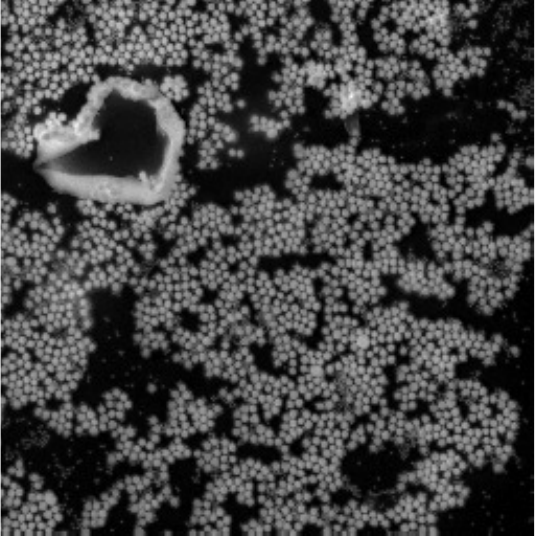}}
\subfigure[]{\includegraphics[width=.24\textwidth]{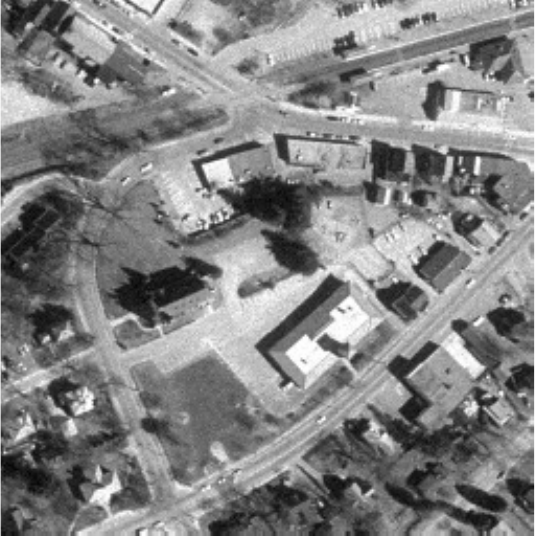}}\qquad\qquad
\subfigure[]{\includegraphics[width=.12\textwidth]{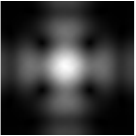}}
\end{center}
\caption{Two testing images ($256\times 256$ pixels): A complex-valued image with  the absolute value in (a), and a real-valued image in (b); A  non-circularly symmetric probe in (c) with $64\times 64$ pixels (only show the absolute value)  }
\label{fig12-0}
\end{figure}

\begin{figure}[]
\begin{center}
\subfigure[]{\includegraphics[width=.24\textwidth]{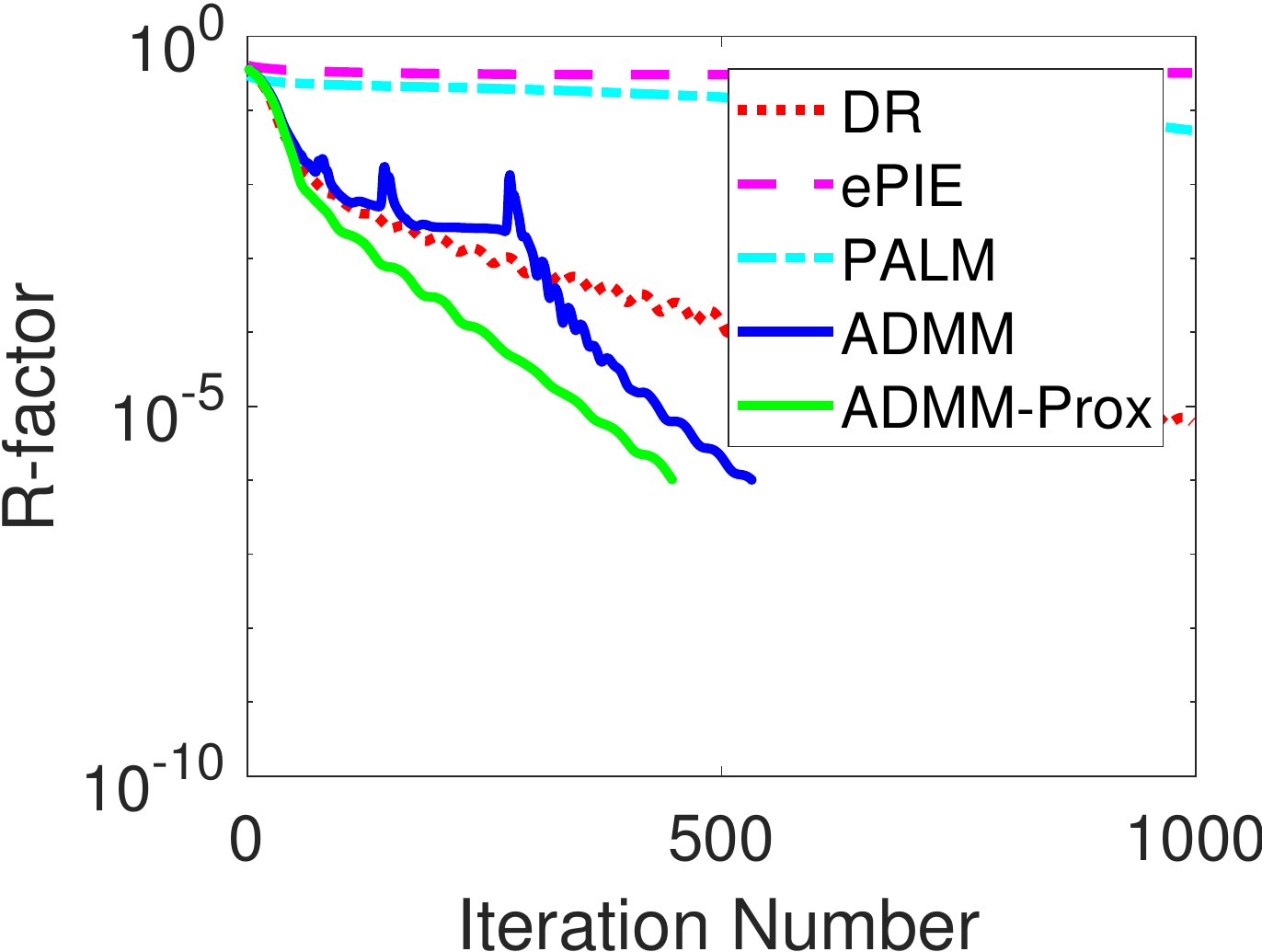}}
\subfigure[]{\includegraphics[width=.24\textwidth]{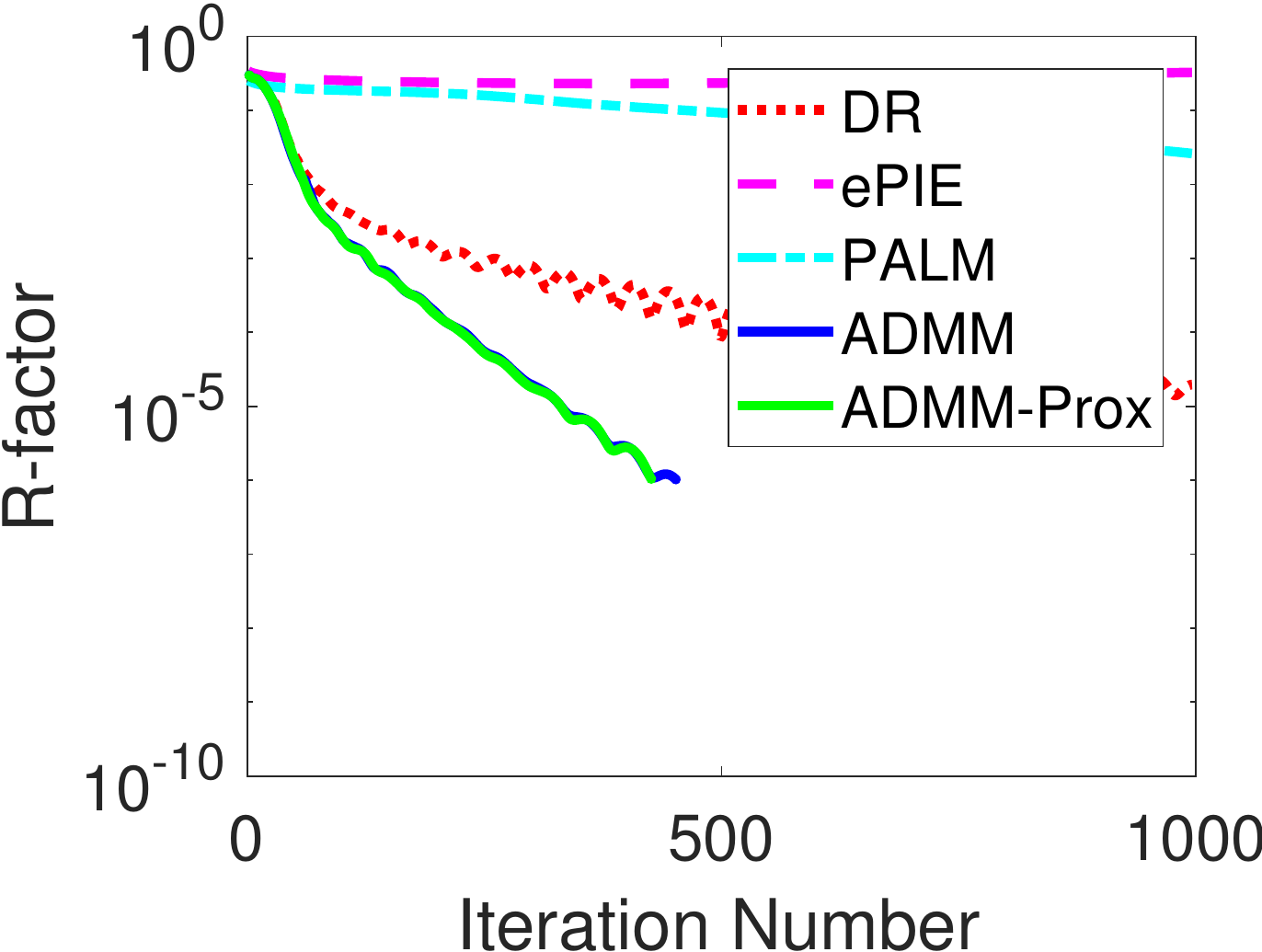}}
\subfigure[]{\includegraphics[width=.24\textwidth]{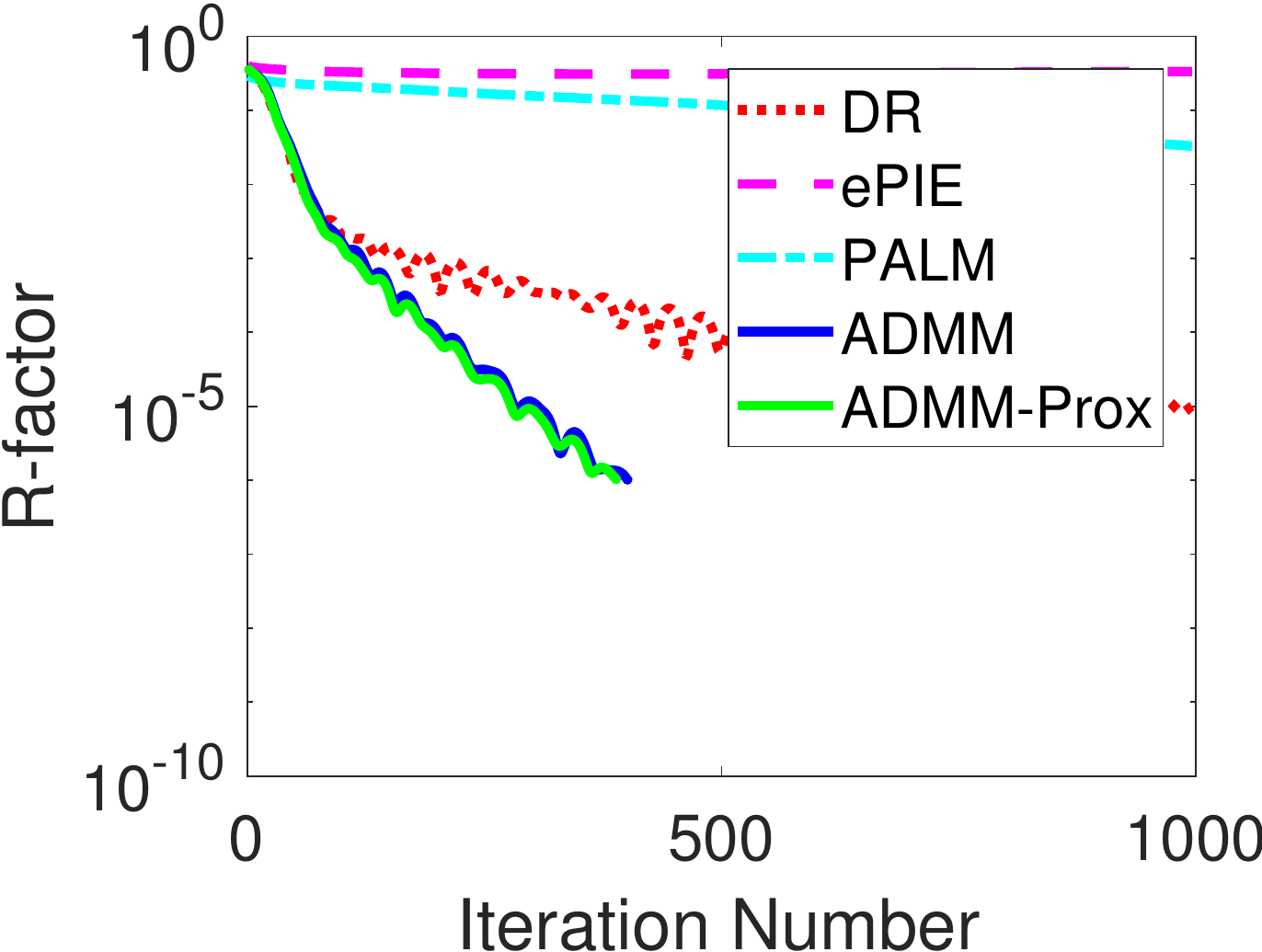}}
\subfigure[]{\includegraphics[width=.24\textwidth]{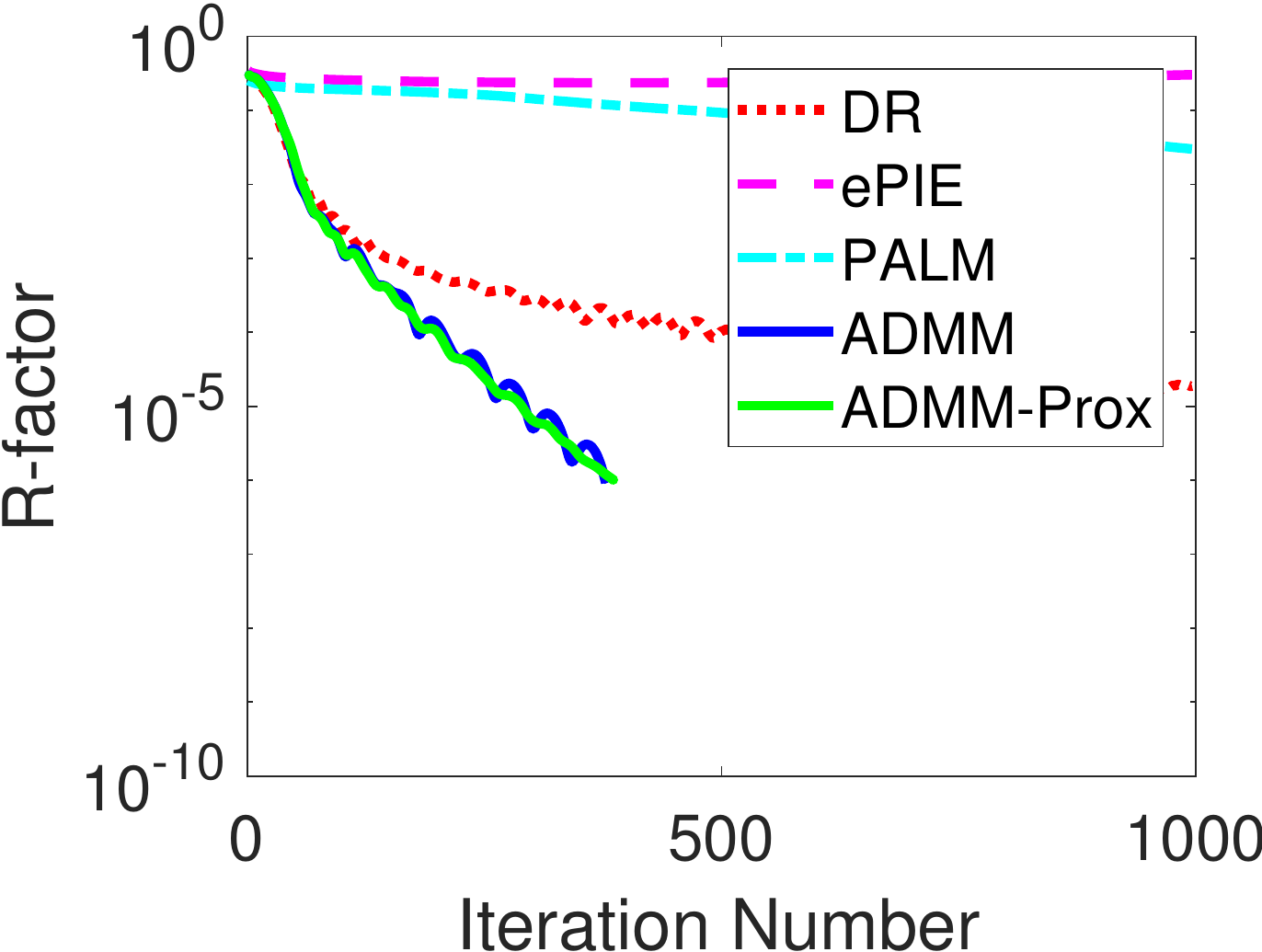}}\\
\subfigure[]{\includegraphics[width=.24\textwidth]{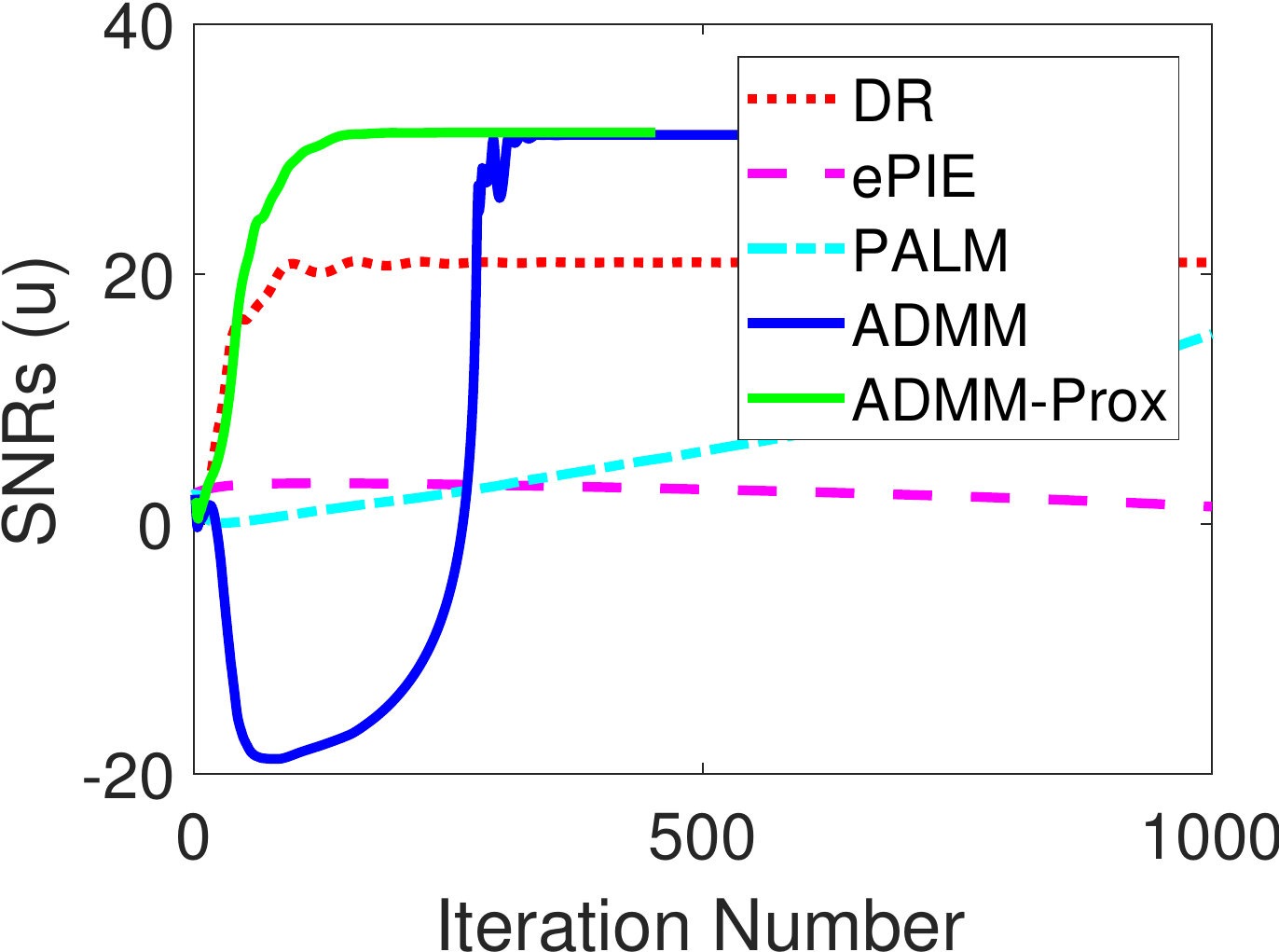}}
\subfigure[]{\includegraphics[width=.24\textwidth]{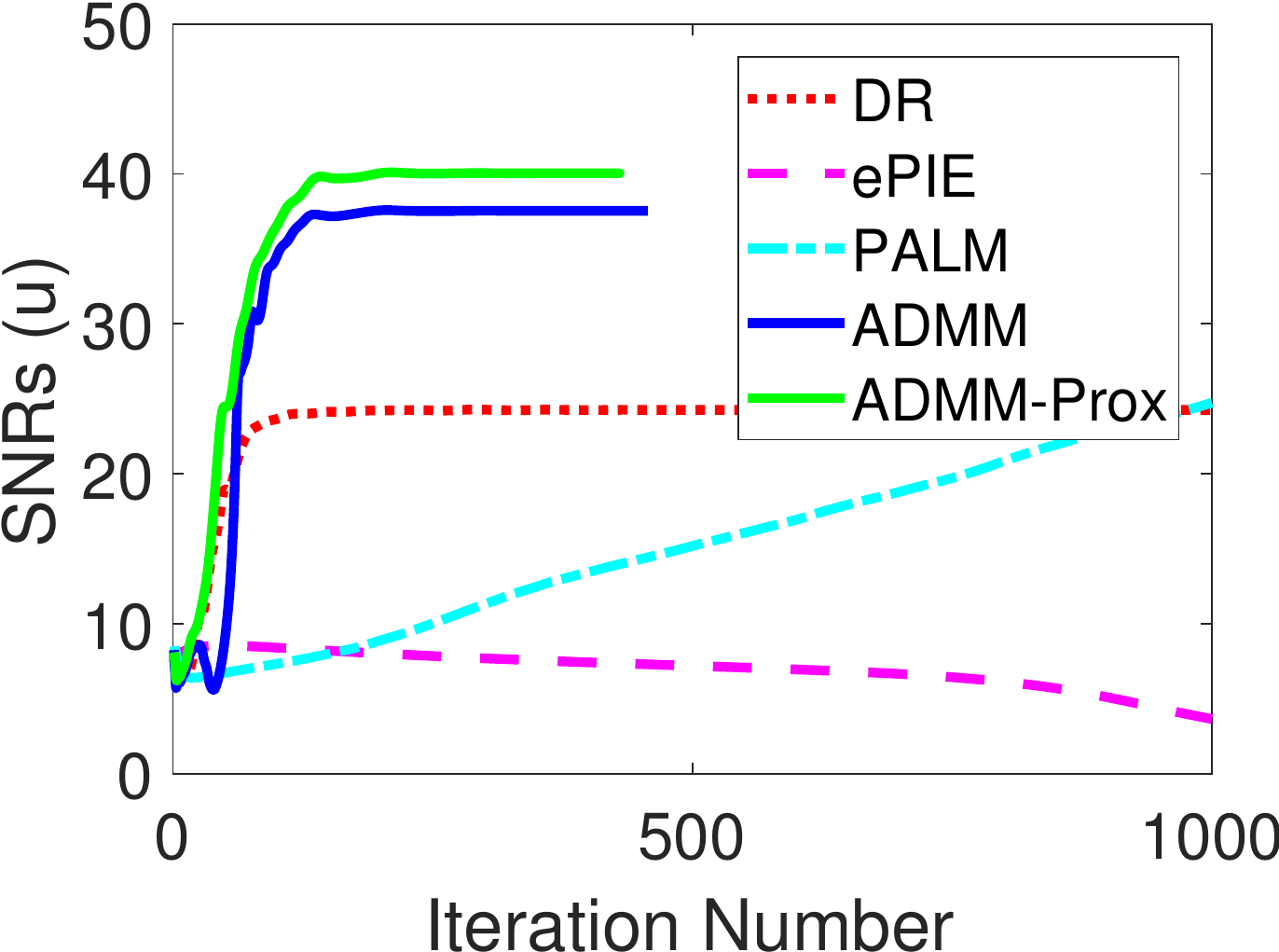}}
\subfigure[]{\includegraphics[width=.24\textwidth]{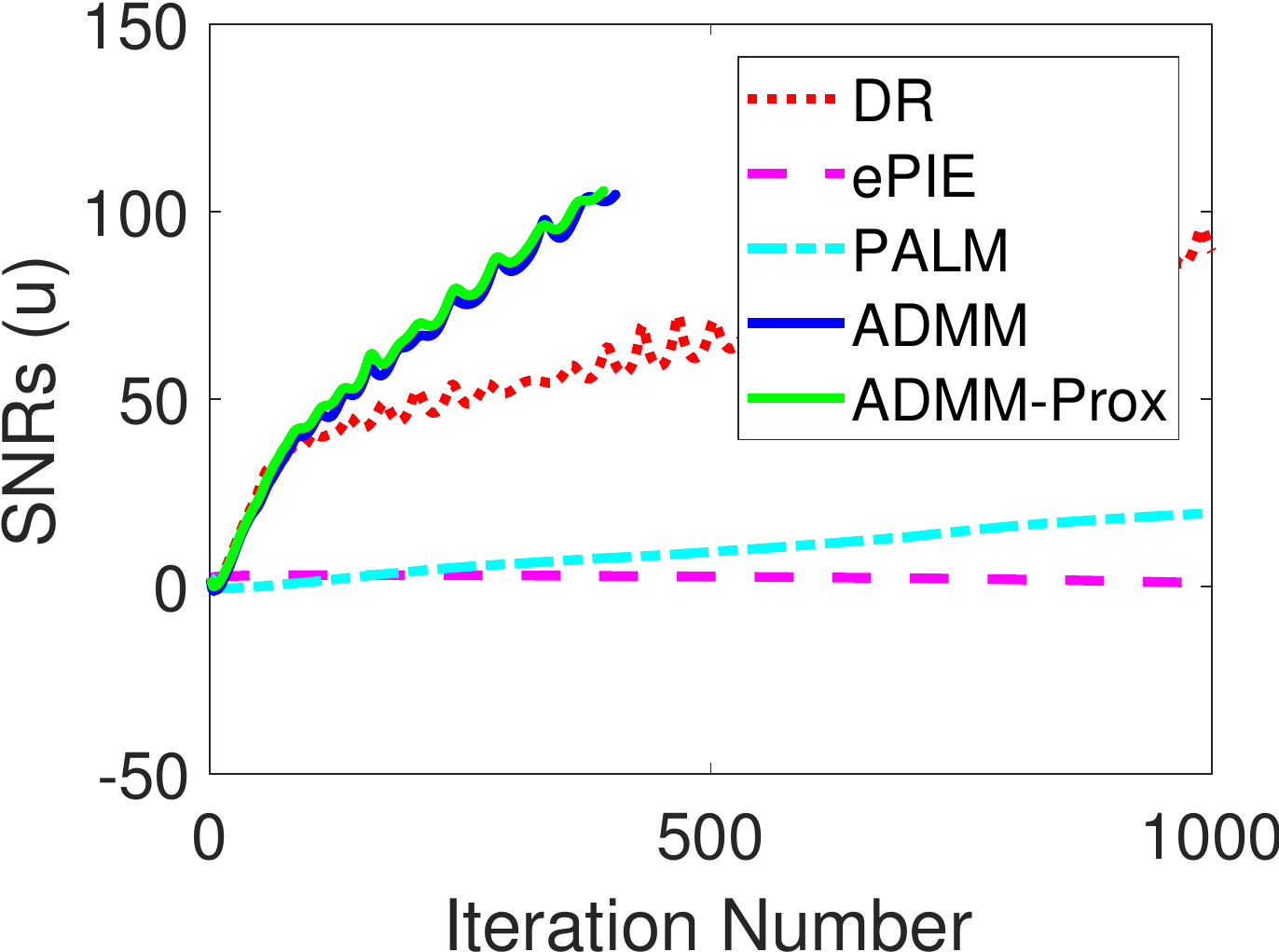}}
\subfigure[]{\includegraphics[width=.24\textwidth]{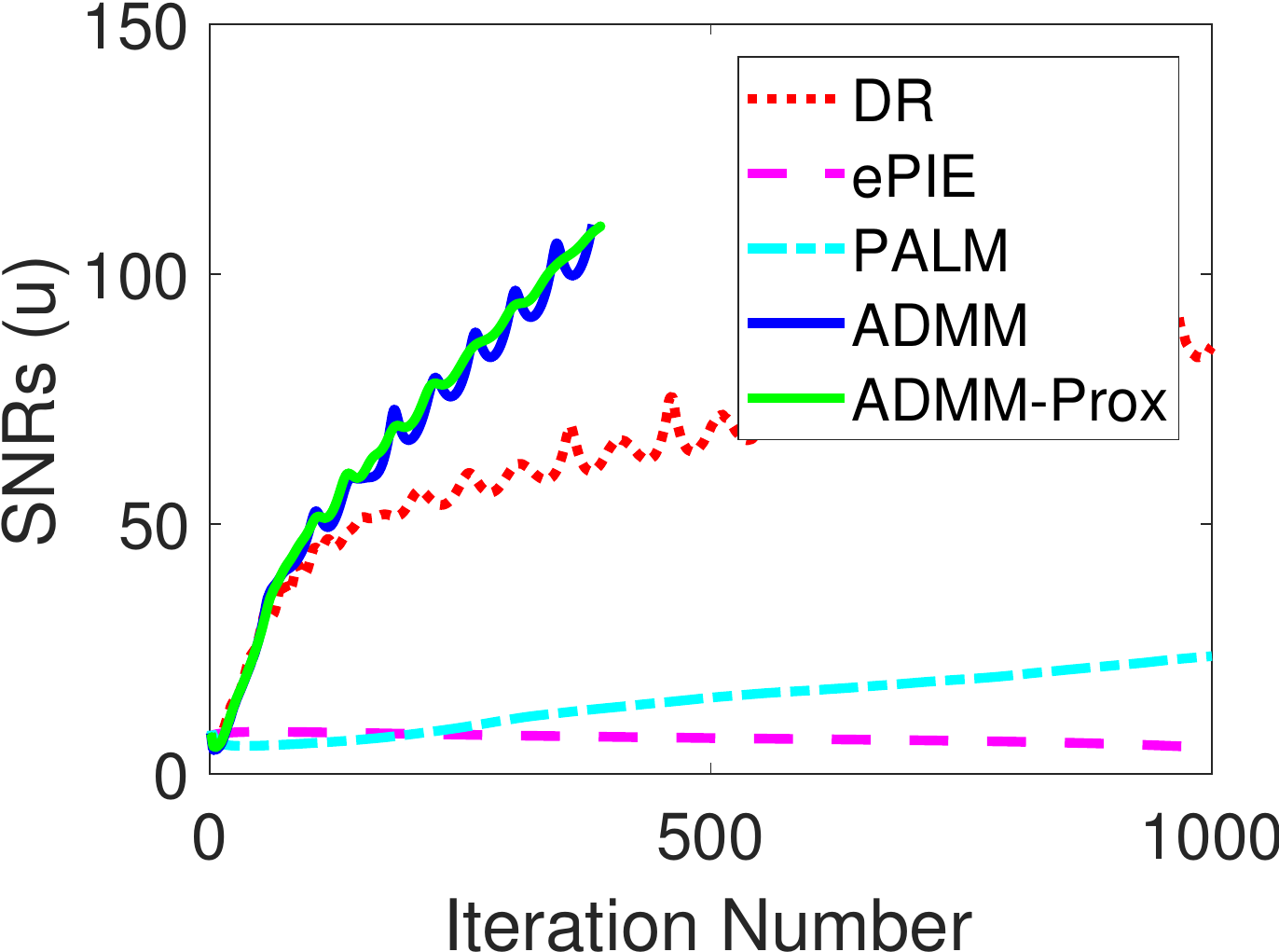}}\\
\subfigure[]{\includegraphics[width=.24\textwidth]{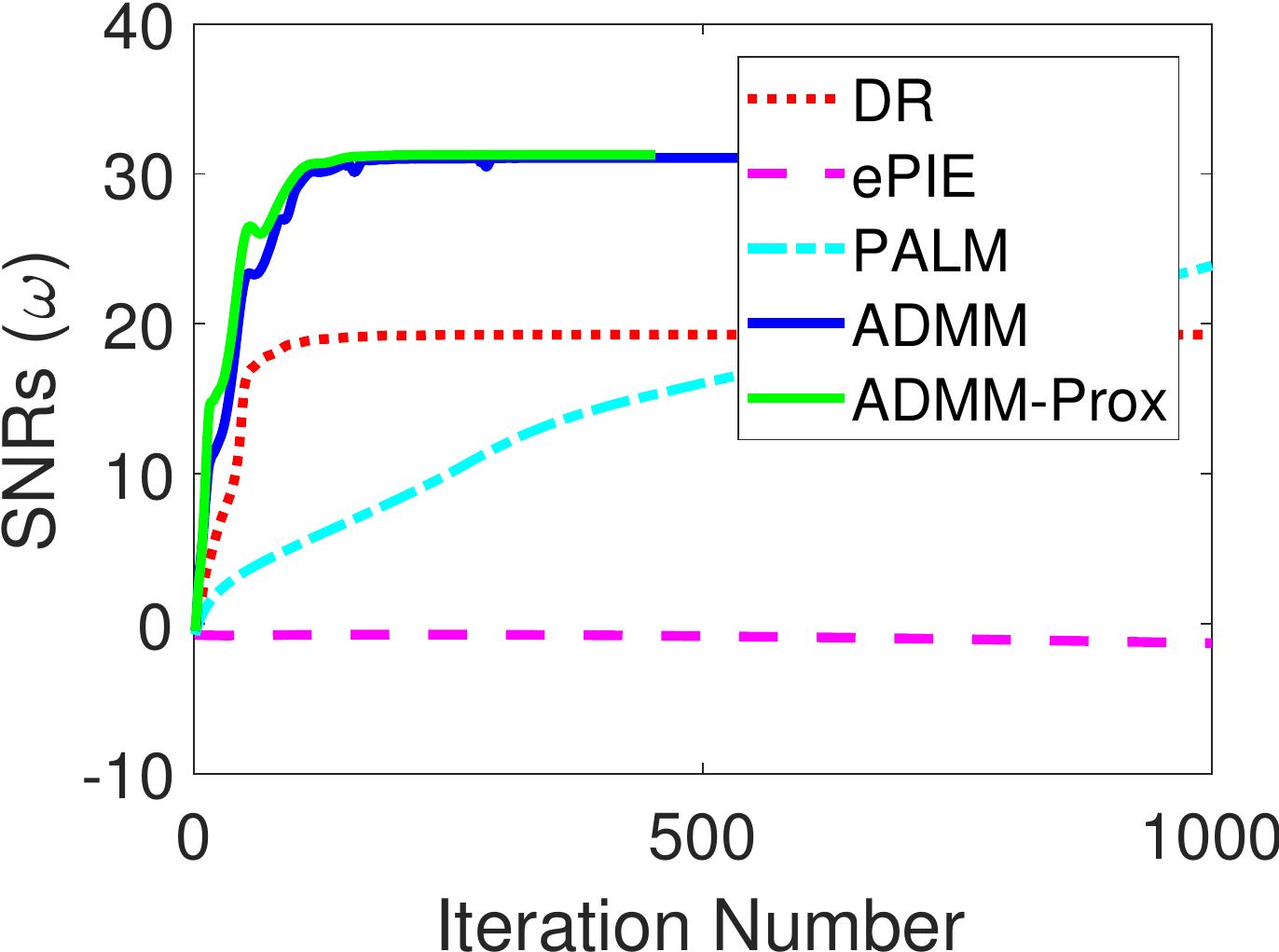}}
\subfigure[]{\includegraphics[width=.24\textwidth]{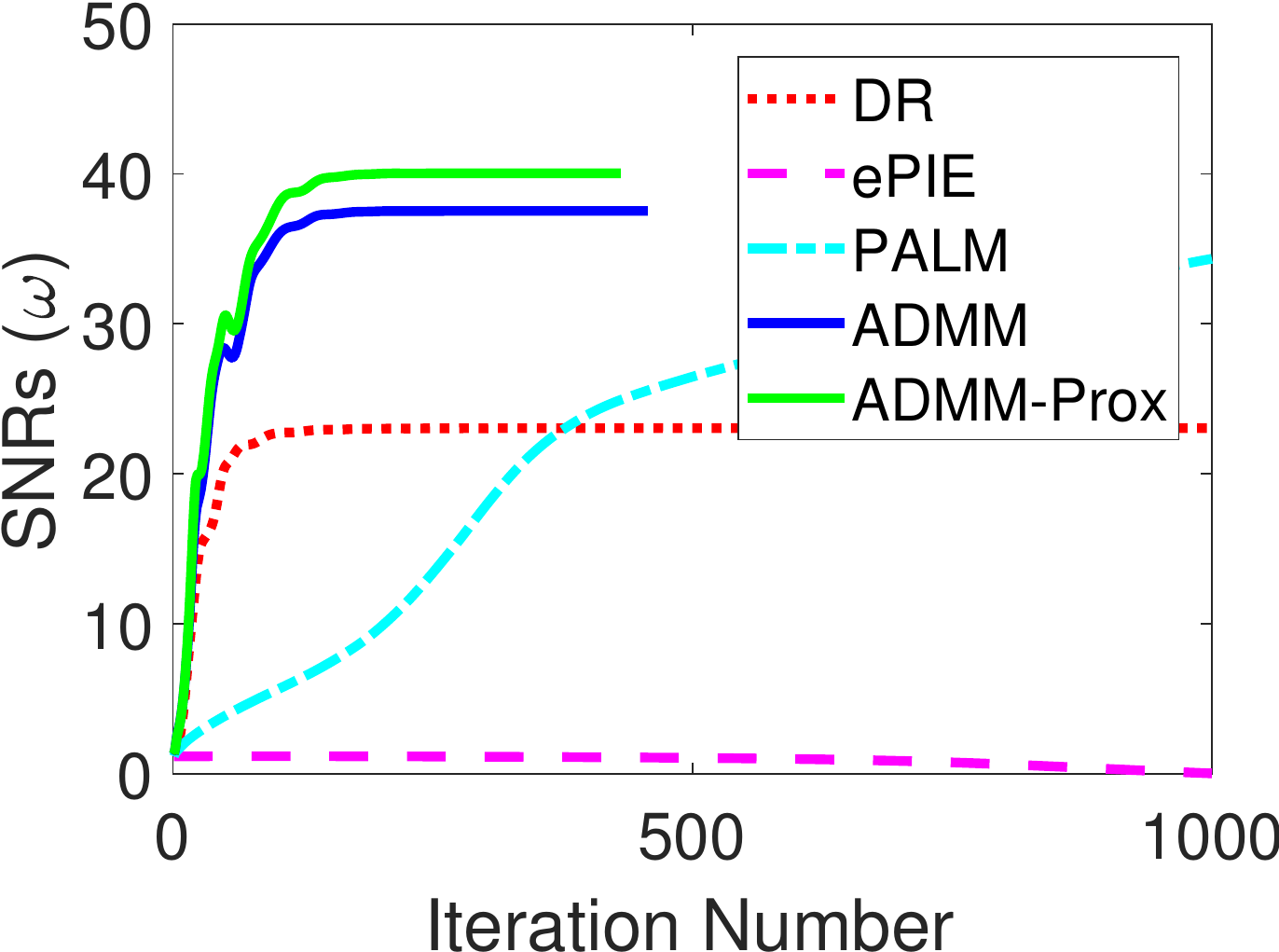}}
\subfigure[]{\includegraphics[width=.24\textwidth]{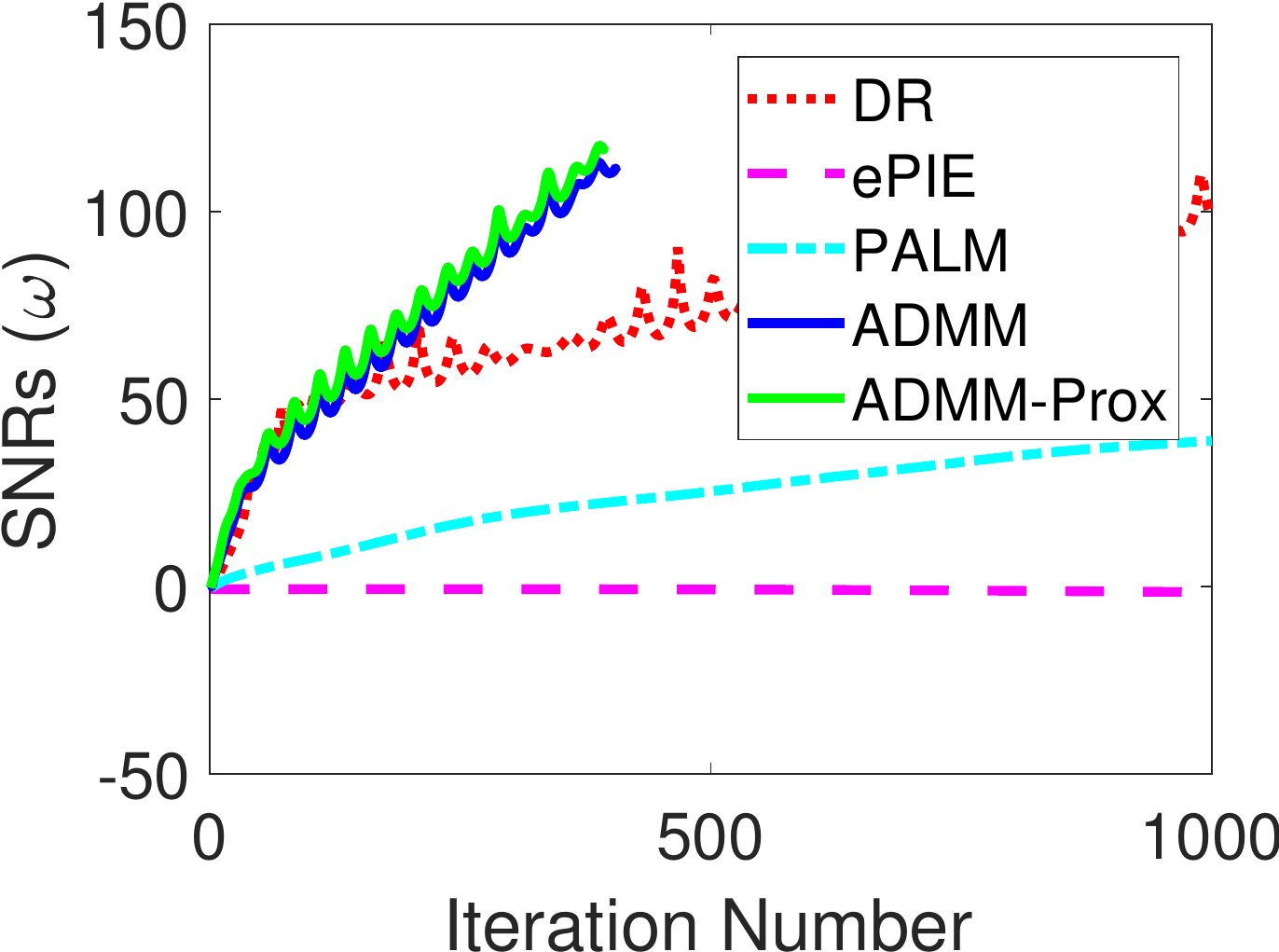}}
\subfigure[]{\includegraphics[width=.24\textwidth]{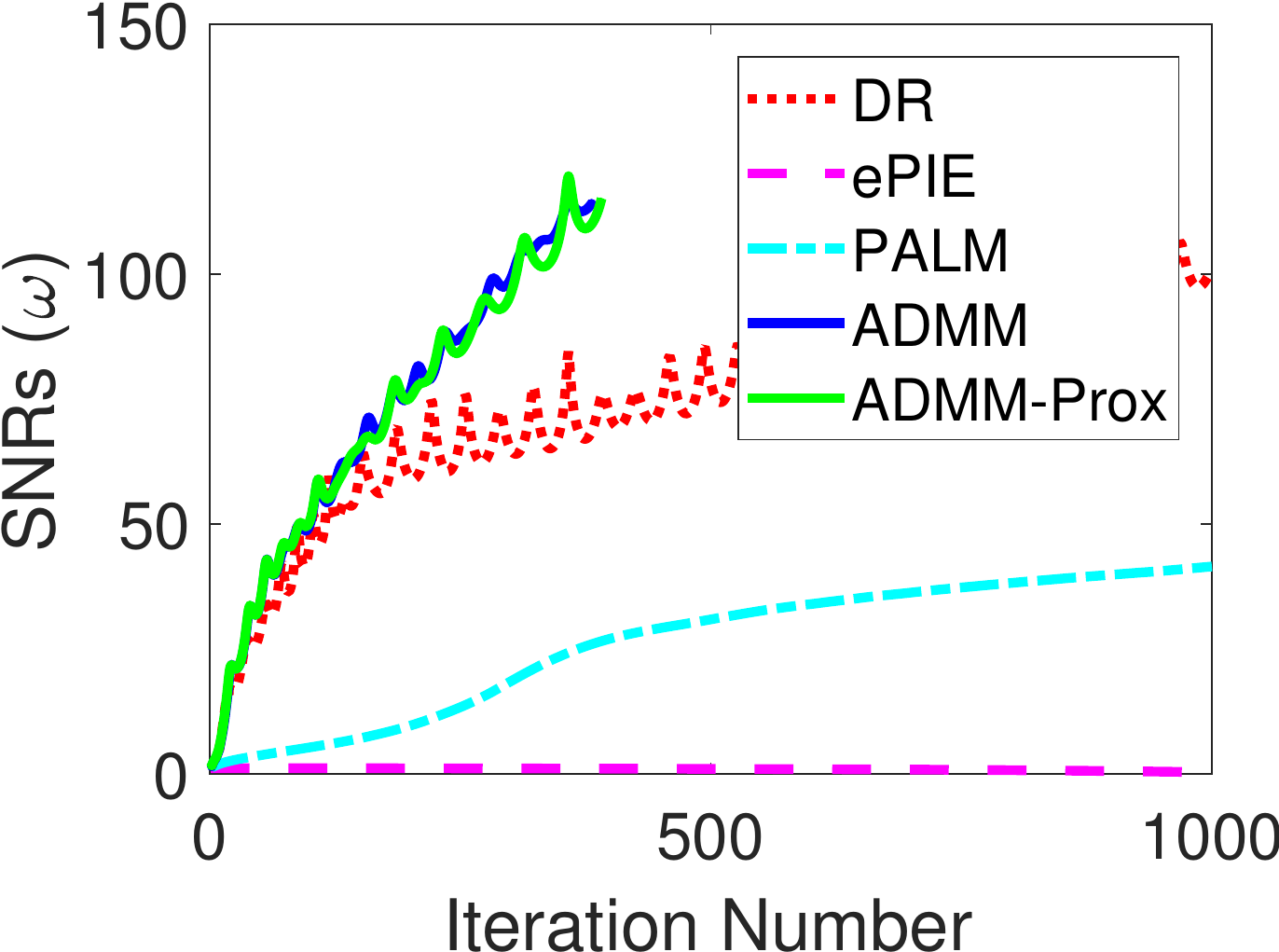}}
\end{center}
\vskip -.1in
\caption{ Convergence histories of R-factor, SNRs of  images and probes from up to down, respectively by ePIE \cite{maiden2009improved}, DR \cite{thibault2009probe},  PALM \cite{hesse2015proximal} and ADMM (Algorithm \ref{algADMM}), with $D=24$ and scanning on  square-lattice  in 1st and 2nd columns and random-lattice in 3rd and 4th columns. All compared algorithms stop if  $\mathrm{R-factor}^k\leq 1.0\times 10^{-6}$ or iteration number reaches $1000.$  Results for the complex-valued image (Figure \ref{fig12-0} (a)) are put in the 1st, 3rd columns, and for the real-valued image (Figure \ref{fig12-0}(b)) in other columns.}
\label{fig12}
\end{figure}

\vskip .1in
\section{Conclusions}\label{sec6}
In this paper we propose different nonlinear optimization models with and without prior information of the probe for the blind phase retrieval ptychography problem. Efficient generalized ADMMs are designed to  solve  the proposed models, and numerous numerical experiments demonstrate the superior performance  of the proposed algorithms in both speed and reconstruction quality compared with the state-of-the-art algorithms. Especially, the performance on GPU reported in this paper have motivated the development of a high performance \texttt{C++} multi-GPU implementation of the ADMM algorithm proposed in this paper. The multi-GPU implementation of ADMM  was reported \cite{enfedaque2018gpu} to process more than 3 million measured samples per second on a single GPU, and it has already been installed and being used
 on the microscopes of the Advanced Light Source (ALS), at Lawrence Berkeley National Laboratory for real-time analysis \cite{daurer2017nanosurveyor} and partial coherence analysis \cite{chang2018partially}.




\section*{Acknowledgement}
We would like to thank the two reviewers, and the associate editor for their valuable comments,  which help to improve the paper greatly. We would also like to thank Professor Wotao Yin in University of California, Los Angeles, for bringing to our attention some important references and further inspiring discussion.

This work of the first author  was partially supported by National Natural Science Foundation of China (Nos.11871372, 11501413), Natural Science
Foundation of Tianjin (No.18JCYBJC16600), 2017-Outstanding Young Innovation Team Cultivation Program (No.043-135202TD1703) and Innovation Project (No.043-135202XC1605) of Tianjin Normal University, Tianjin Young Backbone of Innovative Personnel Training Program and Program for Innovative Research Team in Universities of Tianjin (No.TD13-5078). This work was partially funded by the Center for Applied Mathematics for Energy Research Applications, a joint ASCR-BES funded project within the Office of Science, US Department of Energy, under contract number DOE-DE-AC03-76SF00098.

\bibliographystyle{siam}
\bibliography{rD}

\appendix
\section{PALM or BCD for Model I in \eqref{eqModel}}\label{apdx-0}


 Iterative scheme for PALM or BCD in the $(k+1)^{\text{th}}$ iteration is given below:
 \[
 \left\{
 \begin{split}
 &\omega^{k+1}=\arg\min_\omega \mathbb I_{\mathscr X_1}(\omega)+\tfrac{1}{2\tau^k_1}\|\omega-(\omega^k-\tau^k_1\nabla_\omega \mathcal G(\mathcal A(\omega^k,u^k)))\|^2\\
  &u^{k+1}=\arg\min_u \mathbb I_{\mathscr X_2}(u)+\tfrac{1}{2\tau^k_2}\|u-(u^k-\tau^k_2\nabla_u \mathcal G(\mathcal A(\omega^{k+1},u^k)))\|^2,
 \end{split}
 \right.
 \]
with stepsizes $\tau^k_1, \tau^k_2$,
 where
 \[
 \begin{split}
 &\nabla_\omega\mathcal G(\mathcal A(\omega,u))=\sum_j (\mathcal S_j u^*)\circ\mathcal F^*\big(\nabla_{z_j} \mathcal G(\mathcal A(\omega,u))\big),\\
 &\nabla_u\mathcal G(\mathcal A(\omega,u))=\sum_j \mathcal S^T_j\Big( \omega^*\circ\mathcal F^*\big(\nabla_{z_j} \mathcal G(\mathcal A(\omega,u))\big)\Big),
 \end{split}
 \]
 and $\nabla_z \mathcal G(z)$ is shown in \eqref{eqGrad-P}.

\section{Proof of Lemma \ref{lem1}}\label{apdx1-2}
We need the following estimate:
\begin{equation}\label{eqConsine}
\|v\|^2-\|w\|^2=\|v-w\|^2+2\Re(\langle w, v-w  \rangle)~\forall v, w\in \mathbb C^m.
\end{equation}
Denoting
{
$\mathcal Q_1(\omega,u,z):=\mathcal Q(\omega,u,z)+\mathbb I_{\mathscr X_1}(\omega)+\tfrac{\alpha_1}{2\beta}\|\omega-\omega_0\|_{M_1}^2$ and $\mathcal Q_2(\omega,u,z):=\mathcal Q(\omega,u,z)+\mathbb I_{\mathscr X_2}(u)+\tfrac{\alpha_2}{2\beta}\|u-u_0\|^2_{M_2}$} with $\mathcal Q(\omega,u,z):=\tfrac12\|\mathcal A(\omega,u)-z\|^2,$ and two positive definite matrices $M_1$ and $M_2$,
a basic estimate can be given below:

Letting $\omega^\star:=\arg\min_{\omega}\mathcal Q_1(\omega,u,z)$ and $u^\star:=\arg\min_{u}\mathcal Q_2(\omega,u,z)$,
\begin{align}
&\mathcal Q(\omega_0,u,z)-\mathcal Q(\omega^\star,u,z){\geq} \tfrac{1}{2}\|\mathcal A(\omega_0-\omega^\star, u)\|^2{+\tfrac{\alpha_1}{\beta}\|\omega^\star-\omega_0\|^2_{M_1},}\label{eq13-1}\\
&\mathcal Q(\omega,u_0,z)-\mathcal Q(\omega,u^\star,z){\geq} \tfrac{1}{2}\|\mathcal A(\omega, u_0-u^\star)\|^2{+\tfrac{\alpha_2}{\beta}\|u^\star-u_0\|^2_{M_2},}\label{eq13-2}
\end{align}
{for all $\omega_0\in \mathscr X_1$ and $u_0\in\mathscr X_2,$}
which can be easily derived by calculating the G$\mathrm{\hat{a}}$teaux derivative, \eqref{eqConsine},  $\mathcal A(\omega,u)$ being bilinear {and two constraint sets $\mathscr X_1, \mathscr X_2$ being convex}.

\begin{proof}
First, we consider the subproblem in Step 1 of \eqref{eqADMM} as
\begin{equation}\label{eqED1}
 \begin{split}
 &\Upsilon_\beta(X^k)-  \Upsilon_\beta(\omega^{k+1},u^{k},z^{k},\Lambda^{k})=\stackrel{\eqref{eq13-1}}{\geq}\tfrac{\beta}{2}\|\mathcal A(E_\omega^{k+1}, u^k)\|^2{+{\alpha_1}\|E^{k+1}_{\omega}\|^2_{M^k_1}}\\
&= \tfrac{\beta}{2}\sum\nolimits_j \|E_\omega^{k+1}\circ \mathcal S_j u^k\|^2{+{\alpha_1}\|E^{k+1}_{\omega}\|^2_{M^k_1}}\\
&= \tfrac{\beta}{2} \big\langle |E_\omega^{k+1}|^2, \sum\nolimits_j|\mathcal S_j u^k|^2\big\rangle{+{\alpha_1}\|E^{k+1}_{\omega}\|^2_{M^k_1}}\geq  \tfrac{\beta}{2} I_u^k\|E_\omega^{k+1}\|^2.
 \end{split}
\end{equation}
For the subproblem in Step 2 of \eqref{eqADMM}, similarly,
we have
\begin{equation}\label{eqED2}
\Upsilon_\beta(\omega^{k+1},u^k,z^k,\Lambda^k)\!-\!\Upsilon_\beta(\omega^{k+1},u^{k+1},z^{k},\Lambda^{k})\geq\tfrac{\beta}{2}I_\omega^k \|E_u^{k+1}\|^2.
\end{equation}
%
For the subproblem in Step 3 of \eqref{eqADMM}, by \eqref{eqConsine} and \eqref{eqOptC-3}, we have
\[
\begin{split}
&\Upsilon_\beta(\omega^{k+1},u^{k+1},z^k,\Lambda^k)-  \Upsilon_\beta(\omega^{k+1},u^{k+1},z^{k+1},\Lambda^{k})
\\
&=\mathcal G(z^{k})-\mathcal G(z^{k+1})+\tfrac{\beta}{2}\|E_z^{k+1}\|^2+\Re(\langle\nabla \mathcal G(z^{k+1}),  E_z^{k+1} \rangle).\\
\end{split}
\]
Then, using the Lemma \ref{assump1}, and Cauchy's inequality as
$
\Re(\langle z_1,z_2\rangle)\geq - \tfrac{L}{2} \|z_1\|^2-\tfrac{1}{2L}\|z_2\|^2$~$\forall$ $z_1,$ $z_2$ $\in\mathbb C^m,
$
one readily has
\begin{equation}\label{eqED3}
\begin{split}
&\Upsilon_\beta(\omega^{k+1},u^{k+1},z^k,\Lambda^k)-  \Upsilon_\beta(\omega^{k+1},u^{k+1},z^{k+1},\Lambda^{k})\\
\geq &\tfrac{\beta-2L}{2}\|E_z^{k+1}\|^2-\tfrac{1}{2L}\|\nabla \mathcal G(z^{k+1})-\nabla \mathcal G(z^{k})\|^2
\geq 
\tfrac{\beta-3L}{2}\|E_z^{k+1}\|^2.
\end{split}
 \end{equation}
Then we need to estimate the following relation as
$
\Upsilon_\beta(\omega^{k+1},u^{k+1},z^{k+1},\Lambda^k)\!- \! \Upsilon_\beta(X^{k+1})\!\!$
$=-\tfrac{1}{\beta}\|E_\Lambda^{k+1}\|^2,
$
that yields the following inequality by \eqref{eqOptC-3-2} and Lemma \ref{assump1}
\begin{equation}\label{eqED4}
\begin{split}
&\!\Upsilon_\beta(\omega^{k+1},u^{k+1},z^{k+1},\Lambda^k)\!-\!  \Upsilon_\beta(\omega^{k+1},u^{k+1},z^{k+1},\Lambda^{k+1})\!\geq\!
\!-\tfrac{L^2}{\beta}\|E_z^{k+1}\|^2.
\end{split}
\end{equation}
Finally, summing up \eqref{eqED1} and \eqref{eqED4} gives the desired lemma.
\end{proof}

\section{Proof of Lemma \ref{lem2}}\label{apdx-2}

\begin{proof}
One can readily prove the nonincrease
 of the augmented Lagrangian with sufficiently large $\beta>\tfrac{3+\sqrt{17}}{2}L$ by  Lemma \ref{lem1}. 
 On the other hand,
\[
\begin{split}
&\quad\Upsilon(X^{k+1})\\
&\stackrel{\eqref{eqOptC-3-2}}{=}\mathcal G(z^{k+1})-\Re(\langle z^{k+1}-\mathcal A(\omega^{k+1}, u^{k+1}), \nabla \mathcal G(z^{k+1}) \rangle)
+\tfrac{\beta}{2}\|z^{k+1}-\mathcal A(\omega^{k+1}, u^{k+1})\|^2\\
&\stackrel{\eqref{eqLipDescent}}{\geq} \tfrac{\beta-L}{2}\|z^{k+1}-\mathcal A(\omega^{k+1}, u^{k+1})\|^2+\mathcal G(\mathcal A(\omega^{k+1},u^{k+1})).
\end{split}
\]
Therefore, due to the boundedness of $\omega^k$ and $u^k$, the sequence  $\{z^k\}$ is  also bounded.

By \eqref{eqOptC-3-2}, the boundedness of $\{z^k\}$, and $C^\infty$ smooth of $\mathcal G$, $\{\Lambda^k\}$ is also bounded. 
%
Therefore $\{X^k\}$ is bounded. As a result,  the augmented Lagrangian is bounded as well,
which completes this lemma.
\end{proof}

\section{Proof of Lemma \ref{lem3}}\label{apdx1-4}

\begin{proof}
We first estimate the upper bound of the partial  derivative \emph{w.r.t.} to $\omega$. {By \eqref{eqOptC-1}, there exists a variable $\bar \omega^{k+1}$ defined as
\begin{equation}\label{eqOptC-1-1}
\bar \omega^{k+1}:=-{\omega}^{k+1}\circ{\sum\nolimits_j\left|\mathcal S_j u^k\right|^2}+{\sum\nolimits_j(\mathcal S_j u^k)^*\circ\mathcal F^{-1}(z_j^k+\tfrac{1}{\beta}\Lambda_j^k)}-\tfrac{\alpha_1}{\beta}\mathrm{diag}(M_1^k) \circ E_\omega^{k+1},
\end{equation}
such that
$\bar \omega^{k+1}\in\tfrac{1}{\beta}\partial \mathbb I_{\mathscr X_1}(\omega^{k+1}).$
Immediately,
\[
\bar \omega^{k+1}+ {\omega}^{k+1}\circ {\sum\nolimits_j\left|\mathcal S_j u^{k+1}\right|^2}\!-\!{\sum\nolimits_j(\mathcal S_j u^{k+1})^*\circ\mathcal F^{-1}\hat z_j^{k+1}}\in \tfrac{1}{\beta}\partial_\omega \Upsilon_\beta (X^{k+1}).
\]
Hence we need to estimate the bound of the left hand side term of the above equation.} Readily we have
\begin{equation}\label{eq1}
\begin{split}
&{\mathrm{dist}(0,\partial_\omega \Upsilon_\beta (X^{k+1}))\!\leq\!\beta\Big\|\bar \omega^{k+1}+{\omega}^{k+1}\circ {\sum\nolimits_j\left|\mathcal S_j u^{k+1}\right|^2}\!-\!{\sum\nolimits_j(\mathcal S_j u^{k+1})^*\circ\mathcal F^{-1}\hat z_j^{k+1}}\Big\|}\\
&\stackrel{\eqref{eqOptC-1-1}}{\leq}{{\alpha_1}\|\mathrm{diag}(M_1^k) \circ E_\omega^{k+1}\|}+\beta\Big\|{\omega}^{k+1}\circ {\sum\nolimits_j\left|\mathcal S_j u^{k+1}\right|^2}\!-\!{\omega}^{k+1}\circ {\sum\nolimits_j\left|\mathcal S_j u^k\right|^2}\Big\|\!\\
&\qquad~~~+\beta\Big\|{\sum\nolimits_j(\mathcal S_j u^{k+1})^*\circ\mathcal F^{-1}\hat z_j^{k+1}}-{\sum\nolimits_j(\mathcal S_j u^k)^*\circ\mathcal F^{-1}\hat z_j^{k}}\Big\|\\
&\!\leq\!{{\alpha_1}\|\mathrm{diag}(M_1^k)\|_\infty \| E_\omega^{k+1}\|}+\beta\|{\omega}^{k+1}\|_\infty\big\|\sum\nolimits_j\mathcal S_j (|u^{k+1}|^2\!-\! |u^k|^2)\big\|\!\\
\!&\!+\!\beta\Big\|{\sum\nolimits_j(\mathcal S_j E_u^{k+1})^*\circ\mathcal F^{-1}\hat z_j^{k}}\Big\|\!+\beta\Big\|{\sum\nolimits_j(\mathcal S_j u^{k+1})^*\circ\mathcal F^{-1}(z_j^{k+1}-z_j^k+\tfrac{1}{\beta}\Lambda_j^{k+1}-\tfrac{1}{\beta}\Lambda_j^k)}\Big\|,\!\!
\end{split}
\end{equation}
where the first two terms of last inequality are derived by \eqref{eq3}. 
Then we will estimate  three terms in the last inequality of  \eqref{eq1} one by one.
For the first term of the last inequality for \eqref{eq1}, we have
\begin{equation}
\label{eq5}
\begin{split}
&\!\!\big\|\sum\nolimits_j\mathcal S_j (|u^{k+1}|^2- |u^k|^2)\big\|\leq  \sqrt{\rho(\mathcal S)}\big\||u^{k+1}|^2- |u^k|^2 \big\| \\
&\!\!\!
\!\stackrel{\eqref{eq3}}{\leq}\! \sqrt{\rho(\mathcal S)}\big(\|u^{k+1}\|_\infty\!+\!\|u^k\|_{\infty}\big) \|E_u^{k+1}\|,\!
\end{split}
\end{equation}
where the first inequality is derived by
$
\big\|\sum\nolimits_j \mathcal S_j u\big\|^2=\sum_{0\leq j_1,j_2\leq J-1}\big\langle  \mathcal S_{j_1} u, \mathcal S_{j_2} u\big\rangle
=\big\langle \mathcal S u,  u\big\rangle\leq \rho(S)\|u\|^2,
$
with a positive semi-positive matrix $S:=\sum_{0\leq j_1,j_2\leq J-1} \mathcal S_{j_2}^T\mathcal S_{j_1}\in \mathbb R^{n\times n}$ and its spectral radius $\rho(S).$
For the second term, we have
\begin{equation}\label{eq6}
\begin{split}
&\big\|{\sum\nolimits_j(\mathcal S_j E_u^{k+1})^*\circ\mathcal F^{-1}\hat z_j^{k}}\big\|\leq\sum\nolimits_j\big\|{(\mathcal S_j E_u^{k+1})^*\circ\mathcal F^{-1}\hat z_j^{k}}\big\|\\
&\stackrel{\eqref{eq3}}{\leq}\sum\nolimits_j \|\mathcal F^{-1}\hat z_j^{k}\|_\infty\big\|\mathcal S_j E_u^{k+1}\big\|\leq
\max_j \|\mathcal F^{-1}\hat z_j^{k}\|_\infty \sqrt{J\rho(\tilde {\mathcal S})  }\|E_u^{k+1}\|,
\end{split}
\end{equation}
since
$\sum\nolimits_j \big\|\mathcal S_j u\big\|\leq \sqrt{J} \sqrt{\sum_j \|\mathcal S_j u\|^2} \leq \sqrt{J\rho(\tilde {\mathcal S})  }\|u\|,
$
with  $\tilde {\mathcal S}:=\sum_j \mathcal S_j^T \mathcal S_j\in \mathbb R^{n\times n}.$
For the third term, we have
\begin{equation}\label{eq2}
\begin{split}
&\big\|{\sum\nolimits_j(\mathcal S_j u^{k+1})^*\circ\mathcal F^{-1}(z_j^{k+1}-z_j^k+\tfrac{1}{\beta}\Lambda_j^{k+1}-\tfrac{1}{\beta}\Lambda_j^k)}\big\|
\\
&\!\leq \!\|u^{k+1}\|_{\infty} \sum\nolimits_j\big(\|z_j^{k+1}\!-\!z^k_j\|+\tfrac{1}{\beta}\|\Lambda_j^{k+1}\!-\!\Lambda_j^k\|\big)\leq \sqrt{J} \|u^{k+1}\|_{\infty}  \big(\|E_z^{k+1}\|\!+\!\tfrac{1}{\beta}\|E_\Lambda^{k+1}\|\big),\!\!\!\!
\end{split}
\end{equation}
where the last inequality is derived by
${\sum_j x_j}\leq \sqrt{J\sum_j x_j^2}~\forall x_j\in \mathbb R$.
By inserting \eqref{eq5} and \eqref{eq2} into \eqref{eq1}, we have:
\begin{equation}\label{eq9}
\begin{split}
&\!\!\!\!\|\nabla_{\omega} \Upsilon_\beta(X^{k+1})\|
\!\leq\! {{\alpha_1}\|\mathrm{diag}(M_1^k)\|_\infty \| E_\omega^{k+1}\|}+\beta \Big(\|\omega^{k+1}\|_\infty\sqrt{\rho(\mathcal S)} (\|u^{k+1}\|_\infty\!+\!\|u^k\|_{\infty})\!\\
&\qquad\qquad+ \max_j \|\mathcal F^{-1}\hat z_j^{k}\|_\infty \sqrt{J\rho(\tilde {\mathcal S})  }\Big)\! \|E_u^{k+1}\|
+\beta \sqrt{J} \|u^{k+1}\|_{\infty}  \big(\|E_z^{k+1}\|+\tfrac{1}{\beta}\|E_\Lambda^{k+1}\|\big).
\end{split}
\end{equation}

Then consider the derivative \emph{w.r.t.} $u$.
  {By \eqref{eqOptC-2}, there exists a variable $\bar u^{k+1}$ defined as
  \begin{equation}\label{eqOptC-2-1}
  \!\bar u^{k+1}\!:=-u^{k+1}\circ\!\sum\nolimits_j\mathcal S_j^T|\omega^{k+1}|^2\!+\sum\nolimits_j\mathcal S_j^T((\omega^{k+1})^*\circ\mathcal F^{-1}( z_j^k\!+\tfrac{1}{\beta}\Lambda_j^k))-\tfrac{\alpha_2}{\beta}\mathrm{diag}(M_2^k)\circ E_u^{k+1},\!\!
  \end{equation}
such that
$\bar u^{k+1}\in\partial \mathbb I_{\mathscr X_2}(u^{k+1}).$
Immediately,
\[
\bar u^{k+1}+ u^{k+1}\circ\sum\nolimits_j\mathcal S_j^T|\omega^{k+1}|^2\!-\!\sum\nolimits_j\mathcal S_j^T((\omega^{k+1})^*\circ\mathcal F^{-1}\hat z_j^{k+1})\in \tfrac{1}{\beta}\partial_u \Upsilon_\beta (X^{k+1}).
\]
}
Similarly to \eqref{eq1}, we have
\begin{equation}\label{eq10}
\begin{split}
&{\mathrm{dist}(0,\partial_u \Upsilon_\beta (X^{k+1}))}\leq\!\beta\Big\| {\bar u^{k+1}}+u^{k+1}\circ\sum\nolimits_j\mathcal S_j^T|\omega^{k+1}|^2\!-\!\sum\nolimits_j\mathcal S_j^T((\omega^{k+1})^*\circ\mathcal F^{-1}\hat z_j^{k+1})\Big\|\\
\!&\!\stackrel{\eqref{eqOptC-2-1}}{=}\!\!\beta
\Big\|\!{-\tfrac{\alpha_2}{\beta}\mathrm{diag}(M_2^k)\circ E_u^{k+1}}\!+\!\sum\nolimits_j\mathcal S_j^T((\omega^{k+1})^*\circ\mathcal F^{-1}\hat z_j^{k+1})\!-\!\sum\nolimits_j\mathcal S_j^T((\omega^{k+1})^*\circ\mathcal F^{-1}\hat z_j^{k})\Big\|\!\!\!\!\\
&\leq{{\alpha_2}\|\mathrm{diag}(M_2^k)\|_\infty\| E_u^{k+1}\|}+\beta
\sqrt{J}\|\omega^{k+1}\|_\infty\big(\| E_z^{k+1}\|+\tfrac{1}{\beta}\|E_\Lambda^{k+1}\|\big).
\end{split}
\end{equation}

For the derivatives \emph{w.r.t.} $z$ and $\Lambda$, we have
\begin{equation}\label{eq7}
\begin{split}
&\|\nabla_{z} \Upsilon_\beta(X^{k+1})\|\!=\!\|\nabla \mathcal G(z^{k+1})\!+\!\Lambda^{k+1}\!+\!\beta(z^{k+1}\!-\!\mathcal A(\omega^{k+1},u^{k+1}))\|\!\stackrel{\eqref{eqOptC-3}}{=}\!\|E_\Lambda^{k+1}\|;\!
\end{split}
\end{equation}
\begin{equation}\label{eq8}
\begin{split}
&\|\nabla_{\Lambda} \Upsilon_\beta(X^{k+1})\|
{=}\|z^{k+1}-\mathcal A(\omega^{k+1},u^{k+1})\|=\tfrac{1}{\beta}\|E_\Lambda^{k+1}\|.
\end{split}
\end{equation}

Set $C_3$ as
$C_3={2}\sup_{k}\max\Big\{{{\alpha_2}\|\mathrm{diag}(M_2^k)\|_\infty}+  \beta\Big(\|\omega^{k+1}\|_\infty\sqrt{\rho(\mathcal S)}(\|u^{k+1}\|_\infty+\|u^k\|_{\infty})$
$+ \max_j \|\mathcal F^{-1}\hat z_j^{k}\|_\infty \sqrt{J\rho(\tilde {\mathcal S})  }\Big), {{\alpha_1}\|\mathrm{diag}(M_1^k)\|_\infty},
\beta\sqrt{J}(\|u^{k+1}\|_\infty+\|\omega^{k+1}\|_\infty),~~$
$
\sqrt{J}(\|u^{k+1}\|_\infty+\|\omega^{k+1}\|_\infty)+1+\tfrac{1}{\beta}
  \Big\}$,
which is bounded by Lemma \ref{lem2}.
Since
\[
\begin{split}
&\mathrm{dist}(0,\partial \Upsilon_\beta (X^{k+1}))\leq \mathrm{dist}(0,\partial_\omega \Upsilon_\beta (X^{k+1}))\!+\mathrm{dist}(0,\partial_u \Upsilon_\beta (X^{k+1}))\!\\
&\hskip 3.5cm
+\|\nabla_z \Upsilon_\beta (X^{k+1})\|\!+\|\nabla_\Lambda \Upsilon_\beta (X^{k+1})\|,
\end{split}
\]
by summing up  \eqref{eq9}-\eqref{eq8},
we finally conclude this lemma with the given $C_3$ .
\end{proof}

\section{Proof of Theorem  \ref{thm2}}\label{apdx1-5}

\begin{proof}
By the Lemma \ref{lem1} and Assumption \ref{assump2}, we have
\begin{equation}\label{eq15}
 \begin{split}
 &\Upsilon_\beta(X^k)\!-\! \Upsilon_\beta(X^{k+1})\geq  \!\tfrac{\beta}{2} \min\{C_1,C_2\}(\|E_\omega^{k+1}\|^2
\!+\!\|E_u^{k+1}\|^2)\!+\!(\tfrac{\beta\!-\!3L}{2}\!-\!\tfrac{L^2}{\beta})\|E_z^{k+1}\|^2,\\
& \geq  \tfrac{\beta}{2} \min\{C_1,C_2\}(\|E_\omega^{k+1}\|^2
+\|E_u^{k+1}\|^2)+\tfrac{1}{\beta}\|E_\Lambda^{k+1}\|^2
+(\tfrac{\beta-3L}{2}-\tfrac{2L^2}{\beta})\|E_z^{k+1}\|^2,\\
 \end{split}
\end{equation}
where the last inequality is derived by
$
\|E_\Lambda^{k+1}\|\stackrel{\eqref{eqOptC-3-2}}{=}\|\nabla \mathcal G(z^{k+1})-\nabla \mathcal G(z^k)\|\leq L\|E_z^{k+1}\|.\!\!
$
Therefore,
with $\beta>4L,$
we have
\begin{equation}\label{eq11}
\Upsilon_\beta(X^k)-  \Upsilon_\beta(X^{k+1})\geq C_4\|X^{k+1}-X^k\|^2,
\end{equation}
with  a positive constant defined as
$
C_4:=\min\Big\{\tfrac{\beta}{2} C_1, \tfrac{\beta}{2} C_2, \tfrac{1}{\beta}, \tfrac{\beta-3L}{2}-\tfrac{2L^2}{\beta} \Big\}>0.
$

According to Proposition \ref{thm0}, there exists at least a stationary point for the proposed Model I in \eqref{algADMM} and the corresponding augmented Lagrangian.
{Readily  we have   \begin{equation}
\begin{split}
\Upsilon_\beta&(\omega,u,z,\Lambda):=\mathcal G(z)
+\tfrac{\beta}{2}\|z-\mathcal A(\omega, u)+\tfrac{\Lambda}{\beta}\|^2-\tfrac{1}{2\beta}\|\Lambda\|^2+\mathbb I_{\mathscr X_1}(\omega)+\mathbb I_{\mathscr X_2}(u).
\end{split}
\end{equation}
By separating the real and imaginary parts of the variables $\omega, u, z, \Lambda$ and the operator $\mathcal A$ as  \cite{chang2016Total}, one readily knows that $\mathcal G(z)
+\tfrac{\beta}{2}\|z-\mathcal A(\omega, u)+\tfrac{\Lambda}{\beta}\|^2-\tfrac{1}{2\beta}\|\Lambda\|^2$ is the semi-algebraic  function as well as the indictor functions $\mathbb I_{\mathscr X_1}(\omega)+\mathbb I_{\mathscr X_2}(u)$. {Therefore, $\Upsilon_\beta$ is semi-algebraic (Example 2 in \cite{attouch2010proximal}) such that it satisfies the KL property \cite{attouch2010proximal}.}} {For the case of pIPM, the summation of the real-analytic function $\mathcal G(z)$ and the semi-algebraic function $\tfrac{\beta}{2}\|z-\mathcal A(\omega, u)+\tfrac{\Lambda}{\beta}\|^2-\tfrac{1}{2\beta}\|\Lambda\|^2+\mathbb I_{\mathscr X_1}(\omega)+\mathbb I_{\mathscr X_2}(u)$ is  sub-analytic  \cite{xu2013block}, so the objective functional satisfies the KL
property.}
Furthermore,  by combining it with  Lemmas \ref{lem2} and \ref{lem3} and \eqref{eq11}, {we can  readily conclude this theorem following  \cite{attouch2010proximal,hong2016convergence,wang2015global,lou2016fast,mei2016cauchy,yang2017alter}}.
The reminder of the proof  is standard, and therefore we omit the details.
\end{proof}

\end{document}